\documentclass[11pt]{article}
\usepackage{bbm}
\usepackage{amsmath,amsthm,amssymb}
\usepackage{graphicx}
\usepackage{tikz}
\usepackage[utf8]{inputenc}
\usepackage{pgfplots}
\usepackage{hyperref}
\usepackage{setspace}
\usepackage{pict2e,color}
\usepackage{multirow}
\usepackage{booktabs}
\usepackage{amsfonts}
\usepackage{verbatim} 
\usepackage{sectsty}
\usepackage{url}
\usepackage{empheq}
\usepackage[normalem]{ulem}
\usepackage[titletoc,title]{appendix}
\usepackage[flushleft]{threeparttable}
\usepackage{soul} 
\usepackage{natbib}
\usepackage{algorithm, algpseudocode}
\usepackage{microtype}
\usepackage{subcaption}
\usepackage{pbox}
\usepackage{etoolbox}
\preto\subequations{\ifhmode\unskip\fi}
\usepackage{booktabs}
\usepackage{multirow}
\theoremstyle{definition}

\newtheorem{theorem}{Theorem}
\newtheorem{lemma}{Lemma}

\newtheorem{remark}{Remark}

\newtheorem{assumption}{Assumption}

\allowdisplaybreaks

\title{Distributionally Robust Optimization with Multimodal Decision-Dependent Ambiguity Sets}
\author{Xian Yu\thanks{Department of Integrated Systems Engineering, The Ohio State University, Email: {\tt yu.3610@osu.edu};}~~~Beste Basciftci\thanks{Department of Business Analytics, Tippie College of Business, University of Iowa, Email: {\tt beste-basciftci@uiowa.edu}}}
\date{}

\begin{document}

\maketitle

\begin{abstract}
We consider a two-stage distributionally robust optimization (DRO) model with multimodal uncertainty, where both the mode probabilities and uncertainty distributions could be affected by the first-stage decisions. To address this setting, we propose a generic framework by introducing a $\phi$-divergence based ambiguity set to characterize the decision-dependent mode probabilities and further consider both moment-based and Wasserstein distance-based ambiguity sets to characterize the uncertainty distribution under each mode. We identify two special $\phi$-divergence examples (variation distance and $\chi^2$-distance) and provide specific forms of decision dependence relationships under which we can derive tractable reformulations. Furthermore, we investigate the benefits of considering multimodality in a DRO model compared to a single-modal counterpart through an analytical analysis. 
{\color{black}Additionally, we develop a separation-based decomposition algorithm to solve the resulting multimodal decision-dependent DRO models with finite convergence and optimality guarantee under certain settings.} 
We provide {\color{black}a detailed computational study over two example problem settings, the facility location problem and shipment planning problem with pricing,} to illustrate our results, which demonstrate that omission of multimodality {\color{black}or} decision-dependent uncertainties within DRO frameworks result in inadequately performing solutions with worse in-sample and out-of-sample performances under various settings. 
{\color{black}We further demonstrate the speed-ups obtained by the solution algorithm against the off-the-shelf solver over various instances.}
\end{abstract}

~

\section{Introduction}

In the field of optimization under uncertainty, decision-making problems mainly consider a single well-defined distribution to characterize the underlying uncertainties. On the other hand, in various problem contexts, this conventional unimodality consideration simplifies the representation of the underlying data, that may have the tendency of having several spatially separated regions with increased probabilities, which
can be associated with the potential different modes of the distributions that can be described by multimodality. 
Another level of complexity within these problems arises when there is an interplay between the decisions and the underling uncertainties, which can impact the likelihood of each mode of the distribution by making one mode to be more likely than the other one, and the characteristics of each corresponding distribution itself.
Furthermore, the assumption of the full distribution information within these decision-making problems can lead to solutions that are not robust in case the distribution is misspecified or there is lack of data to estimate this distribution accurately. 
To address these key challenges, in this paper, we propose the requisite optimization under uncertainty framework that jointly takes into account multimodal characteristics of the underlying data, decision-dependent uncertainties, and distributional ambiguities. 

Multimodalities in data sets have been observed in many different problem contexts, 
which should be effectively represented and integrated into the subsequent decision-making processes. 
The first example setting is the newsvendor problem, which is a fundamental problem in optimizing supply chain operations \citep{Zipkin2000}. Here, the goal is to determine the ordering quantities and the resulting inventory levels of products under uncertain demand while considering the relevant costs such as ordering, holding, and backordering. In this context, multimodality of demand can be observed and becomes suitable to present the demand distribution under various circumstances \citep{Hanasusanto2015_MultimodalDRO}, which incorporates the cases when (i) a new product enters the market and it is challenging to predict its demand, (ii) a customer that constitutes the majority of the sales orders a product with irregular bulk amounts, and (iii) a new competitor enters the market and it is challenging to predict the reaction of the customers to this competition. In addition to these cases, demand to a product can fluctuate based on different trends in the market, which is shown in fashion industry, where the unimodality assumption is not justifiable and leading to unfavourable ordering policies \citep{Vaagen2008, Riley2018_FashionBimodal}. 
Similar concerns can arise in representing customer demand over the facility location problem in various service industries, which aims to determine
where to locate facilities to serve a given set of customers under uncertain demand \citep{Basciftci2023_BookChapter}. For instance, for determining where to locate charging stations of electric vehicles, different adoption rates of customers to the electric vehicles can pose challenges in predicting customer demand where a multimodal description of demand becomes more suitable \citep{Mak2013_EVChargingLocation, Shehadeh2021_BimodalDROFLP}. 
In all of these potential cases, unimodality assumption over customer demand
can oversimplify the representation of the underlying uncertainty and fails to capture the true problem setting by providing suboptimal decisions. 
Generalizing these observations demonstrates the importance of incorporating multimodal distributions arising in various problem contexts.

Moreover, in modelling multimodal uncertainties, a challenge for decision-makers is to take into account the dependency between decisions, mode probabilities, and the distribution of each mode, since decisions can have a direct impact on the likelihood of certain modes and the corresponding distributions. An example setting is the facility location problem, where opening of a facility location can increase the likelihood of having higher customer demand in the customer locations nearby, along with an increase in the customer demand itself. Thus, the facility opening decisions can impact both the mode probabilities and the characteristics of the underlying demand distribution, where this analogy should be leveraged into the optimization under uncertainty frameworks to efficiently formulate these problems with such decision-dependencies.  

Distributionally robust optimization {\color{black}(DRO)} is an optimization under uncertainty framework that has become recently popular as a striking balance between stochastic programming and robust optimization, where the former approach assumes a full distribution information of the underlying uncertainty, and the latter approach provides a more conservative framework by disregarding distribution information and considering an uncertainty set to represent the unknown parameters. On the contrary, DRO considers a partial information regarding the underlying distribution by optimizing the worst-case solution while considering potential distributions that come from an ambiguity set. The literature in this area can be divided into categories depending on how the ambiguity sets are defined, which are mainly referred as moment-based and distance-based ambiguity sets.  
Moment-based ambiguity sets consider distributions with given moment information within a specified region of their empirical counterparts \citep{Delage2010, zhang2018ambiguous}. On the other hand, distance-based ambiguity sets consider distributions that are within a specified distance from a reference distribution.
Some of the most studied distance measures in this area are $\phi$-divergence {\color{black}\citep{ben2013robust, Jiang2016, Liu2022_JOGO}} and Wasserstein distance \citep{Gao2016, Esfahani2018}, which can provide computational tractability in reformulations. 

The majority of the literature on DRO considers ambiguity sets that focus on unimodal distributions by neglecting potential different modes of distributions. As a pioneer study in this area, \cite{Hanasusanto2015_MultimodalDRO} focus on distributionally robust multi-item newsvendor problem under multimodal demand uncertainty by considering a moment-based ambiguity set and presenting a decision rule based approximation for its solution. 
By following this approach, \cite{Zhao2020_DROMultimodal} provide a DRO model to optimize the operations of energy hub systems under the ambiguity of multimodal forecast errors of photovoltaic power.   
\cite{Shehadeh2020_DROBimodal} study distributionally robust outpatient colonoscopy scheduling problem under bimodal colonoscopy duration uncertainty with known mean and support information to characterize its ambiguity set. 
Additionally, \cite{Chen2020_DRO} introduce a DRO model called robust
stochastic optimization by integrating features from stochastic programming through scenario trees to associate distributions with random events. The authors propose various ambiguity sets including an event-wise ambiguity set that encapsulates modelling mixture of distributions, including the one proposed in \cite{Hanasusanto2015_MultimodalDRO}. 
By extending this ambiguity set to the facility location problem, \cite{Shehadeh2021_BimodalDROFLP} study a distributionally robust variant of this problem under bimodal customer demand. 
As a different line of research, in machine learning, group DRO approach is proposed to optimize the worst-case expected loss of a prediction model over a mixture of given distributions to characterize the set of distributions corresponding to the test data for accounting potential multimodalities \cite{Sagawa2020_GroupDRO}.  
Although these studies provide potential approaches to integrate multimodality within DRO, they focus on specific problem settings 
and do not provide a generic framework while omitting the potential decision-dependent uncertainties between decisions and the distributions of the modes and each mode itself. 
Furthermore, majority of these studies do not consider the distributional ambiguity revolving around the mode probabilities. 

The notion of decisions impacting the distribution of the underlying uncertainties is a rising area in various optimization under uncertainty frameworks including stochastic programming \citep{hellemo2018decision} and robust optimization \citep{Nohadani2018} with applications in various contexts including healthcare \citep{Nohadani2017}, energy systems \citep{Basciftci2019} and production planning \citep{Feng2021_RobustDecisionDependent}. 
Within DRO frameworks, 
\cite{luo2020distributionally} introduce decision-dependency to various ambiguity sets including moment-based and distance-based variants by mainly focusing on two-stage programs. 
\cite{yu2022multistage} extend decision-dependency to multi-stage distributionally robust problems under moment-based ambiguity sets and provide tractable reformulations under certain cases with a stochastic dual dynamic integer programming based method for its solution.  
As an application of decision-dependency to the facility location problem, \cite{basciftci2021distributionally} considers the impact of facility location decisions on uncertain customer demand by building a piecewise linear relationship between these decisions and the first and second moment information of demand within a DRO framework, and obtain its tractable reformulation. 
Additionally, \cite{noyan2022distributionally} propose an ambiguity set for distributionally robust programs by considering balls centered on a decision-dependent probability distribution by focusing on a class of distances including total variation distance
and the Wasserstein metrics to characterize these ambiguity sets. As the resulting reformulations can be non-convex, the authors present special cases for obtaining tractable reformulations for certain applications such as machine scheduling problem. 
Despite these recent studies, multimodality has not been considered within DRO literature involving decision-dependent uncertainties. Additionally, decision-dependent mode probabilities and the resulting ambiguity around it are omitted. 

In this paper, we propose a generic optimization under uncertainty framework for two-stage distributionally robust programs by leveraging multimodal distributions under various forms of ambiguity sets, where the decisions of the first-stage problem can impact the distribution of the modes as well as each distribution corresponding to these modes. 
To this end, our contributions can be summarized as follows: 
\begin{itemize}
    \item We introduce a new class of ambiguity sets for DRO problems with multimodal uncertainties where the ambiguity sets consider decision-dependent uncertainties for capturing the mode distribution and the corresponding modes. To the best of our knowledge, this is the first study to integrate multimodality within DRO with decision-dependent uncertainties. 
    \item We propose a $\phi$-divergence set to represent the decision-dependent mode probabilities. We then identify two special $\phi$-divergence cases, variation distance and $\chi^2$-distance, to derive tractable reformulations. We further provide potential dependence relationships between decisions and mode probabilities to be leveraged into the subsequent optimization model.  
    \item We propose moment-based and Wasserstein distance-based ambiguity sets to characterize the distributions corresponding to each mode. We integrate these sets into the sets describing decision-dependent mode probabilities and provide various reformulations. We further derive computationally tractable reformulations under special cases, which can be directly solved by the off-the-shelf solvers. 
    \item We {\color{black}theoretically} derive the value of having a multimodal distributionally robust decision-dependent framework against its single-modal counterpart by analyzing {\color{black}the relationship between multimodal and single-modal} ambiguity sets and comparing their {\color{black}optimal objective values}. 
    \item {\color{black}We provide a separation-based decomposition algorithm that can solve the resulting optimization models, which can provide finite convergence and optimality guarantees under certain settings. We demonstrate the computational efficiency of the proposed algorithm against the off-the-shelf solver over large-scale instances.}
    \item We present {\color{black} an extensive computational study over two example problem settings, facility location problem and shipment planning problem with pricing, to illustrate our results}. We evaluate the performance of the proposed framework under both moment-based and distance-based ambiguity sets against various frameworks including single-modal and decision-independent approaches. 
    Our results demonstrate that omitting the consideration of multimodalities {\color{black}or} decision dependent uncertainties in DRO models can lead to deficiently performing solutions with poor qualities from various directions. 
\end{itemize}

The rest of the paper is organized as follows: Section \ref{sec:ProblemFormulation} presents the problem formulation and the ambiguity sets corresponding to the decision-dependent mode probabilities. Section \ref{sec:MultimodalReformulations} provides reformulations under moment-based and Wasserstein-based ambiguity sets in representing the distribution of each mode by combining them with the ambiguities around the mode distributions. Section \ref{sec:MultimodalAnalyticalResults} provides the value of having multimodal ambiguity sets by analytically comparing them against their single-modal counterparts. Section \ref{sec:mode} provides potential approaches to formulate the relationship between decisions and mode probabilities. 
{\color{black}Section \ref{sec:SolutionAlg} proposes a decomposition algorithm to solve the resulting problems more efficiently.} 
Section \ref{sec:ComputationalStudy} presents our computational study to demonstrate the impact of the proposed framework and corresponding reformulations. Section \ref{sec:Conclusion} concludes the paper with final remarks. 



\section{Problem Formulation}
\label{sec:ProblemFormulation}

In this section, we introduce the following generic two-stage multimodal decision-dependent distributionally robust optimization ($\rm{D^3RO}$) model 
\begin{align}
   (\mbox{\bf{Multi-Modal }} {\rm D^3RO}): \min_{\boldsymbol{y}\in\mathcal{Y}} 
   \boldsymbol{c}^{\mathsf T} \boldsymbol{y}+\max_{\mathbb{P} \in \Theta(\boldsymbol{y})}\mathbb{E}_{\boldsymbol{\xi}\sim \mathbb{P}}[h(\boldsymbol{y},\boldsymbol{\xi})], \label{model:DRO} 
\end{align}
where $\boldsymbol{y}$ is the first-stage decision variable from the non-empty and compact feasible region $\mathcal{Y}\subseteq\mathbb{R}^I$. The random vector $\boldsymbol{\xi} \in \mathbb{R}^N$ has a support set $\Xi$ and its distribution is associated with $\mathbb{P}$, which belongs to a multimodal decision-dependent ambiguity set $\Theta(\boldsymbol{y})$, depending on the first-stage decisions $\boldsymbol{y}$.  
The costs of the first-stage and second-stage problems are represented by $\boldsymbol{c}^{\mathsf T} \boldsymbol{y}$ and $h(\boldsymbol{y},\boldsymbol{\xi})$, respectively. 
In terms of decision-dependent uncertainties, in addition to the ambiguity set $\Theta(\boldsymbol{y})$, we may also allow the uncertainty realization $\boldsymbol{\xi}$ to depend on decision $\boldsymbol{y}$, i.e., by replacing $\boldsymbol{\xi}$ with $\boldsymbol{\xi}(\boldsymbol{y})$, which is the case considered in Section \ref{sec:distance}; however, for notation simplicity, we will suppress the dependence of $\boldsymbol{\xi}$ on $\boldsymbol{y}$ over the generic problem setting. 
We note that we do not make any assumptions on the feasible region $\mathcal{Y}$ to derive our reformulations and corresponding results. However, in certain cases, we will assume $\boldsymbol{y}$ to be binary valued for computational tractability.



Following the notation in \cite{shapiro2009lectures,hanasusanto2018conic,xie2020tractable}, we consider the second-stage problem as follows
\begin{subequations}\label{eq:second-stage}
\begin{align}
h(\boldsymbol{y},\boldsymbol{\xi})=\min_{\boldsymbol{x}\in\mathbb{R}^J}\quad& (\boldsymbol{Q}\boldsymbol{\xi}+\boldsymbol{q})^{\mathsf T}\boldsymbol{x}\\
    \text{s.t.}\quad & \boldsymbol{T}(\boldsymbol{y})\boldsymbol{\xi}+W\boldsymbol{x}\ge \boldsymbol{R}(\boldsymbol{y}),
\end{align}
\end{subequations}
where $\boldsymbol{x}\in\mathbb{R}^J$ denotes the wait-and-see decisions in the second-stage problem. 
Here $\boldsymbol{T}(\boldsymbol{y})\in\mathbb{R}^{S\times N}$ and $\boldsymbol{R}(\boldsymbol{y})\in{\mathbb{R}^S}$ are matrix- and vector-valued affine functions. We make the following assumption throughout the paper.
\begin{assumption}[Relatively complete recourse]\label{assump:complete-recourse}
The second-stage problem $h(\boldsymbol{y},\boldsymbol{\xi})$ is feasible under every feasible first-stage decision $\boldsymbol{y} \in \mathcal{Y}$ and every realization of $\boldsymbol{\xi} \in \Xi$. 
\end{assumption}
Assumption \ref{assump:complete-recourse} is used to ensure that the feasible set for the second-stage problem is always non-empty and is satisfied by many operations research problems. 
{\color{black} We further assume $h(\boldsymbol{y},\boldsymbol{\xi})$ to be finite under every $\boldsymbol{y} \in \mathcal{Y}$ and $\boldsymbol{\xi} \in \Xi$ to ensure strong duality between \eqref{eq:second-stage} and its dual problem.} 

The ambiguity set $\Theta(\boldsymbol{y})$ in Model \eqref{model:DRO} contains two layers of decision-dependent robustness, where the first layer of ambiguity is due to the multimodality of the underlying distribution with uncertain mode probabilities corresponding to each mode $l \in \{1,\cdots,L\}$, and the second layer of ambiguity is for representing the distribution of each mode, denoted by $\mathbb{P}_l$. Furthermore, both layers of uncertainty can incorporate decision dependencies with the first-stage decision $\boldsymbol{y}\in\mathcal{Y}$. Thus, under a given first-stage decision $\boldsymbol{y}\in\mathcal{Y}$, $\Theta(\boldsymbol{y})$ is defined as
\begin{equation}
    \Theta(\boldsymbol{y})=\left\{\sum_{l=1}^L p_l\mathbb{P}_l:\ \boldsymbol{p}\in\Delta(\hat{p}(\boldsymbol{y})),\ \mathbb{P}_l\in\mathcal{U}_l(\boldsymbol{y}), \> l = 1, \cdots, L\right\}.\label{eq:genericMultiModalAmbiguity}
\end{equation}

Here, any element in $\Theta(\boldsymbol{y})$ can be represented as a mixture of $L$ probability distributions $\mathbb{P}_l$ with mode probability $p_l$ for all $l=1,\ldots, L$. Set $\Delta(\hat{p}(\boldsymbol{y}))$ includes all candidate mode probabilities based on a reference mode probability $\hat{p}(\boldsymbol{y}) \in \mathbb{R}^{L}_{+}$ with $\sum_{l=1}^L\hat{p}_l(\boldsymbol{y}) = 1$ for every $\boldsymbol{y} \in \mathcal{Y}$, and set $\mathcal{U}_l(\boldsymbol{y})$ contains all candidate probability distributions in each mode $l \in \{1,\cdots,L\}$ {\color{black}with a corresponding support set $\Xi_l$, where $\Xi = \cup_{l=1}^L \Xi_l$}. {\color{black}In practice, the number of modes $L$ represents the decision maker's knowledge of the uncertainty and can be chosen via cross-validation based on some historical data by minimizing the mean squared error.}

In this section, we focus on the first layer of ambiguity in representing mode probabilities and present potential sets to describe $\Delta(\hat{p}(\boldsymbol{y}))$. Then, we consider decision-dependent moment-based and distance-based ambiguity sets for describing $\mathcal{U}_l(\boldsymbol{y})$ and provide reformulations for a class of DRO problems with multimodal decision-dependent ambiguity sets in Section \ref{sec:moment} and in Section \ref{sec:distance}, respectively. 

To introduce $\Delta(\hat{p}(\boldsymbol{y}))$, we consider mode probabilities within a certain distance to the reference probability $\hat{p}(\boldsymbol{y})$ that can be represented through $\phi$-divergence \citep{liese2006divergences,ben2013robust} as follows:
\begin{equation} \label{eq:modeDist-PhiDivergence}
(\text{$\phi$-Divergence}) \quad \Delta(\hat{p}(\boldsymbol{y})):=\left\{\boldsymbol{p} \in \mathbb{R}^{L}_{+}: \ \sum_{l=1}^Lp_l=1, \ I_{\phi} (p, \hat{p}(y)) = \sum_{l=1}^L \hat{p}_l(y) \> \phi\left(\frac{p_l}{\hat{p}_l(y)}\right) \le \rho\right\},
\end{equation}
where $I_{\phi} (\cdot)$ represents the $\phi$-divergence between two probability distributions and $\rho$ corresponds to the robustness level. Here $\phi(t)$ is convex for $t \geq 0$, $\phi(1) = 0$, $0 \phi (a/0) = a \lim_{t \rightarrow \infty} \phi(t)/t$ for $a > 0$, and $0 \phi (0/0) = 0$. We note that when $\rho=0$, set \eqref{eq:modeDist-PhiDivergence} will reduce to a singleton $\Delta(\hat{p}(\boldsymbol{y}))=\{\hat{p}(\boldsymbol{y})\}$, which is suitable when mode probabilities can be accurately estimated as the reference distribution $\hat{p}(\boldsymbol{y})$. {\color{black}In practice, $\rho$ represents the decision-maker's confidence in the mode probability estimates and can be chosen via cross-validation.} The class of $\phi$-divergence includes many popular distances as special cases, including Kullback–Leibler divergence, Hellinger distance, $\chi^2$-distance, Cressie-Read distance, and many others. 

Next, we discuss two special cases with different choices of $\phi$-divergence functions. 
We first consider a special case of set \eqref{eq:modeDist-PhiDivergence} by choosing $\phi(t)=|t-1|$. This leads to the variation distance set defined by the $L_1$-norm with respect to the reference distribution as follows:
\begin{equation} \label{eq:modeDist-VariationDistance}
(\text{Variation Distance}) \quad \Delta(\hat{p}(\boldsymbol{y})):=\left\{\boldsymbol{p} \in \mathbb{R}^{L}_{+}: \ \sum_{l=1}^Lp_l=1, \ \sum_{l=1}^L| p_l - \hat{p}_l(\boldsymbol{y})| \leq \rho\right\}.
\end{equation}
Since this variation distance based set can be represented through a non-empty polyhedral set, our results in the subsequent sections under this setting can be applied to other $\Delta(\hat{p}(\boldsymbol{y}))$ sets with polyhedral representation. 

 As our second setting, we consider a special case of set \eqref{eq:modeDist-PhiDivergence} by leveraging a $\chi^2$-distance with respect to the reference distribution $\hat{p}(\boldsymbol{y})$ with its corresponding $\phi(t) = \frac{1}{t} (t-1)^2$. The resulting ambiguity set can be represented as follows:
\begin{equation} \label{eq:modeDist-ChiSquare}
(\text{$\chi^2$-Distance}) \quad \Delta(\hat{p}(\boldsymbol{y})):=\left\{\boldsymbol{p} \in \mathbb{R}^{L}_{+}: \ \sum_{l=1}^Lp_l=1, \ \sum_{l=1}^L\frac{(p_l-\hat{p}_l(\boldsymbol{y}))^2}{p_l}\le \rho\right\}.
\end{equation}
Note that $\chi^2$-distance is used in \cite{Hanasusanto2015_MultimodalDRO} to model uncertainty of mode probability in a newsvendor problem under a decision-independent setting with a moment-based ambiguity set. On the other hand, our work considers a generic two-stage DRO model with a broader class of distances (by using $\phi$-divergence) and incorporates decision-dependency in both layers of the ambiguity sets while considering both moment-based and distance-based ambiguity sets.

We first provide a dual reformulation of the two-stage multimodal $\rm{D^3RO}$ model \eqref{model:DRO} using the general $\phi$-divergence set defined in \eqref{eq:modeDist-PhiDivergence} in the following theorem. {\color{black} All the proofs in this section are presented in Appendix \ref{sec:omittedProofs}.}

\begin{theorem}[$\phi$-Divergence]\label{thm:DRO-PhiDivergence}
    Using the $\phi$-divergence set defined in \eqref{eq:modeDist-PhiDivergence}, the two-stage multimodal $\rm{D^3RO}$ model \eqref{model:DRO} can be reformulated as
    \begin{subequations}\label{model:DRO-PhiDivergence}
    \begin{align}
        \min_{\boldsymbol{y},\lambda,\eta,\boldsymbol{\psi}} \quad &\boldsymbol{c}^{\mathsf T}\boldsymbol{y}+\eta+\rho\lambda+\lambda\sum_{l=1}^L\hat{p}_l(\boldsymbol{y})\phi^*(\frac{\psi_l-\eta}{\lambda})\\
        \text{s.t.}\quad & \boldsymbol{y}\in\mathcal{Y},\ \lambda\ge 0\\
        & \psi_l=\max_{\mathbb{P}_l \in \mathcal{U}_l(\boldsymbol{y})}\mathbb{E}_{\mathbb{P}_l}[h(\boldsymbol{y},\boldsymbol{\xi})],\ \forall l=1,\ldots, L\label{eq:WorstCaseExpectation}
    \end{align}
    \end{subequations}
    where $\phi^*$ is the conjugate of $\phi$, i.e., $\phi^*(s)=\sup_{t\ge 0}\{st-\phi(t)\}$.
\end{theorem}
We refer interested readers to \cite{ben2013robust} for a list of different $\phi$-divergence functions and their conjugates. In this paper, we focus on variation distance and $\chi^2$-distance as two special cases and provide their reformulations in Theorems \ref{thm:DRO-VariationDistance} and \ref{thm:DRO-ChiDistance}, respectively.


\begin{theorem}[Variation Distance]\label{thm:DRO-VariationDistance}
    When $\phi(t)=|t-1|$, the $\phi$-divergence set becomes the variation distance set \eqref{eq:modeDist-VariationDistance}, 
    and the two-stage multimodal $\rm{D^3RO}$ model \eqref{model:DRO} can be reformulated as
    \begin{subequations}\label{model:variation-distance}
    \begin{align}
        \min_{\boldsymbol{y},\lambda,\eta,\boldsymbol{\psi}} \quad &\boldsymbol{c}^{\mathsf T}\boldsymbol{y}+\eta+\rho\lambda+\sum_{l=1}^L\hat{p}_l(\boldsymbol{y})r_l\\
        \text{s.t.}\quad & \boldsymbol{y}\in\mathcal{Y},\ \lambda\ge 0\\ &\text{\eqref{eq:WorstCaseExpectation}}\nonumber\\
        & r_l\ge \psi_l-\eta,\ \forall l=1,\ldots, L\\
        & r_l\ge -\lambda,\ \forall l=1,\ldots, L\\
        & \psi_l-\eta\le \lambda,\ \forall l=1,\ldots, L
    \end{align}
    \end{subequations}
\end{theorem}

\begin{remark}\label{remark:singleton}
    When $\rho=0$, we assume that we have a precise reference mode probability, i.e., $\Delta(\hat{p}(\boldsymbol{y}))=\{\hat{p}(\boldsymbol{y})\}$. In this case, $\lambda$ can be set to sufficiently large without penalty and $r_l=\psi_l-\eta$ at optimality. As a result,
    Model \eqref{model:variation-distance} reduces to $\min_{\boldsymbol{y}\in\mathcal{Y}}\{\boldsymbol{c}^{\mathsf T}\boldsymbol{y}+\sum_{l=1}^L\hat{p}_l(\boldsymbol{y})\psi_l:\text{\eqref{eq:WorstCaseExpectation}}\}$.
\end{remark}
\begin{theorem}[$\chi^2$-Distance] \label{thm:DRO-ChiDistance}
    When $\phi(t)=\frac{1}{t}(t-1)^2$, the $\phi$-divergence set becomes the $\chi^2$-distance set \eqref{eq:modeDist-ChiSquare}, and the two-stage multimodal $\rm{D^3RO}$ model \eqref{model:DRO} can be reformulated as
    \begin{subequations}\label{model:ChiDistance}
    \begin{align}
        \min_{\boldsymbol{y},\lambda,\eta,\boldsymbol{\psi}} \quad &\boldsymbol{c}^{\mathsf T}\boldsymbol{y}+\eta+\rho\lambda+2\lambda-2\sum_{l=1}^L\hat{p}_l(\boldsymbol{y})r_l\\
        \text{s.t.}\quad & \boldsymbol{y}\in\mathcal{Y},\ \lambda\ge 0\\
        &\text{\eqref{eq:WorstCaseExpectation}}\nonumber\\
        & \sqrt{r_l^2+\frac{1}{4}(\psi_l-\eta)^2}\le \lambda-\frac{1}{2}(\psi_l-\eta),\ \forall l=1,\ldots, L\\
        & \psi_l-\eta\le \lambda,\ \forall l=1,\ldots, L
    \end{align}
    \end{subequations}
\end{theorem}
In these derivations, the reference mode distribution $\hat{p}_l(\boldsymbol{y})$ is presented in its closed form. 
However, when an affine function in terms of the first-stage decision variables $\boldsymbol{y}$ is considered for its representation and $\boldsymbol{y}$ is binary valued, then 
Model \eqref{model:variation-distance} can be represented as a mixed-integer linear program (MILP) given that we provide linear constraints to reformulate Constraints \eqref{eq:WorstCaseExpectation} and leverage McCormick envelopes \citep{mccormick1976computability} to linearize bilinear terms. Similarly, under these assumptions, Model \eqref{model:ChiDistance} can be represented as a mixed-integer second-order conic program (MISOCP). 
Thus, from a computational perspective, it is more expensive to solve a multimodal $\rm{D^3RO}$ with $\chi^2$-distance than the variation distance. 
We note that we provide potential functional forms of the reference mode distribution $\hat{p}_l(\boldsymbol{y})$ in Section \ref{sec:mode}, which can be leveraged into different application settings and lead to computationally tractable reformulations under certain cases. 

\begin{remark}\label{remark:DD-SP}
Before we introduce different representations of ambiguity set $\mathcal{U}_l(\boldsymbol{y})$, let us first consider a special case when $\mathcal{U}_l(\boldsymbol{y})=\{\mathbb{P}_l\}$ is a singleton for each $l=1,\ldots, L$. This setting is suitable when we know the underlying distribution for each mode. In this case, Constraints \eqref{eq:WorstCaseExpectation} become $\psi_l=\mathbb{E}_{\mathbb{P}_l}[h(\boldsymbol{y},\boldsymbol{\xi})],\ \forall l=1,\ldots, L$, and this setting leads to a multimodal decision-dependent stochastic program (MM-DD-SP). Later in Section \ref{sec:distance}, we will discuss decision-dependent sample average approximation (SAA) models when we have access to decision-dependent data samples to estimate the true distribution $\mathbb{P}_l$.
\end{remark}

{\color{black}
\begin{remark}\label{remark:EventwiseAmbiguitySet}
We note that the proposed ambiguity set \eqref{eq:genericMultiModalAmbiguity} encapsulates the generic event-wise ambiguity set presented in \cite{Chen2020_DRO}, while extending this set by incorporating decision-dependencies and explicitly investigating reformulations under various forms of this ambiguity set.
Specifically, the event-wise ambiguity sets combine the features of stochastic programming and robust optimization by considering a set of discrete random variables 
that are associated with events and a set of continuous random variables whose ambiguous conditional distributions are based on the realizations of these events. 
Consequently, by associating modes with stochastic events, under a given first-stage decision $\boldsymbol{y}\in\mathcal{Y}$, $\Theta(\boldsymbol{y})$ can be recast as 
\begin{align}
    \Theta(\boldsymbol{y}) = \left\{ \mathbb{P}(\boldsymbol{y}) \middle\vert \begin{array}{l}
    (\tilde{\boldsymbol{\xi}}, \tilde{p}) \sim  \mathbb{P}(\boldsymbol{y})\\
    (\tilde{\boldsymbol{\xi}} | \tilde{p} = l)
    \sim \mathbb{P}_l \in \mathcal{U}_l(\boldsymbol{y}) \quad \forall l = 1, \cdots, L \\
    \mathbb{P}(\tilde{p} = l) = p_l \quad \forall l = 1, \cdots, L\\
        \boldsymbol{p} \in \Delta(\hat{p}(\boldsymbol{y})) 
  \end{array}\right\}.
\end{align}
We note that any distribution $\mathbb{P}(\boldsymbol{y}) \in \Theta(\boldsymbol{y})$ can be represented as $\sum_{l=1}^L p_l\mathbb{P}_l$, where the distribution of $\boldsymbol{p}$ belongs to the ambiguity set $\Delta(\hat{p}(\boldsymbol{y}))$, and the corresponding distribution of each mode $l = 1, \cdots, L$, represented by $\mathbb{P}_l$, belongs to the ambiguity set $\mathcal{U}_l(\boldsymbol{y})$. 
When $\Delta(\hat{p}(\boldsymbol{y}))$
and $\mathcal{U}_l(\boldsymbol{y})$ do not depend on the first-stage decision $\boldsymbol{y}$, then this special case can be associated with the 
generic event-wise ambiguity set presented in \cite{Chen2020_DRO}, which can capture certain moment-based and distance-based ambiguity sets. To this end, Theorem~\ref{thm:DRO-PhiDivergence} provides a reformulation under any form of ambiguity set $\mathcal{U}_l(\boldsymbol{y})$, when $\hat{p}(\boldsymbol{y})$ can be represented through $\phi$-divergence within set $\Delta(\hat{p}(\boldsymbol{y}))$, demonstrating generalizability of our framework. In the subsequent sections, we derive tractable reformulations by exploiting special forms of the sets $\mathcal{U}_l(\boldsymbol{y})$ and $\Delta(\hat{p}(\boldsymbol{y}))$. 
\end{remark}
}



\section{Tractability of Multi-Modal $\rm{D^3RO}$}
\label{sec:MultimodalReformulations}

In this section, we discuss the tractability of the two-stage multimodal $\rm{D^3RO}$ model \eqref{model:DRO}. Specifically, we focus on providing tractable reformulations {\color{black}for our generic reformulation presented in Theorem \ref{thm:DRO-PhiDivergence} by analyzing} Constraint \eqref{eq:WorstCaseExpectation} under moment-based ambiguity sets $\mathcal{U}_l(\boldsymbol{y})$ in Section \ref{sec:moment} and under distance-based ambiguity sets $\mathcal{U}_l(\boldsymbol{y})$ in Section \ref{sec:distance}, respectively.

\subsection{Moment-based Ambiguity Sets}\label{sec:moment}
When the distribution of the uncertain parameter is unknown, it may be possible to estimate some moment functions for the uncertain parameters from historical data. These nominal moment functions can be used to construct moment-based ambiguity sets to find robust decisions against distributional ambiguity.
We make the following assumption throughout Section \ref{sec:moment}.
\begin{assumption}\label{assumption:finite}
{\color{black}Every probability distribution  $\mathbb{P}\in\mathcal{U}_l(\boldsymbol{y})$ for each mode $l = 1, \cdots, L$ has a decision-independent and compact support set $\Xi_l$ for all solution values $\boldsymbol{y}\in \mathcal{Y}$.} 
\end{assumption}
{\color{black}
Note that for notation simplicity, the support set of the distributions in different modes are assumed to be decision-independent in Assumption \ref{assumption:finite}. However, all the reformulations derived in this section can be easily extended to settings with decision-dependent support sets. 
In practice, these support sets can be selected from the range of uncertain parameters, which can be represented through polyhedral formulations. For example, it is often possible to identify some lower and upper bounds on $\boldsymbol{\xi}$ such that $\boldsymbol{\xi}\in[\boldsymbol{l},\boldsymbol{u}]$. Alternatively, if a discrete support set is considered with $K$ elements, we can either sample $\boldsymbol{\xi}^k$ from this range or let $\boldsymbol{\xi}^k$ be equally spaced in this range from $k = 1, \cdots, K$. In Section \ref{sec:distance}, we will discuss ways to construct decision-dependent support points for distance-based ambiguity sets.}

We define moment-based ambiguity sets $\mathcal{U}_l(\boldsymbol{y})$ by providing bounds on some moment functions following \cite{luo2020distributionally, yu2022multistage}. Given moment basis function vector 
{\color{black}$\boldsymbol{f}(\boldsymbol{\xi})=[f_{m} (\boldsymbol{\xi}),\ m \in \{1, \cdots, M\}]^{\mathsf T}$ with $M$ different moment functions}, 
the moment matching ambiguity set for mode $l$ is given by
{\color{black}
\begin{align} 
\mathcal{U}_l(\boldsymbol{y})=\mathcal{M}(\underline{\boldsymbol{u}}_l(\boldsymbol{y}), \bar{\boldsymbol{u}}_l(\boldsymbol{y})):=
\left\{\mathbb{P}_l \in \mathcal{P}(\Xi_l): \int_{\Xi_l} f_{m} (\boldsymbol{\xi}) \mathbb{P}_l(d\boldsymbol{\xi}) \in [\underline{u}_{l,m}(\boldsymbol{y}), \bar{u}_{l,m}(\boldsymbol{y})], \ m \in \{1, \cdots, M\} \right\},
\label{eq:moment-based}
\end{align}
}
where {\color{black}$\mathcal{P}(\Xi_l)$ 
is the set of all probability measures defined on the support set $\Xi_l$,} 
and $\underline{\boldsymbol{u}}_l(\boldsymbol{y}){\color{black}\le} \bar{\boldsymbol{u}}_l(\boldsymbol{y})$ are decision-dependent lower and upper bounds on the corresponding moment functions, respectively. 
{\color{black}We note that all the probability distributions $\mathbb{P}_l \in \mathcal{P}(\Xi_l)$ satisfy $\int_{\Xi_l} \mathbb{P}_l(d\boldsymbol{\xi}) = 1$ by definition.} 
We next derive monolithic reformulations to represent the multimodal $\rm{D^3RO}$ model \eqref{model:DRO} under two special cases of $\phi$-divergence set representing mode probability set $\Delta(\hat{p}(\boldsymbol{y}))$ in combination with the moment-based ambiguity set \eqref{eq:moment-based}. 


\subsubsection{Variation Distance based Multimodal Ambiguity with Moment-based Setting}
\begin{theorem}[Variation Distance + Moment-based]
\label{thm:MomentBasedVariation}
   If for any feasible $\boldsymbol{y}\in\mathcal{Y}$, the ambiguity set $\mathcal{U}_l(\boldsymbol{y})$ defined in \eqref{eq:moment-based} {\color{black}has a non-empty relative interior}, 
   then the multimodal $\rm{D^3RO}$ model \eqref{model:DRO} with variation distance set $\Delta(\hat{p}(\boldsymbol{y}))$ defined in \eqref{eq:modeDist-VariationDistance} and moment-based ambiguity set $\mathcal{U}_l(\boldsymbol{y})$ defined in \eqref{eq:moment-based} is equivalent to
    \begin{subequations}\label{model:variation+moment-general}
    \begin{align}   \min_{\boldsymbol{y},\lambda,\eta, r_l, \alpha_l,\underline{\boldsymbol{\beta}}_l,\bar{\boldsymbol{\beta}}_l} \quad &\boldsymbol{c}^{\mathsf T}\boldsymbol{y}+\eta+\rho\lambda+\sum_{l=1}^L\hat{p}_l(\boldsymbol{y})r_l\\
        \text{s.t.}\quad & \boldsymbol{y}\in\mathcal{Y},\ \lambda,\ \underline{\boldsymbol{\beta}}_l,\ \bar{\boldsymbol{\beta}}_l\ge 0,\ \forall l=1,\ldots, L \label{model:variation+moment-general-ConstrFirst}\\
        &  \alpha_l+\bar{\boldsymbol{\beta}}_l^{\mathsf T}\bar{\boldsymbol{u}}_l(\boldsymbol{y})-\underline{\boldsymbol{\beta}}_l^{\mathsf T}\underline{\boldsymbol{u}}_l(\boldsymbol{y})-\eta\le r_l,\ \forall l=1,\ldots,L \label{model:variation+moment-general-ConstrSecond}\\
        & \alpha_l+\bar{\boldsymbol{\beta}}_l^{\mathsf T}\bar{\boldsymbol{u}}_l(\boldsymbol{y})-\underline{\boldsymbol{\beta}}_l^{\mathsf T}\underline{\boldsymbol{u}}_l(\boldsymbol{y})-\eta\le \lambda,\ \forall l=1,\ldots, L\\
         & r_l\ge -\lambda,\ \forall l=1,\ldots, L \label{model:variation+moment-general-ConstrEnd}\\
         & {\color{black}\alpha_l+(\bar{\boldsymbol{\beta}}_l-\underline{\boldsymbol{\beta}}_l)^{\mathsf T}\boldsymbol{f}(\boldsymbol{\xi})\ge h(\boldsymbol{y},\boldsymbol{\xi}),\ \forall l=1,\ldots, L,\ \boldsymbol{\xi} \in \Xi_l} \label{eq:MomentBasedVariation-SupportConstr} 
    \end{align}
    \end{subequations}
\end{theorem}
\begin{proof}
    The maximization problem in Constraint \eqref{eq:WorstCaseExpectation} can be formulated as the linear program below
    {\color{black}
    \begin{align*}
        \psi_l=\max_{\mathbb{P}_l}\quad& \int_{\Xi_l} h(\boldsymbol{y},\boldsymbol{\xi}) \mathbb{P}_l(d\boldsymbol{\xi})\\
        \text{s.t.}\quad
        & \int_{\Xi_l} \mathbb{P}_l(d\boldsymbol{\xi}) = 1\\
        &\int_{\Xi_l} \boldsymbol{f}(\boldsymbol{\xi}) \mathbb{P}_l(d\boldsymbol{\xi})\ge\underline{\boldsymbol{u}}_l(\boldsymbol{y})\\
        &\int_{\Xi_l} \boldsymbol{f}(\boldsymbol{\xi}) \mathbb{P}_l(d\boldsymbol{\xi})\le\bar{\boldsymbol{u}}_l(\boldsymbol{y})\\
        &\mathbb{P}_l \in \mathcal{P'}(\Xi_l), 
    \end{align*}
    where $\mathcal{P'}(\Xi_l)$ represents the set of positive measures defined on the support set $\Xi_l$. We next assign dual variables $\alpha_l$ and $\underline{\boldsymbol{\beta}}_l,\ \bar{\boldsymbol{\beta}}_l$ to the above constraints. Then, we write the dual program as follows: 
    \begin{subequations}\label{model:MomentMatchingDual}
    \begin{align}        
\psi_l=\min_{\alpha_l,\underline{\boldsymbol{\beta}}_l,\bar{\boldsymbol{\beta}}_l}\quad & \alpha_l+\bar{\boldsymbol{\beta}}_l^{\mathsf T}\bar{\boldsymbol{u}}_l(\boldsymbol{y})-\underline{\boldsymbol{\beta}}_l^{\mathsf T}\underline{\boldsymbol{u}}_l(\boldsymbol{y})\\
        \text{s.t.}\quad& \alpha_l+(\bar{\boldsymbol{\beta}}_l-\underline{\boldsymbol{\beta}}_l)^{\mathsf T}\boldsymbol{f}(\boldsymbol{\xi})\ge h(\boldsymbol{y},\boldsymbol{\xi}),\ \forall \boldsymbol{\xi} \in \Xi_l \\ 
        &\underline{\boldsymbol{\beta}}_l\ge 0,\ \bar{\boldsymbol{\beta}}_l\ge 0
    \end{align}
    \end{subequations}
Since the ambiguity set $\mathcal{U}_l(\boldsymbol{y})$ is non-empty and by following the proof of Theorem 3.3 from  \cite{luo2020distributionally}, we can ensure strong duality between these primal and dual problems.
}
    Combining Model \eqref{model:MomentMatchingDual} with Theorem \ref{thm:DRO-VariationDistance}, we get the desired result.
\end{proof}

\subsubsection{$\chi^2$-Distance based Multimodal Ambiguity with Moment-based Setting}
\begin{theorem}[$\chi^2$-Distance + Moment-based]
\label{thm:MomentBasedChiSquare}
    If for any feasible $\boldsymbol{y}\in\mathcal{Y}$, the ambiguity set $\mathcal{U}_l(\boldsymbol{y})$ defined in \eqref{eq:moment-based} {\color{black}has a non-empty relative interior}, 
    then the multimodal $\rm{D^3RO}$ model \eqref{model:DRO} with $\chi^2$-distance set $\Delta(\hat{p}(\boldsymbol{y}))$ defined in \eqref{eq:modeDist-ChiSquare} and moment-based ambiguity set $\mathcal{U}_l(\boldsymbol{y})$ defined in \eqref{eq:moment-based} is equivalent to
    \begin{subequations}\label{model:ChiSquared+Moment}
        \begin{align}     \min_{\boldsymbol{y},\lambda,\eta,\boldsymbol{\psi}} \quad &\boldsymbol{c}^{\mathsf T}\boldsymbol{y}+\eta+\rho\lambda+2\lambda-2\sum_{l=1}^L\hat{p}_l(\boldsymbol{y})r_l\\
        \text{s.t.}\quad &\boldsymbol{y}\in\mathcal{Y},\ \lambda,\ \underline{\boldsymbol{\beta}}_l,\ \bar{\boldsymbol{\beta}}_l\ge 0,\ \forall l=1,\ldots, L\\
        & \psi_l\ge \alpha_l+\bar{\boldsymbol{\beta}}_l^{\mathsf T}\bar{\boldsymbol{u}}_l(\boldsymbol{y})-\underline{\boldsymbol{\beta}}_l^{\mathsf T}\underline{\boldsymbol{u}}_l(\boldsymbol{y}),\ \forall l=1,\ldots, L\\
        & \sqrt{r_l^2+\frac{1}{4}(\psi_l-\eta)^2}\le \lambda-\frac{1}{2}(\psi_l-\eta),\ \forall l=1,\ldots, L\\
        & \psi_l-\eta\le \lambda,\ \forall l=1,\ldots, L\\
        & {\color{black}\alpha_l+(\bar{\boldsymbol{\beta}}_l-\underline{\boldsymbol{\beta}}_l)^{\mathsf T}\boldsymbol{f}(\boldsymbol{\xi})\ge h(\boldsymbol{y},\boldsymbol{\xi}),\ \forall l=1,\ldots, L,\ \boldsymbol{\xi} \in \Xi_l} \label{eq:MomentBasedChiSquare-SupportConstr} 
    \end{align}
    \end{subequations}
\end{theorem}
\begin{proof}
    Combining Model \eqref{model:MomentMatchingDual} with Theorem \ref{thm:DRO-ChiDistance}, we get the desired result. 
\end{proof}

\begin{remark}
{\color{black}We note that our reformulations in Theorems \ref{thm:MomentBasedVariation} and \ref{thm:MomentBasedChiSquare} result in semi-infinite programs with infinite number of constraints, corresponding to the constraints \eqref{eq:MomentBasedVariation-SupportConstr} and \eqref{eq:MomentBasedChiSquare-SupportConstr}, respectively.}
Since these programs are not directly solvable by the off-the-shelf solvers, different approaches can be designed to provide their solutions including cutting-surface based and column-and-constraint generation based solution algorithms \citep{Mehrotra2014_CuttingSurface, zeng2013ColumnandConstraintGeneration}, and decision rules for their approximations \citep{Hanasusanto2015_MultimodalDRO, Bertsimas2019_DroDecisionRules}. 
{\color{black}To address this issue with continuous support sets, we propose a separation-based decomposition algorithm in Section \ref{sec:SolutionAlg} and further demonstrate its benefits in our computational study in Section \ref{sec:ComputationalStudy}.} 
Another extension of our model can be to include other moment-based ambiguity sets, such as the one proposed in \cite{Delage2010}. Although our results can be extended to this setting with decision-dependencies by considering the mean vector to lie in an ellipsoid centered at a decision-dependent estimate of the mean vector, and the second moment matrix to lie in a positive semi-definite cone leveraging a decision-dependent matrix, the resulting reformulation leads to a semi-definite program (SDP). The resulting SDP can further have binary or integer valued decision variables, depending on how the feasible region of the first-stage problem $\mathcal{Y}$ is defined and how the decision-dependent mode probabilities and moment functions are constructed, leading to mixed-integer SDPs with further computational difficulties. 
In Theorems \ref{thm:MomentBasedVariation} and \ref{thm:MomentBasedChiSquare}, reformulations are presented over generic decision-dependent lower and upper bounds of the moment-functions. To present computationally tractable reformulations, we provide special cases in Section \ref{sec:SpecialCases-MomentBased}, where the reformulation under the variation distance based multimodal ambiguity setting can result in a MILP formulation and the reformulation under the $\chi^2$-distance based multimodal ambiguity setting can result in a MISOCP formulation, under certain assumptions. 
\end{remark}




\subsubsection{Special Cases} 
\label{sec:SpecialCases-MomentBased}
As a special case, {\color{black}to obtain a tractable reformulation, we consider the setting where the support set of each mode $l = 1, \cdots, L$, is finite with $K$ elements, i.e., $\Xi_l:=\{\boldsymbol{\xi}^k\}_{k=1}^K$, while considering} the first and second moment of each uncertain parameter for the moment basis function $\boldsymbol{f}(\boldsymbol{\xi}^k)$ (i.e., take $M=2N$) and specifying their lower and upper bounds as follows:
\begin{subequations}
\begin{align}
&f_{n}(\boldsymbol{\xi}^k)=\xi_{n}^k,\ \underline{u}_{l,n}(\boldsymbol{y})=\mu_{l,n}(\boldsymbol{y})-\epsilon_{l,n}^{\mu}, \
\bar{u}_{l,n}(\boldsymbol{y})=\mu_{l,n}(\boldsymbol{y})+\epsilon_{l,n}^{\mu},\ \forall n\in[N] ,\label{eq:first_moment}\\
&f_{N+n}(\boldsymbol{\xi}^k)=(\xi_{n}^k)^2,\
\underline{u}_{l,N+n}(\boldsymbol{y})=S_{l,n}(\boldsymbol{y})\underline{\epsilon}_{l,n}^S,\
\bar{u}_{l,N+n}(\boldsymbol{y})=S_{l,n}(\boldsymbol{y})\bar{\epsilon}_{l,n}^S,\ \forall n\in[N]. \label{eq:second_moment}
\end{align}
\end{subequations}
Here, equation \eqref{eq:first_moment} concerns the first moment of each parameter $\xi_n^k$ and uses $\underline{u}_{l,n}(\boldsymbol{y})$ and $\bar{u}_{l,n}(\boldsymbol{y})$ to bound the first moment of parameter $\xi_{n}$ in an $\epsilon_{l,n}^{\mu}$-interval of the empirical mean function $\mu_{l,n}(\boldsymbol{y})$ for all $n\in[N]$. Similarly, equation \eqref{eq:second_moment} uses $\underline{u}_{l,N+n}(\boldsymbol{y})$ and $\bar{u}_{l,N+n}(\boldsymbol{y})$ to bound the second moment of parameter $\xi_{n}$ via scaling the empirical second moment function $S_{l,n}(\boldsymbol{y})$ with scaling parameters $\underline{\epsilon}_{l,n}^S$ and $\bar{\epsilon}_{l,n}^S$ for all $n\in[N]$. We further assume that the empirical first moment $\mu_{l,n}(\boldsymbol{y})$ and second moment $S_{l,n}(\boldsymbol{y})$ affinely depend on the first-stage decision $\boldsymbol{y}$, such that
\begin{align*}
&\mu_{l,n}(\boldsymbol{y})=\bar\mu_{l,n}\left(1+\sum_{i=1}^I\lambda_{l,n,i}^{\mu}y_{i}\right),\\
&S_{l,n}(\boldsymbol{y})=(\bar\mu_{l,n}^2+\bar\sigma_{l,n}^2)\left(1+\sum_{i=1}^I\lambda_{l,n,i}^Sy_{i}\right),
\end{align*}
where the empirical mean and standard deviation of the $n$-th uncertain parameter in mode $l$ are denoted by $\bar\mu_{l,n},\ \bar\sigma_{l,n}$, respectively. {\color{black} Here by assumption, the first and second moments will increase when any of the first-stage variables $y_{i}$ increases. Parameters $\lambda_{l,n,i}^{\mu},\ \lambda_{l,n,i}^S\in\mathbb{R}_{+}$ respectively represent the degree about how one unit of change in $y_{i}$ may affect the values of the first and second moments of $\xi_{n}$ for each $n\in[N]$. In practice, these parameters can be obtained by performing a regression analysis on historical data. For example, in facility location problems, $y_i$ is a binary decision variable, representing whether a facility is open at site $i$. The first and second moments of the random customer demand at site $n$ may increase as we open more facilities nearby, and the increasing rate may depend on the distance ($\text{dist}_{ni}$) between the facility $i$ and the customer site $n$, e.g., $\lambda_{l,n,i}^{\mu}=e^{a \cdot\text{dist}_{ni}+b}$ (see \cite{basciftci2021distributionally} for more discussions). Given a dataset $\mathcal{D}_K:=\{(\boldsymbol{y}^k,\boldsymbol{\text{dist}}^k,\boldsymbol{\mu}^k)\}_{k=1}^K$, one can use different regression approaches to get the best $a,b$ parameters by minimizing the mean squared error. We refer interested readers to \cite{sun2025contextual} for a case study on electric vehicle charging station location problems, where the authors use real-world data to learn the decision-dependent demand.}

Under this special case, the two-stage multimodal $\rm{D^3RO}$ model \eqref{model:variation+moment-general} using variation distance set $\Delta(\hat{p}(\boldsymbol{y}))$ defined in \eqref{eq:modeDist-VariationDistance} and moment-based ambiguity set $\mathcal{U}_l(\boldsymbol{y})$ defined in \eqref{eq:moment-based} can be recast as follows
\begin{subequations}\label{model:variation+moment-special}
    \begin{align}   \min_{\boldsymbol{y},\lambda,\eta,r_l,\alpha_l,\underline{\boldsymbol{\beta}}_l,\bar{\boldsymbol{\beta}}_l} \quad &\boldsymbol{c}^{\mathsf T}\boldsymbol{y}+\eta+\rho\lambda+\sum_{l=1}^L\hat{p}_l(\boldsymbol{y})r_l\\
        \text{s.t.}\quad & \boldsymbol{y}\in\mathcal{Y},\ \lambda,\ \underline{\boldsymbol{\beta}}_l,\ \bar{\boldsymbol{\beta}}_l\ge 0,\ \forall l=1,\ldots, L\\
        & \alpha_l+\sum_{n=1}^N\bar{\beta}_{l,n}(\bar\mu_{l,n}+\epsilon_{l,n}^{\mu})+\sum_{n=1}^N\sum_{i=1}^I\lambda_{l,n,i}^{\mu}\bar\mu_{l,n}\bar{\beta}_{l,n}y_{i} +\sum_{n=1}^N\bar{\beta}_{l,N+n}(\bar\mu_{l,n}^2+\bar\sigma_{l,n}^2)\bar{\epsilon}^S_{l,n}\nonumber\\
        &+\sum_{n=1}^N\sum_{i=1}^I\lambda^S_{l,n,i}\bar{\epsilon}^S_{l,n}(\bar\mu_{l,n}^2+\bar\sigma_{l,n}^2)\bar{\beta}_{l,N+n}y_{i}-\sum_{n=1}^N\underline{\beta}_{l,1+n}(\bar\mu_{l,n}-\epsilon_{l,n}^{\mu})-\sum_{n=1}^N\sum_{i=1}^I\lambda_{l,n,i}^{\mu}\bar\mu_{l,n}\underline{\beta}_{l,n}y_{i}\nonumber\\
\quad& -\sum_{n=1}^N\underline{\beta}_{l,N+n}(\bar\mu_{l,n}^2+\bar\sigma_{l,n}^2)\underline{\epsilon}^S_{l,n}-\sum_{n=1}^N\sum_{i=1}^I\lambda^S_{l,n,i}\underline{\epsilon}^S_{l,n}(\bar\mu_{l,n}^2+\bar\sigma_{l,n}^2)\underline{\beta}_{l,N+n}y_{i}-\eta\le \min\{r_l,\lambda\},\nonumber\\
&\hspace{27em}\forall l=1,\ldots,L\label{eq:Variation+Moment-Bilinear}\\
         & r_l\ge -\lambda,\ \forall l=1,\ldots, L\\
       & \alpha_l+\sum_{n\in [N]}\xi_{n}^k(\bar{\beta}_{l,n}-\underline{\beta}_{l,n})+\sum_{n\in[N]}(\xi_{n}^k)^2(\bar{\beta}_{l,N+n}-\underline{\beta}_{l,N+n})\ge h(\boldsymbol{y},\boldsymbol{\xi}^k),\nonumber\\
       &\hspace{21em}\forall l=1,\ldots, L,\ k=1,\ldots,K 
    \end{align}
    \end{subequations}
Model \eqref{model:variation+moment-special} gives rise to a non-convex optimization problem in general, since there are bilinear terms $\bar{\beta}_{l,n}y_{i},\ \bar{\beta}_{l,N+n}y_{i},$ $\underline{\beta}_{l,n}y_{i}, \ \underline{\beta}_{l,N+n}y_{i}$ in Constraints \eqref{eq:Variation+Moment-Bilinear}. If we further assume the first-stage decision variable $y_{i}$ to be binary valued,  we can provide exact reformulations of these bilinear terms using McCormick envelopes \citep{mccormick1976computability}. For example, if $\bar{\beta}_{l,n}\in [l^{\beta}_{l,n},u^{\beta}_{l,n}]$, we linearize the bilinear term $z_{l,n,i}=\bar{\beta}_{l,n}y_{i}$ as follows: 
\begin{subequations}\label{eq:McCormick}
    \begin{align}
& z_{l,n,i}\le \bar{\beta}_{l,n}-l^{\beta}_{l,n}(1-y_i)\\ 
& z_{l,n,i}\ge \bar{\beta}_{l,n}-u^{\beta}_{l,n}(1-y_i)\\ 
&  z_{l,n,i}\le  u^{\beta}_{l,n}y_i\\
&  z_{l,n,i}\ge  l^{\beta}_{l,n}y_i
\end{align}
\end{subequations}
For the sake of simplicity, we denote Constraints \eqref{eq:McCormick} as $(z_{l,n,i}, \bar{\beta}_{l,n}, y_{i})\in {\color{black}\mathcal{MC}}_{(l^{\beta}_{l,n},u^{\beta}_{l,n})}$. Later in Section \ref{sec:mode}, we will discuss different approaches to model the decision-dependent mode probabilities $\hat{p}(\boldsymbol{y})$ and present tractable reformulations for Model \eqref{model:variation+moment-special}. 

\subsection{Distance-based Ambiguity Sets}\label{sec:distance}

Different than the moment-based ambiguity sets introduced in Section \ref{sec:moment}, in a data-driven framework, we may have access to $K_l$ decision-dependent data samples $\{\hat{\boldsymbol{\xi}}_{lk}(\boldsymbol{y})\}_{k=1}^{K_l}$ from the true distribution under mode $l$. One possible way to construct these decision-dependent samples is to train regression models to learn the latent decision-dependency first and then use empirical residuals to build an empirical distribution $\hat{\mathbb{P}}_l(\boldsymbol{y})=\frac{1}{K_l}\sum_{k=1}^{K_l}\delta_{\hat{\boldsymbol{\xi}}_{lk}(\boldsymbol{y})}$. Specifically, given a dataset $\mathcal{D}_{K_l}:=\{(\boldsymbol{y}_{lk}, \boldsymbol{\xi}_{lk})\}_{k=1}^{K_l}$ from each mode $l$, we first estimate a regression function $\hat{g}_{l}(\boldsymbol{y})$ and construct empirical residuals $\boldsymbol{\epsilon}_{lk}:= \boldsymbol{\xi}_{lk} - \hat{g}_{l}(\boldsymbol{y}_{lk})$ for $k=1,\ldots, K_l$. Then the decision-dependent support points can be constructed as $\hat{\boldsymbol{\xi}}_{lk}(\boldsymbol{y})=\hat{g}_{l}(\boldsymbol{y})+\boldsymbol{\epsilon}_{lk}$. We refer interested readers to  \cite{kannan2022data,kannan2023residuals} for empirical residuals-based SAA and DRO approaches {\color{black}and \cite{zhu2024residuals} for a decision-dependent extension to the empirical residuals-based DRO approach}. 

Centered at the empirical distribution $\hat{\mathbb{P}}_l(\boldsymbol{y})$, we focus on Type-1 Wasserstein ambiguity sets  \citep{esfahani2018data} that are defined as follows:
\begin{equation} 
\mathcal{U}_l(\boldsymbol{y})=\mathbb{B}_{\epsilon_l}(\hat{\mathbb{P}}_l(\boldsymbol{y}))=\{\mathbb{P}_l\in\mathcal{P}(\Xi_l): \mathcal{W}_q(\mathbb{P}_l, \hat{\mathbb{P}}_l(\boldsymbol{y}))\le \epsilon_l\}\label{eq:distance-based}
\end{equation}
where 
$\mathcal{W}_q(\cdot,\cdot)$ is defined as:
$$\mathcal{W}_q(Q_1,Q_2):=\inf \left\{\int_{\Xi^2}||\xi_1-\xi_2||_q\Pi(d\xi_2,d\xi_2): \substack{\Pi\text{ is a joint distribution of $\xi_1$ and $\xi_2$}\\ \text{with marginals $Q_1$ and $Q_2$}}\right\}$$
with norm $||\cdot||_q$ denoting the reference distance in $\mathbb{R}^N$ corresponding to the Wasserstein distance metric between the distributions $Q_1$ and $Q_2$. 
Consequently, the ambiguity set defined in \eqref{eq:distance-based} considers a Wasserstein ball of radius $\epsilon_l$ for each mode $l=1,\cdots,L$ centered around the reference decision-dependent distribution $\hat{\mathbb{P}}_l(\boldsymbol{y})$. 


To obtain reformulations under the Wasserstein-based ambiguity set, we first consider our generic problem setting, and then further analyze two different cases depending on the structure of the second-stage cost function $h(\boldsymbol{y},\boldsymbol{\xi})$ defined in \eqref{eq:second-stage} when (i) the uncertainty only affects the objective ($\boldsymbol{T}(\boldsymbol{y})=0$) and (ii) the uncertainty only affects the constraints ($\boldsymbol{Q}=0$). For computational tractability, we present the results for objective uncertainty in the main manuscript and move the results for constraint uncertainty to Appendix \ref{append:constraint}.
 


\subsubsection{Variation Distance based Multimodal Ambiguity with Wasserstein-based Setting}
\label{sec:VariationDistanceWasserstein}

In this section, we first introduce the generic reformulation under variation distance based multimodal ambiguity with Wasserstein-based set \eqref{eq:distance-based} and then derive an additional result under objective uncertainty. 
Furthermore, we provide a special case of the proposed reformulations by highlighting the relationship between the decision-dependent distributionally robust models and their stochastic programming counterparts. 

\begin{theorem}[Variation Distance + Wasserstein-based]
\label{thm:VariationWasserstein}
    Using the $\Delta(\hat{p}(\boldsymbol{y}))$ defined in \eqref{eq:modeDist-VariationDistance} and $\mathcal{U}_l(\boldsymbol{y})$ defined in \eqref{eq:distance-based}, the two-stage multimodal $\rm{D^3RO}$ model \eqref{model:DRO} is equivalent to
        \begin{subequations}\label{model:distance-equivalent-poly}
    \begin{align}
\min_{\boldsymbol{y},\lambda,\eta,\boldsymbol{\psi}} \quad &\boldsymbol{c}^{\mathsf T}\boldsymbol{y}+\eta+\rho\lambda+\sum_{l=1}^L\hat{p}_l(\boldsymbol{y})r_l\\
        \text{s.t.}\quad & \boldsymbol{y}\in\mathcal{Y},\ \lambda\ge 0\\
        & \epsilon_l\gamma_l+\frac{1}{K_l}\sum_{k=1}^{K_l}w_{lk}-\eta\le r_l,\ \forall l=1,\ldots, L\\
        & r_l\ge -\lambda,\ \forall l=1,\ldots, L\\
        & \epsilon_l\gamma_l+\frac{1}{K_l}\sum_{k=1}^{K_l}w_{lk}-\eta\le \lambda,\ \forall l=1,\ldots, L\\
        &[-h_y]^*(\boldsymbol{z}_{lk}-\boldsymbol{\nu}_{lk})+\sigma_{\Xi_l}(\boldsymbol{\nu}_{lk})-\boldsymbol{z}_{lk}^{\mathsf T}\hat{\boldsymbol{\xi}}_{lk}(\boldsymbol{y})\le w_{lk},\ \forall k=1,\ldots,K_l,\ l=1,\ldots, L\label{eq:Wasserstein-Conjugate}\\
        & ||\boldsymbol{z}_{lk}||_{q^*}\le \gamma_l,\ \forall k=1,\ldots,K_l,\ l=1,\ldots, L,\label{eq:Wasserstein-DualNorm}
    \end{align}
    \end{subequations}
    where $h_y(\boldsymbol{\xi}):=h(\boldsymbol{y},\boldsymbol{\xi})$ with $\boldsymbol{y}$ suppressed and $[-h_y]^*(\boldsymbol{z}_{lk})$ is the conjugate of function $-h_y$ evaluated at $\boldsymbol{z}_{lk}$, i.e., $[-h_y]^*(\boldsymbol{z}_{lk})=\sup_{\boldsymbol{\xi}\in\mathbb{R}^N}\{\boldsymbol{z}_{lk}^{\mathsf T}\boldsymbol{\xi}+h(\boldsymbol{y},\boldsymbol{\xi})\}$.
\end{theorem}
\begin{proof}
    According to \cite{esfahani2018data}, Theorem 4.2, we have
    \begin{subequations}\label{model:WassersteinDual}
    \begin{align}
       \psi_l=\inf\quad& \epsilon_l\gamma_l+\frac{1}{K_l}\sum_{k=1}^{K_l}w_{lk}\\
        \text{s.t.}\quad &  [-h_y]^*(\boldsymbol{z}_{lk}-\boldsymbol{\nu}_{lk})+\sigma_{\Xi_l}(\boldsymbol{\nu}_{lk})-\boldsymbol{z}_{lk}^{\mathsf T}\hat{\boldsymbol{\xi}}_{lk}(\boldsymbol{y})\le w_{lk},\ \forall k=1,\ldots,K_l,\ l=1,\ldots, L\label{eq:Wasserstein-Conjugate}\\
        & ||\boldsymbol{z}_{lk}||_{q^*}\le \gamma_l,\ \forall k=1,\ldots,K_l,\ l=1,\ldots, L.\label{eq:Wasserstein-DualNorm}
    \end{align}
    \end{subequations}
    Combining Model \eqref{model:WassersteinDual} with Theorem \ref{thm:DRO-VariationDistance} yields the desired result.
\end{proof}

Next, we derive the following result under objective uncertainty of the second-stage problem, when polyhedral uncertainty sets are considered to define the support of the distribution of each mode. 

\begin{theorem}[Variation Distance + Wasserstein-based + Objective Uncertainty]
\label{thm:VariationWasserstein_Obj}
    Suppose $\boldsymbol{T}(\boldsymbol{y})=0$, $\Xi_l=\{\boldsymbol{\xi}:\ \boldsymbol{C}_l\boldsymbol{\xi}\le \boldsymbol{d}_l\}$, and for any given $\boldsymbol{y}\in\mathcal{Y}$, the feasible region $\{\boldsymbol{x}:W\boldsymbol{x}\ge \boldsymbol{R}(\boldsymbol{y})\}$ is non-empty and compact. The two-stage multimodal $\rm{D^3RO}$ model \eqref{model:DRO} with variation distance set $\Delta(\hat{p}(\boldsymbol{y}))$ defined in \eqref{eq:modeDist-VariationDistance} and Wasserstein ambiguity set $\mathcal{U}_l(\boldsymbol{y})$ defined in \eqref{eq:distance-based} can be tractable for any $q\in[1,\infty]$ and admits the following equivalent formulation:
    \begin{subequations}\label{model:Variation+Wasserstein-Objective}
    \begin{align}
    \min\quad &\boldsymbol{c}^{\mathsf T}\boldsymbol{y}+\eta+\rho\lambda+\sum_{l=1}^L\hat{p}_l(\boldsymbol{y})r_l\\
        \text{s.t.}\quad
         & \boldsymbol{y}\in\mathcal{Y},\ \lambda,\ \boldsymbol{\mu}_{lk}\ge 0,\ \forall k=1,\ldots,K_l,\ l=1,\ldots, L\label{eq:Variation-Wasserstein-domain}\\
        & \epsilon_l\gamma_l+\frac{1}{K_l}\sum_{k=1}^{K_l}w_{lk}-\eta\le r_l,\ \forall l=1,\ldots, L\\
        & r_l\ge -\lambda,\ \forall l=1,\ldots, L\\
        & \epsilon_l\gamma_l+\frac{1}{K_l}\sum_{k=1}^{K_l}w_{lk}-\eta\le \lambda,\ \forall l=1,\ldots, L \label{eq:Variation-Wasserstein-domain2}\\
        &(\boldsymbol{Q}\hat{\boldsymbol{\xi}}_{lk}(\boldsymbol{y})+\boldsymbol{q})^{\mathsf T}\boldsymbol{x}_{lk}+(\boldsymbol{d}_l-\boldsymbol{C}_l\hat{\boldsymbol{\xi}}_{lk}(\boldsymbol{y}))^{\mathsf T}\boldsymbol{\mu}_{lk}\le w_{lk},\ \forall k=1,\ldots,K_l,\ l=1,\ldots, L\label{eq:Variation+Wasserstein-Bilinear}\\
    &\boldsymbol{W}\boldsymbol{x}_{lk}\ge \boldsymbol{R}(\boldsymbol{y}),\ \forall k=1,\ldots,K_l,\ l=1,\ldots,L \label{eq:Variation-Wasserstein-Constr} \\
        & ||\boldsymbol{Q}^{\mathsf T}\boldsymbol{x}_{lk}-\boldsymbol{C}_l^{\mathsf T}\boldsymbol{\mu}_{lk}||_{q^*}\le \gamma_l,\ \forall k=1,\ldots,K_l,\ l=1,\ldots,L, \label{eq:Variation-Wasserstein-dualGamma}
    \end{align}
    \end{subequations}
    where $\frac{1}{q}+\frac{1}{q^*}=1$.
\end{theorem}
\begin{proof}
Because $\boldsymbol{T}(\boldsymbol{y})=0$ and the feasible region $\{\boldsymbol{x}:W\boldsymbol{x}\ge \boldsymbol{R}(\boldsymbol{y})\}$ is always non-empty and compact, we have
    \begin{align*}
        [-h_y]^*(\boldsymbol{z}_{lk}-\boldsymbol{\nu}_{lk})=&\sup_{\boldsymbol{\xi}}\{(\boldsymbol{z}_{lk}-\boldsymbol{\nu}_{lk})^{\mathsf T}\boldsymbol{\xi}+\inf_{\boldsymbol{x}_{lk}}\{(\boldsymbol{Q}\boldsymbol{\xi}+\boldsymbol{q})^{\mathsf T}\boldsymbol{x}_{lk}: \boldsymbol{W}\boldsymbol{x}_{lk}\ge \boldsymbol{R}(\boldsymbol{y})\}\}\\
        =&\inf_{\boldsymbol{x}_{lk}}\{\boldsymbol{q}^{\mathsf T}\boldsymbol{x}_{lk}+\sup_{\boldsymbol{\xi}}\{\boldsymbol{\xi}^{\mathsf T}(\boldsymbol{z}_{lk}-\boldsymbol{\nu}_{lk}+\boldsymbol{Q}^{\mathsf T}\boldsymbol{x}_{lk})\}: \boldsymbol{W}\boldsymbol{x}_{lk}\ge \boldsymbol{R}(\boldsymbol{y})\}\\
        =&\begin{cases}
            \boldsymbol{q}^{\mathsf T}\boldsymbol{x}_{lk},\ \text{if there exists $\boldsymbol{x}_{lk}$ with $\boldsymbol{z}_{lk}-\boldsymbol{\nu}_{lk}=-\boldsymbol{Q}^{\mathsf T}\boldsymbol{x}_{lk}$ and $\boldsymbol{W}\boldsymbol{x}_{lk}\ge \boldsymbol{R}(\boldsymbol{y})$}\\
            +\infty,\ \text{otherwise}
        \end{cases}
    \end{align*}
    On the other hand, 
    \begin{align*}
        \sigma_{\Xi_l}(\boldsymbol{\nu}_{lk})=\begin{cases}\sup_{\boldsymbol{\xi}_{lk}}\boldsymbol{\nu}_{lk}^{\mathsf T}\boldsymbol{\xi}_{lk}\\
        \text{s.t.}\quad \boldsymbol{C}_l\boldsymbol{\xi}_{lk}\le \boldsymbol{d}_l
        \end{cases}=\begin{cases}
            \inf_{\mu_{lk}\ge 0}\boldsymbol{d}_l^{\mathsf T}\boldsymbol{\mu}_{lk}\\
            \text{s.t.}\quad \boldsymbol{C}_l^{\mathsf T}\boldsymbol{\mu}_{lk}=\boldsymbol{\nu}_{lk}
        \end{cases}
    \end{align*}
    As a result, Constraints \eqref{eq:Wasserstein-Conjugate}-\eqref{eq:Wasserstein-DualNorm} become
    \begin{subequations}\label{eq:Wasserstein-h-conjugate}
    \begin{align}
         &(\boldsymbol{Q}\hat{\boldsymbol{\xi}}_{lk}(\boldsymbol{y})+\boldsymbol{q})^{\mathsf T}\boldsymbol{x}_{lk}+(\boldsymbol{d}_l-\boldsymbol{C}_l\hat{\boldsymbol{\xi}}_{lk}(\boldsymbol{y}))^{\mathsf T}\boldsymbol{\mu}_{lk}\le w_{lk},\ \forall k=1,\ldots,K_l,\ l=1,\ldots, L\\
    &\boldsymbol{W}\boldsymbol{x}_{lk}\ge \boldsymbol{R}(\boldsymbol{y}),\ \forall k=1,\ldots,K_l,\ l=1,\ldots,L\\
        & ||\boldsymbol{Q}^{\mathsf T}\boldsymbol{x}_{lk}-\boldsymbol{C}_l^{\mathsf T}\boldsymbol{\mu}_{lk}||_{q^*}\le \gamma_l,\ \forall k=1,\ldots,K_l,\ l=1,\ldots,L\\
        & \mu_{lk}\ge 0, \ \forall k=1,\ldots, K_l,\ l=1,\ldots,L
    \end{align}
    \end{subequations}
    Combining Model \eqref{model:variation-distance} with the above constraints, we get the desired results.
\end{proof}

\begin{remark}
    According to Remark \ref{remark:singleton}, when $\rho=0$, Model \eqref{model:Variation+Wasserstein-Objective} reduces to $\min\{\boldsymbol{c}^{\mathsf T}\boldsymbol{y}+\sum_{l=1}^L\hat{p}_l(\boldsymbol{y})(\epsilon_l\gamma_l+\frac{1}{K_l}\sum_{k=1}^{K_l}w_{lk}):\text{\eqref{eq:Variation-Wasserstein-domain},\eqref{eq:Variation+Wasserstein-Bilinear}--\eqref{eq:Variation-Wasserstein-dualGamma}}\}$. Furthermore, if $\epsilon_l=0,\ \forall l=1,\ldots, L$, we are in an \textit{ambiguity-free} setting and Model \eqref{model:Variation+Wasserstein-Objective} reduces to a decision-dependent sample average approximation (DD-SAA) problem. Indeed, for $\epsilon_l=0$, the variable $\gamma_l$ can be set to sufficiently large at no penalty, and thus $\boldsymbol{\mu}_{lk}=0$ and $w_{lk}=(\boldsymbol{Q}\hat{\boldsymbol{\xi}}_{lk}(\boldsymbol{y})+\boldsymbol{q})^{\mathsf T}\boldsymbol{x}_{lk}$ at optimality. In this case, Model \eqref{model:Variation+Wasserstein-Objective} is equivalent to
    \begin{subequations}\label{model:DD-SAA}
    \begin{align}
     \mbox{{\bf (DD-SAA)}}:  \min\quad& \boldsymbol{c}^{\mathsf T}\boldsymbol{y}+\sum_{l=1}^L\frac{\hat{p}_l(\boldsymbol{y})}{K_l}\sum_{k=1}^{K_l}(\boldsymbol{Q}\hat{\boldsymbol{\xi}}_{lk}(\boldsymbol{y})+\boldsymbol{q})^{\mathsf T}\boldsymbol{x}_{lk}\\
        \text{s.t.}\quad& \boldsymbol{y}\in\mathcal{Y}\\
        &\boldsymbol{W}\boldsymbol{x}_{lk}\ge \boldsymbol{R}(\boldsymbol{y}),\ \forall k=1,\ldots,K_l,\ l=1,\ldots,L
    \end{align}
    \end{subequations}
    The DD-SAA model \eqref{model:DD-SAA} is a sample average approximation of the MM-DD-SP we introduced in Remark \ref{remark:DD-SP} when $\rho=0$. Note that we do not call Model \eqref{model:DD-SAA} a multimodal DD-SAA since we will show that it is equivalent to a single-modal DD-SAA in Remark \ref{remark:DD-SAA}.
\end{remark}


\subsubsection{$\chi^2$-Distance based Multimodal Ambiguity with Wasserstein-based Setting}

Similar to Section \ref{sec:VariationDistanceWasserstein}, in this section, we first present the generic reformulation under $\chi^2$-distance based multimodal ambiguity with Wasserstein-based set \eqref{eq:distance-based} and then derive an additional result under objective uncertainty. 

\begin{theorem}[$\chi^2$-Distance + Wasserstein-based]\label{thm:ChiDistance+Wasserstein}
    Using the $\chi^2$-distance set $\Delta(\hat{p}(\boldsymbol{y}))$ defined in \eqref{eq:modeDist-ChiSquare} and Wasserstein ambiguity set $\mathcal{U}_l(\boldsymbol{y})$ defined in \eqref{eq:distance-based}, the two-stage multimodal $\rm{D^3RO}$ model \eqref{model:DRO} is equivalent to
    \begin{subequations}
            \begin{align}
        \min_{\boldsymbol{y},\lambda,\eta,\boldsymbol{\psi}} \quad &\boldsymbol{c}^{\mathsf T}\boldsymbol{y}+\eta+\rho\lambda+2\lambda-2\sum_{l=1}^L\hat{p}_l(\boldsymbol{y})r_l\\
        \text{s.t.}\quad & \boldsymbol{y}\in\mathcal{Y},\ \lambda\ge 0\\
        & \sqrt{r_l^2+\frac{1}{4}(\psi_l-\eta)^2}\le \lambda-\frac{1}{2}(\psi_l-\eta),\ \forall l=1,\ldots, L\\
        & \psi_l-\eta\le \lambda,\ \forall l=1,\ldots, L\\
        & \epsilon_l\gamma_l+\frac{1}{K_l}\sum_{k=1}^{K_l}w_{lk}\le \psi_l,\ \forall l=1,\ldots, L\\
        &[-h_y]^*(\boldsymbol{z}_{lk}-\boldsymbol{\nu}_{lk})+\sigma_{\Xi_l}(\boldsymbol{\nu}_{lk})-\boldsymbol{z}_{lk}^{\mathsf T}\hat{\boldsymbol{\xi}}_{lk}(\boldsymbol{y})\le w_{lk},\ \forall k=1,\ldots,K_l,\ l=1,\ldots, L\\
        & ||\boldsymbol{z}_{lk}||_{q^*}\le \gamma_l,\ \forall k=1,\ldots,K_l,\ l=1,\ldots, L
    \end{align}
    \end{subequations}
\end{theorem}
\begin{proof}
    Combining Model \eqref{model:WassersteinDual} with Theorem \ref{thm:DRO-ChiDistance} yields the desired result.
\end{proof}

\begin{theorem}[$\chi^2$-Distance + Wasserstein-based + Objective Uncertainty]
\label{thm:ChiSquareWasserstein_Obj}
    Suppose $\boldsymbol{T}(\boldsymbol{y})=0$, $\Xi_l=\{\boldsymbol{\xi}:\ \boldsymbol{C}_l\boldsymbol{\xi}\le \boldsymbol{d}_l\}$, and for any given $\boldsymbol{y}\in\mathcal{Y}$, the feasible region $\{\boldsymbol{x}:
    W\boldsymbol{x}\ge \boldsymbol{R}(\boldsymbol{y})\}$ is non-empty and compact. Then the two-stage multimodal ${\rm D^3RO}$ model \eqref{model:DRO} with $\chi^2$-distance set $\Delta(\hat{p}(\boldsymbol{y}))$ defined in \eqref{eq:modeDist-ChiSquare} and Wasserstein ambiguity set $\mathcal{U}_l(\boldsymbol{y})$ defined in \eqref{eq:distance-based} can be tractable for any $q\in[1,\infty]$ and admits the following equivalent formulation:
    \begin{subequations}\label{model:ChiSquare+Wasserstein}
    \begin{align}
          \min_{\boldsymbol{y},\lambda,\eta,\boldsymbol{\psi}} \quad &\boldsymbol{c}^{\mathsf T}\boldsymbol{y}+\eta+\rho\lambda+2\lambda-2\sum_{l=1}^L\hat{p}_l(\boldsymbol{y})r_l\\
        \text{s.t.}\quad & \boldsymbol{y}\in\mathcal{Y},\ \lambda,\ \boldsymbol{\mu}_{lk}\ge 0,\ \forall k=1,\ldots,K_l,\ l=1,\ldots,L\\
        & \sqrt{r_l^2+\frac{1}{4}(\psi_l-\eta)^2}\le \lambda-\frac{1}{2}(\psi_l-\eta),\ \forall l=1,\ldots, L\\
        & \psi_l-\eta\le \lambda,\ \forall l=1,\ldots, L\\
        & \epsilon_l\gamma_l+\frac{1}{K_l}\sum_{k=1}^{K_l}w_{lk}\le \psi_l,\ \forall l=1,\ldots, L\\
          &(\boldsymbol{Q}\hat{\boldsymbol{\xi}}_{lk}(\boldsymbol{y})+\boldsymbol{q})^{\mathsf T}\boldsymbol{x}_{lk}+(\boldsymbol{d}_l-\boldsymbol{C}_l\hat{\boldsymbol{\xi}}_{lk}(\boldsymbol{y}))^{\mathsf T}\boldsymbol{\mu}_{lk}\le w_{lk},\ \forall k=1,\ldots,K_l,\ l=1,\ldots,L\\
    &\boldsymbol{W}\boldsymbol{x}_{lk}\ge \boldsymbol{R}(\boldsymbol{y}),\ \forall k=1,\ldots,K_l,\ l=1,\ldots,L\\
        & ||\boldsymbol{Q}^{\mathsf T}\boldsymbol{x}_{lk}-\boldsymbol{C}_l^{\mathsf T}\boldsymbol{\mu}_{lk}||_{q^*}\le \gamma_l,\ \forall k=1,\ldots,K_l,\ l=1,\ldots,L,
    \end{align}
    \end{subequations}
    where $\frac{1}{q}+\frac{1}{q^*}=1$.
\end{theorem}
\begin{proof}
    Combining Constraints \eqref{eq:Wasserstein-h-conjugate} with Theorem \ref{thm:ChiDistance+Wasserstein} yields the desired result.
\end{proof}

In Theorems \ref{thm:VariationWasserstein}- \ref{thm:ChiSquareWasserstein_Obj}, reformulations are presented over generic reference distributions with decision-dependent realizations. To present computationally tractable reformulations, we provide special cases in Section \ref{sec:SpecialCases-DistanceBased}, where the reformulations for the objective uncertainty setting under the variation distance based multimodal ambiguity setting can result in a MILP formulation and the reformulation under the $\chi^2$-distance based multimodal ambiguity setting can result in a MISOCP formulation, under certain assumptions. 


\subsubsection{Special Cases}
\label{sec:SpecialCases-DistanceBased}

As a special case, we consider the setting where  the uncertainty realization affinely depends on the first-stage decision variable, i.e.,
\begin{align*}
& \hat{\xi}_{l,k,n}(\boldsymbol{y})=\bar{\xi}_{l,k,n}+\sum_{i=1}^I\lambda^{\xi}_{l,k,n,i}{y}_i,\ \forall l=1,\ldots, L,\ k=1,\ldots,K_l,\ n=1,\ldots, N
\end{align*}
Under this setting, Constraints \eqref{eq:Variation+Wasserstein-Bilinear} become
\begin{align*}  &\boldsymbol{q}^{\mathsf T}\boldsymbol{x}_{lk}+\sum_{j=1}^Jx_{lkj}\sum_{n=1}^NQ_{jn}\bar{\xi}_{l,k,n}+\sum_{j=1}^J\sum_{n=1}^NQ_{jn}\sum_{i=1}^I\lambda^{\xi}_{l,k,n,i}x_{lkj}y_i\\
&+\boldsymbol{d}_l^{\mathsf T}\boldsymbol{\mu}_{lk}-\sum_{h=1}^H\mu_{lkh}\sum_{n=1}^NC_{lhn}\bar{\xi}_{l,k,n}-\sum_{h=1}^H\sum_{n=1}^NC_{lhn}\sum_{i=1}^I\lambda^{\xi}_{l,k,n,i}\mu_{lkh}y_i\le w_{lk},\ \forall k=1,\ldots,K_l,\ l=1,\ldots, L
\end{align*}
Given binary valued first-stage decisions $y_i$, we can provide exact reformulations for bilinear terms $x_{lkj}y_i$ and $\mu_{lkh}y_i$ using McCormick envelopes. 
Consequently, we can extend our results in Theorems \ref{thm:VariationWasserstein_Obj} and \ref{thm:ChiSquareWasserstein_Obj} under this setting to obtain efficient reformulations. 
We remind the readers that in Section \ref{sec:mode}, we will discuss different approaches to model the decision-dependent mode probabilities $\hat{p}(\boldsymbol{y})$ and present tractable reformulations for the aforementioned results under certain cases.

\section{Value of Multimodality}
\label{sec:MultimodalAnalyticalResults}

In Models \eqref{model:variation+moment-general}, \eqref{model:ChiSquared+Moment}, \eqref{model:Variation+Wasserstein-Objective} and \eqref{model:ChiSquare+Wasserstein}, each constraint is repeated for $L$ times corresponding to each different mode, which is much more challenging to solve compared to a traditional single-modal model, which corresponds to the case of $L=1$. Then, to evaluate the value of the proposed approach and compare multimodal and single-modal ambiguity sets, we aim to answer the following questions: Is it worth solving this computationally expensive multimodal model? What is the additional benefit of multimodal framework compared to the single-modal one? Indeed, as we will show in Remark \ref{remark:DD-SAA}, if we are using an SAA approach and have a precise mode probability ($\rho=0$), then there is no difference between a multimodal stochastic program and a single-modal one that combines the information from multiple modes. However, if we only have partial information on the distribution in each mode and construct moment-based or distance-based ambiguity sets for DRO models (as we illustrated in Sections \ref{sec:moment} and \ref{sec:distance}), then we can show that our multimodal model can always obtain an in-sample cost that is at least as good as the model that fuses the information from different modes to a single distribution. In this section, we consider the variation distance set \eqref{eq:modeDist-VariationDistance} as $\Delta(\hat{\boldsymbol{p}}(\boldsymbol{y}))$ and illustrate the benefit of considering a multimodal DRO model using both moment-based and distance-based ambiguity sets in Theorems \ref{thm:benefit-moment} and \ref{thm:benefit-distance}, respectively. 
\begin{remark}\label{remark:DD-SAA}
    We first consider a multimodal DD-SAA setting when $\rho=0$: $\min_{\boldsymbol{y}\in\mathcal{Y}}\{\boldsymbol{c}^{\mathsf T}\boldsymbol{y}+\sum_{l=1}^L\frac{\hat{p}_l(\boldsymbol{y})}{K_l}\sum_{k=1}^{K_l}h(\boldsymbol{y},\hat{\boldsymbol{\xi}}_{lk}(\boldsymbol{y}))\}$, where we use the empirical distribution supported on $K_l$ data samples ($\frac{1}{K_l}\sum_{k=1}^{K_l}\delta_{\hat{\boldsymbol{\xi}}_{lk}(\boldsymbol{y})}$) to approximate the true distribution $\mathbb{P}_l$ in each mode $l$. Note that this multimodal setting can be equivalently recast as a single-modal setting where we group all the data samples from different modes into one mode, i.e., having $\sum_{l=1}^LK_l$ scenarios $\{\hat{\boldsymbol{\xi}}_{lk}(\boldsymbol{y})\}_{k=1,\ldots, K_l,\ l=1,\ldots,L}$, each with probability $\frac{\hat{p}_l(\boldsymbol{y})}{K_l}$. We provide an illustration on a toy example in Figure \ref{fig:SP}.
\begin{figure}[ht!]
    \centering
    \includegraphics[width=0.6\textwidth]{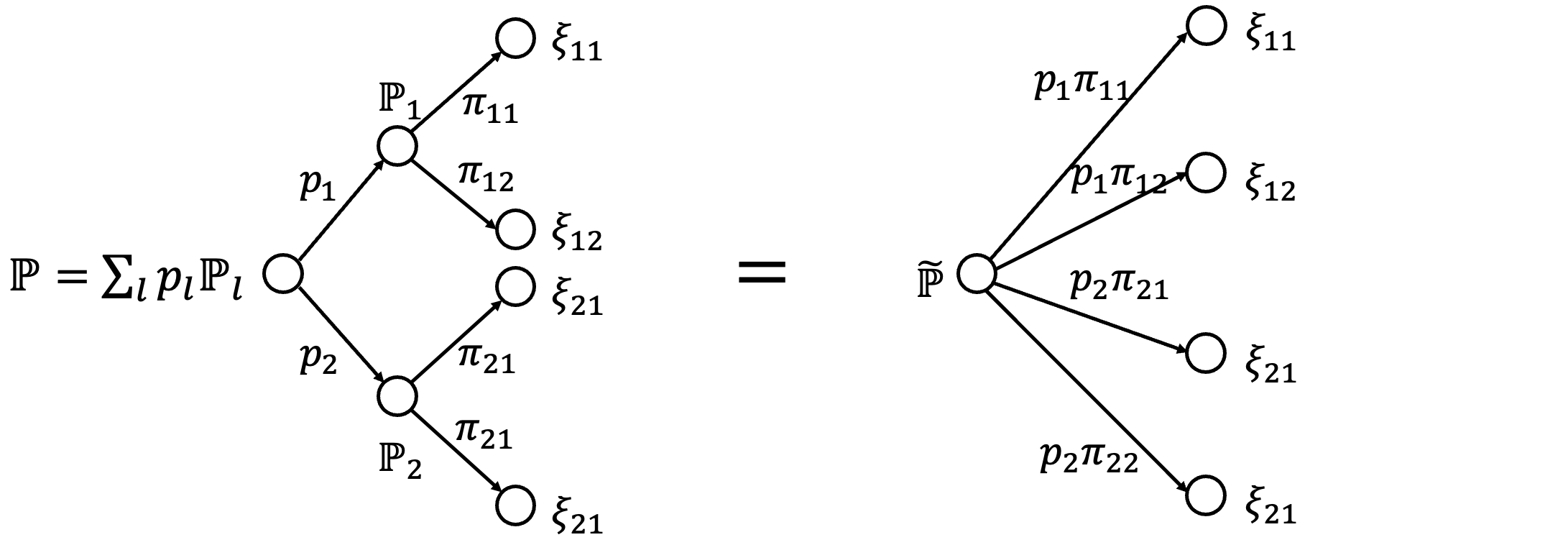}
    \caption{The equivalence between multimodal SAA (left) and single-modal SAA (right).}
    \label{fig:SP}
\end{figure}
\end{remark}

For deriving the value of multimodality under the moment-based ambiguity setting, we first define the multimodal decision-dependent ambiguity set $\Theta_M(\boldsymbol{y})$ by considering the ambiguity set \eqref{eq:genericMultiModalAmbiguity} under the variation distance set \eqref{eq:modeDist-VariationDistance} for $\Delta(\hat{\boldsymbol{p}}(\boldsymbol{y}))$ and moment-based set \eqref{eq:moment-based} for $\mathcal{U}_l(\boldsymbol{y})$. 
On the other hand, one alternative to this set $\Theta_M(\boldsymbol{y})$ is to group the moment bounds under $L$ different modes to a single moment bound and construct the following single-modal moment-based ${\rm D^3RO}$ while leveraging the ambiguity around the mode probabilities using the definition of the variation distance. We define this alternative problem setting as follows:
\begin{align}
   (\mbox{\bf{Single-Modal Moment-based }} {\rm D^3RO}): \min_{\boldsymbol{y}\in\mathcal{Y}} \boldsymbol{c}^{\mathsf T} \boldsymbol{y}+\max_{\mathbb{P} \in \Theta_M^{\prime}(\boldsymbol{y})}\mathbb{E}_{\boldsymbol{\xi}\sim \mathbb{P}}[h(\boldsymbol{y},\boldsymbol{\xi})] \label{model:DRO-singlemodalMoment}
\end{align}
where
{\color{black}
\begin{align}
    \Theta_M^{\prime}(\boldsymbol{y})=&\Big\{\mathbb{P} \in \mathcal{P}(\Xi): \int_{\Xi} f_{m} (\boldsymbol{\xi}) \mathbb{P}(d\boldsymbol{\xi}) \in \left[\sum_{l=1}^L(\hat{p}_l(\boldsymbol{y})-\rho)\underline{u}_{l,m}(\boldsymbol{y}), \sum_{l=1}^L(\hat{p}_l(\boldsymbol{y})+\rho)\bar{u}_{l,m}(\boldsymbol{y})\right],\nonumber\\
    &\hspace{10.5cm} m \in \{1, \cdots, M\} \Big\}.\label{eq:moment-based-grouped}
\end{align}
}
{\color{black}
    Here, for presentation purposes, we assume  $\bar{\boldsymbol{u}}_l(\boldsymbol{y})\ge\underline{\boldsymbol{u}}_l(\boldsymbol{y})\ge 0$ for every feasible decision $\boldsymbol{y}\in\mathcal{Y}$ and mode $l=1,\ldots,L$ to ensure that $\sum_{l=1}^L(\hat{p}_l(\boldsymbol{y})+\rho)\bar{u}_{l,m}(\boldsymbol{y})\ge \sum_{l=1}^L(\hat{p}_l(\boldsymbol{y})-\rho)\underline{u}_{l,m}(\boldsymbol{y})$. We note that this assumption is not restrictive. Indeed, if $\boldsymbol{\xi}\ge 0$ (e.g., when we model the customer demand as the uncertainty), a non-trivial lower or upper bound on the moment function $\boldsymbol{f}(\boldsymbol{\xi})$ should always be non-negative. 
    If this assumption is not satisfied (e.g., $\underline{\boldsymbol{u}}_l(\boldsymbol{y})\le\bar{\boldsymbol{u}}_l(\boldsymbol{y})\le 0$),  then we can revise the bounds in \eqref{eq:moment-based-grouped} to be $[\sum_{l=1}^L(\hat{p}_l(\boldsymbol{y})+\rho)\underline{u}_{l,m}(\boldsymbol{y}), \sum_{l=1}^L(\hat{p}_l(\boldsymbol{y})-\rho)\bar{u}_{l,m}(\boldsymbol{y})]$ and derive the corresponding results.
    }

Note that when $\rho=0$, the lower and upper bound in \eqref{eq:moment-based-grouped} become a convex combination of the lower and upper bounds ($\underline{\boldsymbol{u}}_l(\boldsymbol{y})$, $\bar{\boldsymbol{u}}_l(\boldsymbol{y})$) in each mode defined in \eqref{eq:moment-based}, using the mode probability {\color{black}estimate} $\hat{p}_l(\boldsymbol{y})$ as the weight. {\color{black} When we are not confident about the mode probability estimates (i.e., $\rho\not=0$), we can adjust the weights to take into account the ambiguity around $\hat{p}_l(\boldsymbol{y})$ by subtracting and adding $\rho$ for the lower and upper bound, respectively.
}
We show that the multimodal moment-based ${\rm D^3RO}$ \eqref{model:DRO} under the ambiguity set $\Theta_M(\boldsymbol{y})$  always results in an in-sample cost that is at least as good as its single-modal counterpart \eqref{model:DRO-singlemodalMoment} in the next theorem. 


\begin{theorem}\label{thm:benefit-moment}
    If we use the Variation distance set \eqref{eq:modeDist-VariationDistance} as $\Delta(\hat{\boldsymbol{p}}(\boldsymbol{y}))$ {\color{black}and moment-based set \eqref{eq:moment-based} as $\mathcal{U}_l(\boldsymbol{y})$ and assume $\bar{\boldsymbol{u}}_l(\boldsymbol{y}),\underline{\boldsymbol{u}}_l(\boldsymbol{y})\ge 0$ for every feasible decision $\boldsymbol{y}\in\mathcal{Y}$ and mode $l=1,\ldots,L$}, then the multimodal ambiguity set $\Theta_M(\boldsymbol{y})$ is nested in the single-modal set $\Theta_M^{\prime}(\boldsymbol{y})$, i.e., $\Theta_M(\boldsymbol{y})\subseteq\Theta_M^{\prime}(\boldsymbol{y})$. As a result, the optimal objective value of the multimodal moment-based ${\rm D^3RO}$ \eqref{model:DRO} is no more than the one of the single-modal counterpart \eqref{model:DRO-singlemodalMoment}, i.e.,
    \begin{align} \min_{\boldsymbol{y}\in\mathcal{Y}} \boldsymbol{c}^{\mathsf T} \boldsymbol{y}+\max_{P \in \Theta_M(\boldsymbol{y})}\mathbb{E}_{P_\xi}[h(\boldsymbol{y},\boldsymbol{\xi})]\le \min_{\boldsymbol{y}\in\mathcal{Y}} \boldsymbol{c}^{\mathsf T} \boldsymbol{y}+\max_{P \in \Theta_M^{\prime}(\boldsymbol{y})}\mathbb{E}_{P_\xi}[h(\boldsymbol{y},\boldsymbol{\xi})].\label{eq:MultimodalComparison-Moment}
    \end{align}
\end{theorem}
\begin{proof}
    For any $\mathbb{P}=\sum_{l=1}^L p_l\mathbb{P}_{l}\in\Theta_M(\boldsymbol{y})$, we have {\color{black}$\sum_{l=1}^Lp_l=1$}, $|p_l-\hat{p}_l(\boldsymbol{y})|\le\sum_{l=1}^L|p_l-\hat{p}_l(\boldsymbol{y})|\le \rho$ and $\mathbb{P}_l\in\mathcal{M}(\underline{\boldsymbol{u}}_l(\boldsymbol{y}),\bar{\boldsymbol{u}}_l(\boldsymbol{y}))$. {\color{black}First of all, $\mathbb{P}\in\mathcal{P}(\Xi)$ because $\int_{\Xi}\mathbb{P}(d\boldsymbol{\xi})=\sum_{l=1}^Lp_l\int_{\Xi_l}\mathbb{P}_l(d\boldsymbol{\xi})=\sum_{l=1}^Lp_l=1$. Furthermore,
    \begin{align*}
         &\int_{\Xi} f_{m} (\boldsymbol{\xi}) \mathbb{P}(d\boldsymbol{\xi})=\sum_{l=1}^Lp_l\int_{\Xi_l} f_{m} (\boldsymbol{\xi}) \mathbb{P}_l(d\boldsymbol{\xi})\le \sum_{l=1}^L p_l\bar{u}_{l,m}(\boldsymbol{y})\le \sum_{l=1}^L (\hat{p}_l(\boldsymbol{y})+\rho)\bar{u}_{l,m}(\boldsymbol{y}),\ \forall m\in\{1,\ldots,M\}\\
        &\int_{\Xi} f_{m} (\boldsymbol{\xi}) \mathbb{P}(d\boldsymbol{\xi})=\sum_{l=1}^Lp_l\int_{\Xi_l} f_{m} (\boldsymbol{\xi}) \mathbb{P}_l(d\boldsymbol{\xi})\ge \sum_{l=1}^L p_l\underline{u}_{l,m}(\boldsymbol{y})\ge \sum_{l=1}^L (\hat{p}_l(\boldsymbol{y})-\rho)\underline{u}_{l,m}(\boldsymbol{y}),\ \forall m\in\{1,\ldots,M\},
    \end{align*}
    }
    which implies that $\mathbb{P}=\sum_{l=1}^L p_l\mathbb{P}_{l}\in\Theta_M^{\prime}(\boldsymbol{y})$. Since $\Theta_M(\boldsymbol{y})\subseteq\Theta_M^{\prime}(\boldsymbol{y})$, we have $\max_{\mathbb{P}\in \Theta_M^{\prime}(\boldsymbol{y})}\mathbb{E}_{\mathbb{P}}[h(\boldsymbol{y},\boldsymbol{\xi})]\ge \max_{\mathbb{P}\in \Theta_M(\boldsymbol{y})}\mathbb{E}_{\mathbb{P}}[h(\boldsymbol{y},\boldsymbol{\xi})]$.
\end{proof}

\begin{remark}
    The equality in \eqref{eq:MultimodalComparison-Moment} holds when $\underline{\boldsymbol{u}}_l(\boldsymbol{y})=\underline{\boldsymbol{u}}(\boldsymbol{y}),\ \bar{\boldsymbol{u}}_l(\boldsymbol{y})=\bar{\boldsymbol{u}}(\boldsymbol{y}),\ \forall l=1,\ldots, L$ and $\rho=0$. Indeed, for any $\mathbb{P}\in \Theta_M^{\prime}(\boldsymbol{y})$, if $\underline{\boldsymbol{u}}_l(\boldsymbol{y})=\underline{\boldsymbol{u}}(\boldsymbol{y}),\ \bar{\boldsymbol{u}}_l(\boldsymbol{y})=\bar{\boldsymbol{u}}(\boldsymbol{y}),\ \forall l=1,\ldots, L$ and $\rho=0$, then $\mathbb{P}\in \mathcal{M}(\underline{\boldsymbol{u}}_l(\boldsymbol{y}),\bar{\boldsymbol{u}}_l(\boldsymbol{y}))$ for all $l=1,\ldots,L$. As a result, $\mathbb{P}=\sum_{l=1}^L\hat{p}_l(\boldsymbol{y})\mathbb{P}\in\Theta_M(\boldsymbol{y})$.
\end{remark}

For deriving the value of multimodality under the Wasserstein distance-based ambiguity setting, we define the multimodal decision-dependent ambiguity set $\Theta_D(\boldsymbol{y})$ by considering the ambiguity set \eqref{eq:genericMultiModalAmbiguity} under the variation distance set \eqref{eq:modeDist-VariationDistance} for $\Delta(\hat{\boldsymbol{p}}(\boldsymbol{y}))$ and Wasserstein distance-based set \eqref{eq:distance-based} for $\mathcal{U}_l(\boldsymbol{y})$. 
For Wasserstein-based ambiguity set, one natural alternative of ambiguity set to the set $\Theta_D(\boldsymbol{y})$ 
is to group the empirical distributions under $L$ different modes to a single representative empirical distribution $\sum_{l=1}^L\hat{p}_l(\boldsymbol{y})\hat{\mathbb{P}}_l(\boldsymbol{y})$ and to find one worst-case distribution in the Wasserstein ambiguity set while considering the ambiguity in mode distributions. To this end, we first construct the following single-modal distance-based ${\rm D^3RO}$ as follows:
\begin{align}
   (\mbox{\bf{Single-Modal Distance-based }} {\rm D^3RO}): \min_{\boldsymbol{y}\in\mathcal{Y}} \boldsymbol{c}^{\mathsf T} \boldsymbol{y}+\max_{\mathbb{P} \in \Theta_D^{\prime}(\boldsymbol{y})}\mathbb{E}_{\boldsymbol{\xi}\sim \mathbb{P}}[h(\boldsymbol{y},\boldsymbol{\xi})] \label{model:DRO-singlemodalDistance}
\end{align}
where
\begin{align}
    \Theta_D^{\prime}(\boldsymbol{y})=\left\{\mathbb{P}\in\mathcal{P}(\Xi):\mathcal{W}_q\left(\mathbb{P}, \sum_{l=1}^L\hat{p}_l(\boldsymbol{y})\hat{\mathbb{P}}_l(\boldsymbol{y})\right)\le \sum_{l=1}^L \hat{p}_l(\boldsymbol{y})\epsilon_l + \rho\cdot\text{diam}({\color{black}\Xi})\right\}\label{eq:distance-based-grouped}
\end{align}
with $\text{diam}(\Xi):=\sup\{||x-y||_q: x,y\in\Xi\}$. 
{\color{black}Note that when $\rho=0$, the radius reduces to a convex combination of the radii from each mode, using the mode probability estimate $\hat{p}_l(\boldsymbol{y})$ as the weight. However, when we are not confident about the mode probability estimates $\hat{\boldsymbol{p}}(\boldsymbol{y})$ (i.e., $\rho\not=0$), we need to enlarge the radius to take into account the ambiguity around it, which is assumed to be proportional to $\rho$. We show that the multimodal distance-based ${\rm D^3RO}$ \eqref{model:DRO} under the ambiguity set $\Theta_D(\boldsymbol{y})$ always results in an in-sample cost that is at least as good as its single-modal counterpart \eqref{model:DRO-singlemodalDistance} in the next theorem.}


\begin{theorem}\label{thm:benefit-distance}
   If we use the Variation distance set \eqref{eq:modeDist-VariationDistance} as $\Delta(\hat{\boldsymbol{p}}(\boldsymbol{y}))$ {\color{black}and Wasserstein distance-based set \eqref{eq:distance-based} as $\mathcal{U}_l(\boldsymbol{y})$}, then the multimodal ambiguity set $\Theta_D(\boldsymbol{y})$ is nested in the single-modal set $\Theta_D^{\prime}(\boldsymbol{y})$, i.e., $\Theta_D(\boldsymbol{y})\subseteq\Theta_D^{\prime}(\boldsymbol{y})$. As a result, the optimal objective value of the multimodal moment-based ${\rm D^3RO}$ \eqref{model:DRO} is no more than the one of the single-modal counterpart \eqref{model:DRO-singlemodalDistance}, i.e.,
    \begin{align}
      \min_{\boldsymbol{y}\in\mathcal{Y}} \boldsymbol{c}^{\mathsf T} \boldsymbol{y}+\max_{\mathbb{P} \in \Theta_D(\boldsymbol{y})}\mathbb{E}_{\mathbb{P}_\xi}[h(\boldsymbol{y},\boldsymbol{\xi})]\le \min_{\boldsymbol{y}\in\mathcal{Y}} \boldsymbol{c}^{\mathsf T} \boldsymbol{y}+\max_{\mathbb{P} \in \Theta_D^{\prime}(\boldsymbol{y})}\mathbb{E}_{\mathbb{P}_\xi}[h(\boldsymbol{y},\boldsymbol{\xi})].\label{eq:MultimodalComparison-Distance}
    \end{align}
\end{theorem}
\begin{proof}
For any $\mathbb{P}=\sum_{l=1}^Lp_l\mathbb{P}_l\in\Theta_D(\boldsymbol{y})$, we have {\color{black}$\sum_{l=1}^Lp_l=1$}, $\sum_{l=1}^L|p_l-\hat{p}_l(\boldsymbol{y})|\le \rho$, $\mathbb{P}_l\in\mathcal{P}(\Xi_l)$ and $\mathcal{W}_q(\mathbb{P}_l, \hat{\mathbb{P}}_l(\boldsymbol{y}))\le \epsilon_l$. {\color{black}First of all, $\mathbb{P}\in\mathcal{P}(\Xi)$ because $\int_{\Xi}\mathbb{P}(d\boldsymbol{\xi})=\sum_{l=1}^Lp_l\int_{\Xi_l}\mathbb{P}_l(d\boldsymbol{\xi})=\sum_{l=1}^Lp_l=1$.
Let $\Xi^2=\Xi \times \Xi$ and }
$$\pi_l=\arg\inf \left\{\int_{\Xi_l^2}||\xi_1-\xi_2||_q\Pi(d\xi_1,d\xi_2): \substack{\Pi\text{ is a joint distribution of $\xi_1$ and $\xi_2$}\\ \text{with marginals $\mathbb{P}_l$ and $\hat{\mathbb{P}}_l(\boldsymbol{y})$}}\right\}.$$ 
Then $\sum_{l=1}^L\hat{p}_l(\boldsymbol{y})\pi_l\in\mathcal{P}(\Xi^2)$ has marginals $\sum_{l=1}^L\hat{p}_l(\boldsymbol{y})\mathbb{P}_l$ and $\sum_{l=1}^L\hat{p}_l(\boldsymbol{y})\hat{\mathbb{P}}_l(\boldsymbol{y})$. Thus,
\begin{align*}
    \mathcal{W}_q\left(\sum_{l=1}^L \hat{p}_l(\boldsymbol{y})\mathbb{P}_l, \sum_{l=1}^L\hat{p}_l(\boldsymbol{y})\hat{\mathbb{P}}_l(\boldsymbol{y})\right)\overset{(a)}{\le}&\int_{\Xi^2}||\xi_1-\xi_2||_q\left(\sum_{l=1}^L\hat{p}_l(\boldsymbol{y})\pi_l\right)(d\xi_1,d\xi_2)\\
    =&\sum_{l=1}^L \hat{p}_l(\boldsymbol{y})\int_{\Xi^2}||\xi_1-\xi_2||_q\pi_l(d\xi_1,d\xi_2)\\
    \overset{(b)}{=}&\sum_{l=1}^L \hat{p}_l(\boldsymbol{y})\mathcal{W}_q(\mathbb{P}_l, \hat{\mathbb{P}}_l(\boldsymbol{y}))\\
    \le &\sum_{l=1}^L \hat{p}_l(\boldsymbol{y})\epsilon_l,
\end{align*}
where $(a)$ is true because $\sum_{l=1}^L\hat{p}_l(\boldsymbol{y})\pi_l$ is a feasible transportation plan to move mass from $\sum_{l=1}^L\hat{p}_l(\boldsymbol{y})\mathbb{P}_l$ to $\sum_{l=1}^L\hat{p}_l(\boldsymbol{y})\hat{\mathbb{P}}_l(\boldsymbol{y})$;  $(b)$ is true because $\pi_l$ is supported on $\Xi_l^2$.
On the other hand, denoting $\text{TV}(\mathbb{P},\mathbb{Q})$ as the total variation distance between distribution $\mathbb{P}$ and $\mathbb{Q}$, we have the following relationship
\begin{align*}
    \mathcal{W}_q\left(\sum_{l=1}^L p_l\mathbb{P}_l, \sum_{l=1}^L\hat{p}_l(\boldsymbol{y})\mathbb{P}_l\right)
    \overset{(a)}{\le} & \text{diam}(\Xi)\text{TV}\left(\sum_{l=1}^L p_l\mathbb{P}_l, \sum_{l=1}^L\hat{p}_l(\boldsymbol{y})\mathbb{P}_l\right)\\
    = & \text{diam}(\Xi)\sup_{A\subset\Xi}\left|\int_{A}\sum_{l=1}^L p_l\mathbb{P}_l(d\xi)-\int_{A}\sum_{l=1}^L\hat{p}_l(\boldsymbol{y})\mathbb{P}_l(d\xi)\right|\\
    = & \text{diam}(\Xi)\sup_{A\subset\Xi}\left|\int_{A}\sum_{l=1}^L (p_l-\hat{p}_l(\boldsymbol{y}))\mathbb{P}_l(d\xi)\right|\\
    \le & \text{diam}(\Xi)\sup_{A\subset\Xi}\int_{A}\sum_{l=1}^L |p_l-\hat{p}_l(\boldsymbol{y})|\mathbb{P}_l(d\xi)\\
    = & \text{diam}(\Xi)\sup_{A\subset\Xi}\sum_{l=1}^L |p_l-\hat{p}_l(\boldsymbol{y})|\int_{A}\mathbb{P}_l(d\xi)\\
    \overset{(b)}{\le} & \rho\cdot\text{diam}(\Xi)
\end{align*}
where (a) is due to Theorem 4 in \cite{gibbs2002choosing} and $(b)$ is because the supremum is achieved at $A=\Xi$.
As a result, 
\begin{align*}
\mathcal{W}_q\left(\sum_{l=1}^L p_l\mathbb{P}_l,  \sum_{l=1}^L\hat{p}_l(\boldsymbol{y})\hat{\mathbb{P}}_l(\boldsymbol{y})\right)\le &\mathcal{W}_q\left(\sum_{l=1}^L p_l\mathbb{P}_l, \sum_{l=1}^L\hat{p}_l(\boldsymbol{y})\mathbb{P}_l\right)+\mathcal{W}_q\left(\sum_{l=1}^L \hat{p}_l(\boldsymbol{y})\mathbb{P}_l, \sum_{l=1}^L\hat{p}_l(\boldsymbol{y})\hat{\mathbb{P}}_l(\boldsymbol{y})\right)\\
\le &\sum_{l=1}^L \hat{p}_l(\boldsymbol{y})\epsilon_l+\rho\cdot\text{diam}(\Xi)
\end{align*}
which implies that $\mathbb{P}=\sum_{l=1}^Lp_l\mathbb{P}_l\in\Theta^{\prime}_D(\boldsymbol{y})$ and $\Theta_D(\boldsymbol{y})\subseteq\Theta_D^{\prime}(\boldsymbol{y})$. {\color{black}Thus, we have $\max_{\mathbb{P}\in \Theta_D^{\prime}(\boldsymbol{y})}\mathbb{E}_{\mathbb{P}}[h(\boldsymbol{y},\boldsymbol{\xi})]\ge \max_{\mathbb{P}\in \Theta_D(\boldsymbol{y})}\mathbb{E}_{\mathbb{P}}[h(\boldsymbol{y},\boldsymbol{\xi})]$.}
\end{proof}
\begin{remark}\label{remark:DRO-SAA}
    The equality in \eqref{eq:MultimodalComparison-Distance} holds when $\hat{\mathbb{P}}_l(\boldsymbol{y})=\hat{\mathbb{P}}(\boldsymbol{y}),\ \epsilon_l=\epsilon,\ \forall l=1,\ldots, L$ and $\rho=0$. Indeed, for any $\mathbb{P}\in \Theta_D^{\prime}(\boldsymbol{y})$, if $\hat{\mathbb{P}}_l(\boldsymbol{y})=\hat{\mathbb{P}}(\boldsymbol{y}),\ \epsilon_l=\epsilon,\ \forall l=1,\ldots, L$ and $\rho=0$, then $\mathcal{W}_q(\mathbb{P}, \hat{\mathbb{P}}_l(\boldsymbol{y}))=\mathcal{W}_q(\mathbb{P}, \sum_{l=1}^L\hat{p}_l(\boldsymbol{y})\hat{\mathbb{P}}_l(\boldsymbol{y}))\le \epsilon_l$ for all $l=1,\ldots, L$. As a result, $\mathbb{P}=\sum_{l=1}^L\hat{p}_l(\boldsymbol{y})\mathbb{P}\in \Theta_D(\boldsymbol{y})$.
\end{remark}

\section{Mode Probabilities}
\label{sec:mode}

The reformulations presented in Sections \ref{sec:moment} and \ref{sec:distance} are for generic problem settings that can lead to non-linear, non-convex optimization problems. In this section, we present tractable reformulations based on different forms of decision-dependence in the nominal mode probabilities.
Specifically, we discuss three possible ways to describe the decision-dependence in the nominal mode probabilities that can be applicable to various application settings and derive tractable reformulations under certain cases.
\subsection{Affine Dependence}
\label{sec:ModeAffineDependence}

We first consider the case when the mode probability has an affine dependence on the first-stage decision variable $\boldsymbol{y}$, i.e.,
\begin{equation} \label{eq:ModeAffineDependence}
\hat{{p}}_l(\boldsymbol{y})=\bar{p}_l+(\boldsymbol{\lambda}_l^p)^{\mathsf T}\boldsymbol{y},\ \forall l=1,\ldots, L,
\end{equation}
where $\bar{p}_l$ is the base probability associated with mode $l$, and parameter  $\lambda_{l,i}^p \in\mathbb{R}$ corresponds to the degree about how {\color{black}one unit change in $y_i$} may affect the probability distribution of mode $l$.
This mode function in \eqref{eq:ModeAffineDependence} represents the case when investments in certain technologies or activities can increase the probabilities of some modes while reducing the probabilities of the other modes. 
In this setting, we assume that $\sum_{l=1}^L\bar{p}_l=1$, $\sum_{l=1}^L\boldsymbol{\lambda}_l^p=0$ and $\hat{p}_l(\boldsymbol{y})\ge 0$ to ensure that the nominal mode probability $\hat{\boldsymbol{p}}(\boldsymbol{y})$ lies in a probability simplex for any possible $\boldsymbol{y}$. An alternative approach is to assume $\bar{p}_l=0,\ \forall l=1,\ldots, L,\ \sum_{l=1}^L\boldsymbol{\lambda}_l^p=1$ and $\sum_{i=1}^I y_i=1$. This reduces to the convex combination of distributions discussed in Section 4.1.2 of \cite{hellemo2018decision}.

When the first-stage decision $\boldsymbol{y}$ is continuous, the resulting formulations obtained in Section \ref{sec:MultimodalReformulations} are non-convex due to the existence of bilinear terms. In this case, one can use some off-the-shelf non-convex optimization solvers to solve the resulting problems directly. On the other hand, if the first-stage decision $\boldsymbol{y}$ is binary valued, we can use McCormick envelopes to exactly reformulate the bilinear terms as linear constraints. Hence by assuming the first-stage decisions $\boldsymbol{y}$ to be binary and using the affine function \eqref{eq:ModeAffineDependence} {\color{black}under finite support},
we provide MILP reformulations for two-stage multimodal ${\rm D^3RO}$ model \eqref{model:DRO} under variation distance based  $\Delta(\hat{p}(\boldsymbol{y}))$ set with moment-based and distance-based ambiguity sets in Appendix \ref{sec:AppendixSpecialCaseReformulations}. 
These reformulations are further leveraged in our computational study in {\color{black}Section \ref{sec:Computations-FacilityLocation} to illustrate our findings on a facility location problem with binary first-stage decision $\boldsymbol{y}$.} 
We note that using this analogy, MISOCP reformulations can be obtained under $\chi^2$-distance based $\Delta(\hat{p}(\boldsymbol{y}))$ set with moment-based and distance-based ambiguity sets. 

\subsection{Linear Scaling}\label{sec:linear-scaling}

Next, we consider the case where the first-stage decision variable can scale the mode probability linearly \cite{hellemo2018decision}. Let $\bar{p}_l$ be the nominal mode probability in each mode $l$ such that $\sum_{l=1}^L\bar{p}_l=1$ and for illustration, let $y\in \mathcal{Y} \subseteq \mathbb{R}_{+}$ be a one-dimensional decision variable. For certain modes $l\in\hat{L} \subset \{1,\cdots,L\}$, assume that variable $y$ scales the probability linearly, while the probability for the remaining modes are adjusted accordingly, i.e., 
\begin{align*}
    \hat{p}_l(y)=\begin{cases}
        \bar{p}_ly,\ \forall l\in\hat{L},\\
        \frac{1-y\sum_{l\in\hat{L}}\bar{p}_l}{\sum_{l\in[L]\setminus\hat{L}}\bar{p}_l}\bar{p}_l,\ \forall l\in [L]\setminus\hat{L}
    \end{cases}
\end{align*}

In this setting, we assume that $\bar{p}_ly \in [0,1]$ for every mode $l\in \hat{L}$ and $\frac{1-y\sum_{l\in\hat{L}}\bar{p}_l}{\sum_{l\in[L]\setminus\hat{L}}\bar{p}_l}\bar{p}_l \in [0,1]$ for every mode $l\in [L]\setminus\hat{L}$ while assuming $\hat{\boldsymbol{p}}(\boldsymbol{y})$ to lie in a probability simplex under every $y \in \mathcal{Y}$.
When the first-stage decision $y$ represents an investment decision, as in the example setting presented in Section \ref{sec:ModeAffineDependence}, then such decisions can scale the likelihood of certain modes, which can be captured through linear scaling with increased and decreased mode probabilities. 
{\color{black}We utilize this mode function in Section \ref{sec:Computations-Shipment} to illustrate our findings on a shipment planning problem with pricing, where the first-stage pricing decision $y$ is continuous.}


\subsection{Binary Interdiction}

Additionally, motivated by \cite{noyan2022distributionally}, we consider a setting to present mode probabilities, which can be relevant with various network reliability and interdiction related optimization problems including disaster planning problems. Under this setting, for a given network with $I$ links, the goal is to determine which links to reinforce to prevent random failures associated with these links. 
These decisions can be captured through binary valued $y_i$ decisions for each link $i \in [I]$, where each link $i$ has a baseline survival probability of $\sigma_i^0\in[0,1]$, which will be increased to $\sigma_i^1\in[\sigma_i^0,1]$ if the link is reinforced ($y_i=1$). 
Under each mode $l$, the failure state of the system can be described via a binary vector $\zeta_{l}$ of length $I$, whose $i^{th}$ component takes value of 1 if and only if link $i$ survives. Depending on the system state, there are $L=2^I$ different modes, where $\zeta_{li}=1$ if link $i$ survives under mode $l$, and 0 otherwise. 
Furthermore, under each mode $l$ describing system status, we may have different information on the uncertain parameter (e.g., demand) to be considered within the ambiguity sets associated with these modes. 
Then, we propose the following probability function for each mode $l$ dependent to the first-stage decisions $\boldsymbol{y}$ as follows:
\begin{align*}
\hat{p}_l(\boldsymbol{y})=\prod_{i\in I: \zeta_{li}=1}[(1-y_i)\sigma_i^0+y_i\sigma_i^1)]\prod_{i\in I: \zeta_{li}=0}[(1-y_i)(1-\sigma_i^0)+y_i(1-\sigma_i^1)],
\end{align*}
which involves multi-linear terms of the decision variables.  
To reformulate this mode probability function, 
following the distribution shaping technique developed in \cite{laumanns2014distribution}, we can rewrite $\hat{p}_l(\boldsymbol{y})$ as $\pi_{li}$ with the following linear constraints:
\begin{align*}
    &\pi_{li}\le \frac{\sigma_i^1}{\sigma_i^0}\pi_{l,i-1}+1-y_i,\ \forall i\in [I],\ l\in [L]: \zeta_{li}=1\\
    &\pi_{li}\le \frac{1-\sigma_i^1}{1-\sigma_i^0}\pi_{l,i-1}+1-y_i,\ \forall i\in [I],\ l\in [L]: \zeta_{li}=0\\
    &\pi_{li}\le \pi_{l,i-1}+y_i,\ \forall i\in [I],\ l\in [L]\\
    &\sum_{l=1}^L \pi_{li}=1,\ \forall i\in I\\
    & \boldsymbol{\pi}\in [0,1]^{I\times L}
\end{align*}
Since $\pi_{li}$ is a continuous decision variable, this will give rise to non-convex bilinear terms in the reformulations, which need to be solved via non-convex optimization solvers.

{\color{black}
\section{Solution Algorithm}
\label{sec:SolutionAlg}
In this section, we present a separation-based decomposition algorithm for solving a class of multimodal decision-dependent DRO reformulations proposed in this paper. We illustrate our algorithm over the variation distance based multimodal ambiguity with moment-based setting provided in formulation \eqref{model:variation+moment-general}, and consider the case where the lower and upper bounds on the moment functions $\underline{\boldsymbol{u}}_l(\boldsymbol{y}), \bar{\boldsymbol{u}}_l(\boldsymbol{y})$ and reference mode distribution $\hat{p}_l(\boldsymbol{y})$ are affinely dependent to $\boldsymbol{y}$. An example formulation for this setting can be found in Section \ref{sec:SpecialCases-MomentBased} under finite support assumption. Note that, our solution algorithm can handle continuous support sets, where we consider $\Xi_l$ being a polyhedral non-empty set for each mode $l = 1,\ldots,L$ in the remainder of this section. Furthermore, the algorithmic steps can be extended for solving alternative reformulations including the variation distance based multimodal ambiguity with Wasserstein-based setting provided in Section~\ref{sec:VariationDistanceWasserstein}, by designing a similar separation-based approach.  
For formulation \eqref{model:variation+moment-general}, we first observe that constraint \eqref{eq:MomentBasedVariation-SupportConstr} can be represented as follows:
\begin{equation} \label{eq:SeparationIneqaulity_v1}
\alpha_l \ge \max_{\xi \in \Xi_l} \left\{h(\boldsymbol{y},\boldsymbol{\xi}) - (\bar{\boldsymbol{\beta}}_l-\underline{\boldsymbol{\beta}}_l)^{\mathsf T}\boldsymbol{f}(\boldsymbol{\xi}) \right\},\ \forall l=1,\ldots, L. 
\end{equation}
Here, the maximization problem within this constraint involves a max-min problem that is not readily solvable in its current problem. To this end, by considering Assumption \ref{assump:complete-recourse}, we note that the inner problem $h(\boldsymbol{y},\boldsymbol{\xi})$ presented in \eqref{eq:second-stage} is a feasible linear program under every feasible first-stage decision $\boldsymbol{y} \in \mathcal{Y}$ and every realization of $\boldsymbol{\xi} \in \Xi_l$. Thus, the dual of $h(\boldsymbol{y},\boldsymbol{\xi})$ can be written as follows:
\begin{subequations}\label{eq:second-stage-dual}
\begin{align}
\max_{\boldsymbol{\omega}\in\mathbb{R}^P}\quad& (\boldsymbol{R}(\boldsymbol{y})-\boldsymbol{T}(\boldsymbol{y})\boldsymbol{\xi})^{\mathsf T} \boldsymbol{\omega} \\
\text{s.t.}\quad & W^{\mathsf T} \boldsymbol{\omega} = \boldsymbol{Q}\boldsymbol{\xi}+\boldsymbol{q} \label{eq:second-stage-dual-constr1} \\
& \boldsymbol{\omega} \geq 0.
\end{align}
\end{subequations}
Let $\bar{\Omega} = \{\boldsymbol{\omega}\in\mathbb{R}_{+}^P: \eqref{eq:second-stage-dual-constr1} \}$ be the feasible region of this dual problem \eqref{eq:second-stage-dual}. Then, the inequality \eqref{eq:SeparationIneqaulity_v1} can be written as 
\begin{equation} 
\label{eq:SeparationConstraint}
\alpha_l \ge \max_{\xi \in \Xi_l, \boldsymbol{\omega} \in \bar{\Omega}}  \left\{(\boldsymbol{R}(\boldsymbol{y})-\boldsymbol{T}(\boldsymbol{y})\boldsymbol{\xi})^{\mathsf T} \boldsymbol{\omega} - (\bar{\boldsymbol{\beta}}_l-\underline{\boldsymbol{\beta}}_l)^{\mathsf T}\boldsymbol{f}(\boldsymbol{\xi}) \right\},\ \forall l=1,\ldots, L. 
\end{equation}
We note that, under given $(\boldsymbol{y},\underline{\boldsymbol{\beta}}_l,\bar{\boldsymbol{\beta}}_l)$ values, the maximization problem in \eqref{eq:SeparationConstraint} can include bilinear terms in its objective function depending on the uncertainty in $h(\boldsymbol{y},\boldsymbol{\xi})$ and moment information $\boldsymbol{f}(\boldsymbol{\xi})$ utilized to construct the ambiguity set. 
In such cases, these bilinear terms can be linearized using various techniques, and alternative reformulations of this maximization problem can be obtained by the addition of new decision variables and constraints such as McCormick constraints, while keeping a polyhedron feasible region. These reformulations can be exact depending on the special characteristics of the underlying problem and support set. To this end, to represent the corresponding separation problem in its linear programming form, 
for notational brevity, we consider the case where $\boldsymbol{T}(\boldsymbol{y}) = 0$ while first-moment information is utilized in the construction of each ambiguity set, such as the sets including support, mean, mean absolute deviation information, that are widely utilized in practice. 
Thus, under given $(\boldsymbol{y},\underline{\boldsymbol{\beta}}_l,\bar{\boldsymbol{\beta}}_l)$ values, we refer to the maximization problem in \eqref{eq:SeparationConstraint} under this setting as $\Psi_l(\boldsymbol{y},\underline{\boldsymbol{\beta}}_l,\bar{\boldsymbol{\beta}}_l)$ and reformulate problem \eqref{model:variation+moment-general} as follows:    
\begin{subequations}\label{model:variation+moment-general-Algorithm}
    \begin{align}   \min_{\boldsymbol{y},\lambda,\eta,\alpha_l,\underline{\boldsymbol{\beta}}_l,\bar{\boldsymbol{\beta}}_l}
     \quad &\boldsymbol{c}^{\mathsf T}\boldsymbol{y}+\eta+\rho\lambda+\sum_{l=1}^L\hat{p}_l(\boldsymbol{y})r_l\\
        \text{s.t.}\quad & \eqref{model:variation+moment-general-ConstrFirst} - \eqref{model:variation+moment-general-ConstrEnd} 
        \notag \\
        & \alpha_l \geq \Psi_l(\boldsymbol{y},\underline{\boldsymbol{\beta}}_l,\bar{\boldsymbol{\beta}}_l) = \max_{\xi \in \Xi_l, \boldsymbol{\omega} \in \bar{\Omega}}\left\{\boldsymbol{R}(\boldsymbol{y})^{\mathsf T}\boldsymbol{\omega}- (\bar{\boldsymbol{\beta}}_l-\underline{\boldsymbol{\beta}}_l)^{\mathsf T}\boldsymbol{f}(\boldsymbol{\xi}) \right\}, \ \forall l=1,\ldots, L. \label{eq:SeparationConstraint-LinearObj}
    \end{align}
    \end{subequations}
This reformulation helps us to present our decomposition algorithm by leveraging constraint \eqref{eq:SeparationConstraint-LinearObj} as part of our separation subroutine. To this end, we first make the following observation. 
\begin{lemma} \label{lem:Finiteness}
Under a feasible solution $\boldsymbol{y} \in \mathcal{Y}$, 
function $\Psi_l(\boldsymbol{y},\underline{\boldsymbol{\beta}}_l,\bar{\boldsymbol{\beta}}_l)$ is a convex piecewise linear function in $\boldsymbol{y},\underline{\boldsymbol{\beta}}_l,\bar{\boldsymbol{\beta}}_l$ with a finite number of pieces for every mode $l = 1, \ldots, L$. 
\end{lemma}
\begin{proof}
We observe that with Assumption \ref{assump:complete-recourse} and $\Xi_l$ being a polyhedral non-empty set, the corresponding optimization problem in $\Psi_l(\boldsymbol{y},\underline{\boldsymbol{\beta}}_l,\bar{\boldsymbol{\beta}}_l)$ is feasible and bounded. 
Then, for any given feasible solution $(\boldsymbol{y},\underline{\boldsymbol{\beta}}_l,\bar{\boldsymbol{\beta}}_l)$, 
$\max_{\xi \in \Xi_l, \boldsymbol{\omega} \in \bar{\Omega}} \{\boldsymbol{R}(\boldsymbol{y})^{\mathsf T}\boldsymbol{\omega}- (\bar{\boldsymbol{\beta}}_l-\underline{\boldsymbol{\beta}}_l)^{\mathsf T}\boldsymbol{f}(\boldsymbol{\xi})\}$  
corresponds to the maximum of linear functions of $\boldsymbol{y},\underline{\boldsymbol{\beta}}_l,\bar{\boldsymbol{\beta}}_l$, resulting in a convex piecewise linear function. 
As each piece of this piecewise linear function can be associated with the extreme points of the polyhedral sets described by $\Xi_l \times \bar{\Omega}$, 
which has finitely many extreme points, it demonstrates the desired result. 
\end{proof}

Combining above, the solution algorithm is presented in Algorithm \ref{alg:solutionAlg}. The algorithm starts by solving a master problem, which is a relaxation of \eqref{model:variation+moment-general-Algorithm} by removing constraint \eqref{eq:SeparationConstraint-LinearObj}. Then, once the optimal solution of this master problem is obtained, that is, $(\boldsymbol{y}^*,\lambda^*,\eta^*,\alpha_l^*,\underline{\boldsymbol{\beta}}_l^*,\bar{\boldsymbol{\beta}}_l^*)$, the separation problem characterized by $\Psi_l(\boldsymbol{y}^*,\underline{\boldsymbol{\beta}}_l^*,\bar{\boldsymbol{\beta}}_l^*)$ is solved for each mode $l = 1, \dots, L$. If the relaxed inequality in \eqref{eq:SeparationConstraint-LinearObj} is violated, then the corresponding constraints are added to the set of cuts represented by $\{\mathcal{L}_l(\boldsymbol{y}, \alpha_l) \geq 0\}$ to be considered in the subsequent iteration. 
The algorithm stops when no violation is identified for the relaxed constraint corresponding to each mode. 
The algorithm further provides lower and upper bounds in each iteration, which are obtained using the optimal objective function value of the master problem and by solving a variant of the original problem that leverages the output of the separation problem, respectively. 
These bounds help in evaluating the quality of the obtained solutions and present an additional convergence criterion.    

\begin{algorithm}[h]
\caption{{\color{black}Decomposition-based Solution Algorithm}}\label{alg:solutionAlg}
\begin{algorithmic}[1]
\State Initialize $\{\mathcal{L}_l(\boldsymbol{y}, \alpha_l) \geq 0\} = \emptyset$, $UB= \infty$, $LB=-\infty$. 
\State Solve the master problem \[
Z^* = \min_{\boldsymbol{y},\lambda,\eta, r_l, \alpha_l,\underline{\boldsymbol{\beta}}_l,\bar{\boldsymbol{\beta}}_l} \left\{ \boldsymbol{c}^{\mathsf T}\boldsymbol{y}+\eta+\rho\lambda+\sum_{l=1}^L\hat{p}_l(\boldsymbol{y})r_l: \eqref{model:variation+moment-general-ConstrFirst} - \eqref{model:variation+moment-general-ConstrEnd}, \{\mathcal{L}_l(\boldsymbol{y}, \alpha_l) \geq 0\} \ \forall l=1,\ldots, L \right\},
\]
obtain the optimal solution $(\boldsymbol{y}^*,\lambda^*,\eta^*,\alpha_l^*,\underline{\boldsymbol{\beta}}_l^*,\bar{\boldsymbol{\beta}}_l^*)$,  and update $LB = Z^*$. \label{alg:MasterProblemStep}
\For{$l=1,\ldots, L$}
\State Solve $\Psi_l(\boldsymbol{y}^*,\underline{\boldsymbol{\beta}}_l^*,\bar{\boldsymbol{\beta}}_l^*)$ for mode $l$ and obtain its optimal value $\Psi_l^*$ and optimal solution $(\xi^*,\boldsymbol{\omega}^*)$. 
\If{$\alpha_l^* < \Psi_l^*$}
\State Add cut $\alpha_l \geq \boldsymbol{R}(\boldsymbol{y})^{\mathsf T}\boldsymbol{\omega}^*- (\bar{\boldsymbol{\beta}}_l-\underline{\boldsymbol{\beta}}_l)^{\mathsf T}\boldsymbol{f}(\boldsymbol{\xi^*})$ to the set $\{\mathcal{L}_l(\boldsymbol{y}, \alpha_l) \geq 0\}$.
\EndIf
\EndFor
\State Obtain $UB' = \min_{\lambda,\eta, r_l} \{\boldsymbol{c}^{\mathsf T}\boldsymbol{y^*} + \eta+\rho\lambda+\sum_{l=1}^L\hat{p}_l(\boldsymbol{y}^*)r_l: \eqref{model:variation+moment-general-ConstrSecond} - \eqref{model:variation+moment-general-ConstrEnd},\ \lambda \geq 0,\  \underline{\boldsymbol{\beta}}_l = \underline{\boldsymbol{\beta}}_l^*,\ \bar{\boldsymbol{\beta}}_l = \bar{\boldsymbol{\beta}}_l^*,\ \alpha_l = \Psi_l^*, \ \forall l=1,\ldots, L \}$, and update $UB=\min\{UB, UB'\}$.
\State If no cut is added for any of the modes $l=1,\ldots, L$, then optimal solution is $(\boldsymbol{y}^*,\lambda^*,\eta^*,\alpha_l^*,\underline{\boldsymbol{\beta}}_l^*,\bar{\boldsymbol{\beta}}_l^*)$, otherwise go to Step \ref{alg:MasterProblemStep}.  
\end{algorithmic}
\end{algorithm}


\begin{theorem}
Given that there is an oracle that can solve the master problem to optimality, 
Algorithm \ref{alg:solutionAlg} converges in finitely many iterations to the optimal solution of \eqref{model:variation+moment-general-Algorithm}. 
\end{theorem}
\begin{proof}
Using Lemma \ref{lem:Finiteness}, we can observe the finiteness of the algorithm. Since the algorithm stops when no violation is found, then the result follows. 
\end{proof}

We note that the master problem can be solved to optimality in many practical applications, such as the facility location problem studied in Section \ref{sec:Computations-FacilityLocation}, where the master problem has a MILP reformulation and the separation problem is a linear program as support and first moment information are utilized to construct the corresponding ambiguity set. Our results in Section~\ref{sec:CompStudy-RunTime} over a polyhedral support set (see Table \ref{tab:computational}) demonstrate convergence to the optimal solution 
along with significant computational speed-ups over various instances. We note that our algorithm is also valid when a discrete support set is considered for each mode $l=1,\cdots,L$, by solving the separation problem under each realization of the random variable $\boldsymbol{\xi}$, obtaining their optimum objective function values, selecting the maximum of these values as $\Psi_l(\boldsymbol{y}^*,\underline{\boldsymbol{\beta}}_l^*,\bar{\boldsymbol{\beta}}_l^*)$ and identifying the corresponding realization as $\boldsymbol{\xi}^*$ to be used in the subsequent cut generation process. 
}
{\color{black} We further note that as a potential future research direction, alternative decomposition algorithms can be extended to our setting by leveraging algorithms that solve general semi-infinite optimization problems, 
which can provide solutions with certain feasibility and optimality guarantees under specific assumptions 
\citep{Mehrotra2014_CuttingSurface, Luo2019_CuttingSurface}.} 

\section{Computational Results}
\label{sec:ComputationalStudy}

{\color{black}In our computational study, we provide an extensive set of results over two sample problem settings by analyzing our proposed modeling, reformulation, and solution techniques from various perspectives. Specifically, Section \ref{sec:Computations-FacilityLocation} illustrates our findings on a facility location problem with binary first-stage decisions, where we adopt the affine dependence discussed in Section \ref{sec:ModeAffineDependence} to model the mode probabilities. Section \ref{sec:Computations-Shipment} highlights our results on a shipment planning and pricing problem with a continuous first-stage pricing decision, where we utilize the linear scaling scheme discussed in Section \ref{sec:linear-scaling} to model the mode probabilities. 
}


\subsection{\color{black}Facility Location Problem}
\label{sec:Computations-FacilityLocation}
{\color{black}In this section, }
we consider a two-stage stochastic uncapacitated facility location problem, where the distribution of the random customer demand is multimodal and could be affected by our first-stage investment decisions. Specifically, we focus on the following model
\begin{align}
\min_{\boldsymbol{y}\in\mathcal{Y}\subseteq\{0,1\}^{I}} \boldsymbol{f}^{\mathsf T} \boldsymbol{y}+\max_{\mathbb{P} \in \Theta(\boldsymbol{y})}\mathbb{E}_{\boldsymbol{\xi}\sim \mathbb{P}}[h(\boldsymbol{y},\boldsymbol{\xi}(\boldsymbol{y}))] \label{model:facility-location}
\end{align}
Here $y_i=1$ if we open facility $i$ in the first stage and 0 otherwise. The investment cost is denoted by $f_i$ for all $i=1,\ldots, I$. 
For the second-stage problem, $h(\boldsymbol{y},\boldsymbol{\xi}(\boldsymbol{y}))$ measures the total resource-allocation cost minus the total revenue:
\begin{align*}
h(\boldsymbol{y},\boldsymbol{\xi}(\boldsymbol{y}))=\min\quad& \sum_{i=1}^I\sum_{j=1}^Jc_{ij}\xi_j(\boldsymbol{y})x_{ij}-\sum_{j=1}^Jr_j\xi_j(\boldsymbol{y})\sum_{i=1}^I x_{ij}+\sum_{j=1}^J p_js_j\xi_j(\boldsymbol{y})\\
    \text{s.t.}\quad & \sum_{i=1}^Ix_{ij}+s_j=1,\ \forall j=1,\ldots, J\\
    &x_{ij}\le y_i,\ \forall i=1,\ldots, I,\ j=1,\ldots, J\\
    &x_{ij},\ s_j\ge 0,\ \forall i=1,\ldots, I,\ j=1,\ldots, J,
\end{align*}
where $\xi_j\in\mathbb{R}_+$ denotes the uncertain customer demand at customer site $j$, $c_{ij},r_j, p_j$ are the unit transportation cost, unit revenue for satisfying demand, and unit penalty for unmet demand, respectively, while $x_{ij}\in\mathbb{R}_+$ denotes the fraction of the demand $\xi_j$ filled by facility $i$. 

In the remainder of this Section, we first provide the experimental setting in Section \ref{sec:CompStudy-ParameterSetup}. We then analyze the performances of the proposed decision-dependent multimodal models against their  single-modal and decision-independent counterparts under both moment-based and distance-based ambiguity sets in Section \ref{sec:CompStudy-Sensitivity}. We further discuss the impact of multimodality and decision-dependency under various parameter settings in Sections \ref{sec:CompStudy-EffectofMultimodality} and \ref{sec:CompStudy-EffectofDD}, respectively. We then analyze the performances of the proposed approaches under misspecified models with different out-of-sample scenarios in Section \ref{sec:CompStudy-MissecifiedModel}, and conclude our results with computational time comparisons of different approaches in Section \ref{sec:CompStudy-RunTime}.

\subsubsection{Parameter Setup}
\label{sec:CompStudy-ParameterSetup}
We randomly generate a set of $I=5$ potential facility locations and $J=10$ customer sites on a $100\times 100$ grid. At default, the unit investment cost $f_i$ is sampled uniformly between 1000 and 3000, the unit transportation cost $c_{ij}$ is calculated based on the Euclidean distance between facility $i$ and customer site $j$, the unit revenue $r_j$ is sampled uniformly between 50 and 100, and the unit penalty is $p_j=30$.

We consider $L=3$ modes, where only two of these modes are affected by our first-stage decisions. In the first mode, we assume that the product will become popular in the market and the demand will be greatly affected by our first-stage investment decisions. The ground truth demand model is assumed to be $\xi_{1,j}(\boldsymbol{y})=\bar{\mu}_{1,j}(1+0.5\sum_{i=1}^Ie^{-dist(i,j)/25}y_i)+\epsilon_{1,j}$ where $\epsilon_{1,j}\sim\mathcal{N}(0,\sigma_{1,j}^2)$ is the additive error. If we invest more in the first stage, mode 1 will be more likely to happen and as a result, we consider the mode probability to be $\hat{p}_1(\boldsymbol{y})=\bar{p}_1+0.01\sum_{i=1}^Iy_i$. The second mode assumes that the product fails to sell well with a much lower nominal demand mean $\bar{\mu}_{2,j}<\bar{\mu}_{1,j}$, but more investments in the first stage can still have some impact (lower than mode 1) on the demand. In this case, we assume the ground truth demand model to be $\xi_{2,j}(\boldsymbol{y})=\bar{\mu}_{2,j}(1+0.1\sum_{i=1}^Ie^{-dist(i,j)/25}y_i)+\epsilon_{2,j}$ where $\epsilon_{2,j}\sim\mathcal{N}(0,\sigma_{2,j}^2)$ is the additive error. As more investments in the first stage come in, mode 2 will be less likely to happen and as a result, we assume the mode probability to be $\hat{p}_2(\boldsymbol{y})=\bar{p}_2-0.01\sum_{i=1}^Iy_i$. The third mode represents a decision-independent situation, where the product demand is moderate ($\bar{\mu}_{2,j}<\bar{\mu}_{3,j}<\bar{\mu}_{1,j}$) and will not be affected by our investment decisions. In this case, the ground truth model is assumed to be $\xi_{3,j}(\boldsymbol{y})=\bar{\mu}_{3,j}+\epsilon_{3,j}$ where $\epsilon_{3,j}\sim\mathcal{N}(0,\sigma_{3,j}^2)$ is the additive error. Moreover, the mode probability is assumed to be $\hat{p}_3=\bar{p}_3=1-\bar{p}_1-\bar{p}_2$.

We first set $\bar{p}_1=0.5, \ \bar{p}_2=0.3$ and $\bar{p}_3=0.2$. For in-sample test, in mode 1, we sample the nominal demand $\bar{\mu}_{1,j}$ for each customer site $j$ following a Uniform distribution $\mathcal{U}(50, 100)$ and the nominal standard deviation of demand is set to $\bar{\sigma}_{1,j}=0.1\bar{\mu}_{1,j}$. For moment-based ambiguity sets, the empirical first moment is set to $\mu_{1,j}(\boldsymbol{y})=\bar{\mu}_{1,j}(1+0.5\sum_{i=1}^Ie^{-dist(i,j)/25}y_i)$, and $\epsilon^{\mu}_{1,j}=0$. The support size of demand values is taken as 200 with values in the range $\{1,\ldots,200\}$ at default. For distance-based ambiguity sets, the empirical uncertainty realizations are sampled from the true model, i.e.,  $\hat{\xi}_{1,k,j}(\boldsymbol{y})=\bar{\mu}_{1,j}(1+0.5\sum_{i=1}^Ie^{-dist(i,j)/25}y_i)+\hat{\epsilon}_{1,k,j}$ where $\hat{\epsilon}_{1,k,j}\sim\mathcal{N}(0,\sigma_{1,j}^2)$ are the empirical residuals. The support set for distance-based ambiguity set is $\Xi_l=\{0\le \xi\le 200\}$. We set $K_1=50,\ K_2=30,\ K_3=20$ in the distance-based ambiguity sets.
Accordingly, for modes 2 and 3, we set $\bar{\mu}_{2,j}=0.25\bar{\mu}_{1,j},\ \bar{\mu}_{3,j}=0.5\bar{\mu}_{1,j},\  \bar{\sigma}_{2,j}=0.1\bar{\mu}_{2,j},\ \bar{\sigma}_{3,j}=0.1\bar{\mu}_{3,j},\ \epsilon^{\mu}_{2,j}=\epsilon^{\mu}_{3,j}=0$ and define other parameters in a similar fashion. 


Given an optimal first-stage decision $\boldsymbol{y}^*$, we set the total out-of-sample scenarios to 1000 at default. Then we sample $1000\times (\bar{p}_1+0.01\sum_{i=1}^Iy^*_i)$ scenarios from mode 1 ($\xi_{1,j}(\boldsymbol{y}^*)=\bar{\mu}_{1,j}(1+0.5\sum_{i=1}^Ie^{-dist(i,j)/25}y^*_i)+\epsilon_{1,j}$), $1000\times (\bar{p}_1-0.01\sum_{i=1}^Iy^*_i)$ scenarios from mode 2 ($\xi_{2,j}(\boldsymbol{y}^*)=\bar{\mu}_{2,j}(1+0.1\sum_{i=1}^Ie^{-dist(i,j)/25}y^*_i)+\epsilon_{2,j}$), and $1000\times \bar{p}_3$ scenarios from mode 3 ($\xi_{3,j}(\boldsymbol{y})=\bar{\mu}_{3,j}+\epsilon_{3,j}$).

We use Gurobi 10.0.0 coded in Python 3.11.0 for solving all mixed-integer programming models, where the computational time limit is set to one hour. Our numerical tests are conducted on a Macbook Pro with 8 GB RAM and an Apple M1 Pro chip. 

\subsubsection{Sensitivity Analysis}
\label{sec:CompStudy-Sensitivity}

We first conduct sensitivity analysis on a special instance (shown in Figure \ref{fig:locations2}) for moment-based ambiguity sets in Section \ref{sec:results-moment} and distance-based ambiguity sets in Section \ref{sec:results-distance}, respectively. 
\begin{figure}[ht!]
\centering
\begin{tikzpicture}[scale=0.75]
\begin{axis}[legend style={nodes={scale=0.8, transform shape}}, 
legend pos= outer north east]
    \addplot[
        scatter/classes={        a={mark=triangle*, fill=red},
        b={mark=*,fill=blue}},
        scatter, mark=*, only marks, 
        scatter src=explicit symbolic,
        nodes near coords*={\Label},
        visualization depends on={value \thisrow{label} \as \Label} 
    ] table [meta=class] {
        x y class label
        42 73 a {}
        4 47 a {}
        91 83 a {}
      13 65 a {}
        47 59 a {}
        14 85 a {}
        93 40 a {}
        62 60 a {}
        26 80 a {}
        63 99 a {}
         33 52 b \#1
        10 72 b \#2
        90 5 b \#3
        7 18 b \#4
       24 87 b \#5
    };
    \legend{Customer sites, Potential facilities}
\end{axis}
\end{tikzpicture}
\caption{Locations of customer sites and potential facilities on a 100$\times$100 grid} 
\label{fig:locations2}
\end{figure}
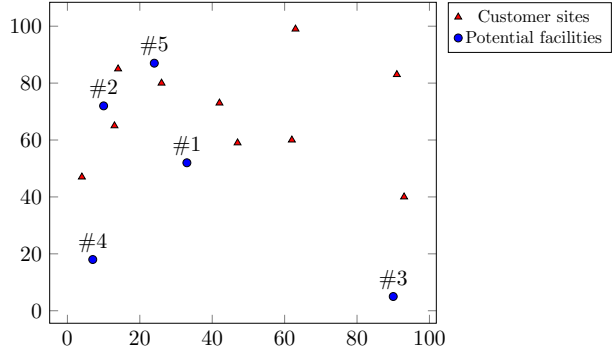

\paragraph{Moment-based Ambiguity Sets}\label{sec:results-moment}
We compare the solution pattern, in-sample (IS) and out-of-sample (OOS) cost returned by (i) Multi-Modal Moment-based $\rm{D^3RO}$ (MM-M-$\rm{D^3RO}$), (ii) Single-Modal Moment-based $\rm{D^3RO}$ (SM-M-$\rm{D^3RO}$) and (iii) Multi-Modal Moment-based DRO (MM-M-DRO) when we vary the first moment interval length $\epsilon^{\mu}/\bar{\mu}$ from 0 to 0.5 in Table \ref{tab:sensitivity-moment} and when we vary the support size $K$ in Table \ref{tab:sensitivity-moment-K}, respectively. 
We note that in comparison to our proposed model MM-M-$\rm{D^3RO}$, SM-M-$\rm{D^3RO}$ focuses on a single-modal setting with decision-dependent uncertainties as illustrated in Section \ref{sec:MultimodalAnalyticalResults}, whereas MM-M-DRO focuses on a multimodal setting without the decision-dependencies. 

\begin{table}[ht!]
  \centering
  \caption{Solution Comparison of Moment-based Ambiguity Sets with Varying $\epsilon^{\mu}/\bar{\mu}$}
  \resizebox{\textwidth}{!}{
    \begin{tabular}{l|rrr|rrr|rrr}
    \hline
    & \multicolumn{3}{c|}{MM-M-$\rm{D^3RO}$} & \multicolumn{3}{c|}{SM-M-$\rm{D^3RO}$} & \multicolumn{3}{c}{MM-M-DRO} \\
    $\epsilon^{\mu}/\bar{\mu}$& Solution & IS Cost& OOS Cost& Solution & IS Cost& OOS Cost& Solution & IS Cost& OOS Cost\\
    \hline
    0 & \textbf{[1,2,5]} & -2079 & \textbf{-3257} & [1,2,3,5] & 6932  & -2455 & [1,2] & 947   & -1465 \\
    0.1 & \textbf{[1,2,5]} & -603  & \textbf{-3257}& [1,2,3,5] & 8131  & -2455 & [1,2] & 2236  & -1465 \\
     0.2  & \textbf{[1,2,5]} & 875   & \textbf{-3257}& [1,2,3,5] & 9330  & -2455 & [1,2] & 3604  & -1465 \\
     0.3  & \textbf{[1,2,5]} & 2411  & \textbf{-3257} & [1,2,3] & 10517 & -584  & [1,2] & 5220  & -1465 \\
   0.4  & \textbf{[1,2,5]} & 3947  & \textbf{-3257} & [1,2,3] & 11630 & -584  & [1,2] & 6837  & -1465 \\
    0.5  & \textbf{[1,2,5]} & 5484  & \textbf{-3257} & [1,2,3] & 12742 & -584  & [1,2,3] & 8352  & -584 \\
    \hline
    \end{tabular}%
    }
  \label{tab:sensitivity-moment}%
\end{table}%

From Table \ref{tab:sensitivity-moment}, when we increase the interval length, all models return a higher in-sample cost, because we become more conservative. In terms of out-of-sample cost, MM-M-$\rm{D^3RO}$ is the most stable model, always producing the same first-stage decision and out-of-sample cost, while SM-M-$\rm{D^3RO}$ and MM-M-DRO produce worse solutions and out-of-sample costs when we increase the interval length. Under all settings, MM-M-$\rm{D^3RO}$ outperforms the other two, by choosing facilities \#1, \#2, and \#5 that are closer to customer sites. In our following tests, we fix $\epsilon^{\mu}=0$.

\begin{table}[ht!]
  \centering
  \caption{Solution Comparison of Moment-based Ambiguity Sets with Varying Support Size $K$}
  \resizebox{\textwidth}{!}{
    \begin{tabular}{l|rrr|rrr|rrr}
    \hline
    & \multicolumn{3}{c|}{MM-M-$\rm{D^3RO}$} & \multicolumn{3}{c|}{SM-M-$\rm{D^3RO}$} & \multicolumn{3}{c}{MM-M-DRO} \\
    $K$& Solution & IS Cost& OOS Cost& Solution & IS Cost& OOS Cost& Solution & IS Cost& OOS Cost\\
    \hline
    100 & \multicolumn{3}{c|}{{unbounded}} & [1,2,3,5] & 6559  & -2455 & [1,2] & 947   & -1465 \\
    200 & \textbf{[1,2,5]} & -2079 & \textbf{-3257} & [1,2,3,5] & 6932  & -2455 & [1,2] & 947   & -1465 \\
    300 & \textbf{[1,2,5]} & -2079 & \textbf{-3257} & [1,2,3,5] & 6932  & -2455 & [1,2] & 947   & -1465 \\
    400 & \textbf{[1,2,5]} & -2079 & \textbf{-3257} & [1,2,3,5] & 6932  & -2455 & [1,2] & 947   & -1465 \\
    500 & \textbf{[1,2,5]} & -2079 & \textbf{-3257} & [1,2,3,5] & 6932  & -2455 & [1,2] & 947   & -1465 \\
    \hline
    \end{tabular}%
    }
  \label{tab:sensitivity-moment-K}%
\end{table}%
From Table \ref{tab:sensitivity-moment-K}, when $K=100$, MM-M-$\rm{D^3RO}$ becomes unbounded, because the inner maximization problem in Model \eqref{model:DRO} is infeasible due to insufficient data points in the support set. Other than this setting, the three models are all insensitive to the support size $K$. Because of this, we fix $K=200$ in our following tests.

\paragraph{Distance-based Ambiguity Sets}\label{sec:results-distance}
We compare the solution pattern, in-sample (IS) and out-of-sample (OOS) cost returned by (i) Multi-Modal Distance-based $\rm{D^3RO}$ (MM-D-$\rm{D^3RO}$), (ii) Single-Modal Distance-based $\rm{D^3RO}$ (SM-D-$\rm{D^3RO}$) and (iii) Multi-Modal Distance-based DRO (MM-D-DRO) when we vary the radius $\epsilon$ in Table \ref{tab:sensitivity-distance}. Note that when $\epsilon=[0, 0, 0]$, the MM-D-$\rm{D^3RO}$ reduces to the MM-DD-SP under mode ambiguity. We compare this setting with the other two constant radii (when $\epsilon=[0.2, 0.2, 0.2]$ and $\epsilon=[10, 10, 10]$) as well as the setting when the radius is inversely proportional to the number of data points in each mode ($\epsilon=[0.2, 0.33, 0.5]$) to reflect different confidence for the estimation in each mode. 

\begin{table}[ht!]
  \centering
  \caption{Solution Comparison of Distance-based Ambiguity Sets with Varying Radius $\epsilon$}
  \resizebox{\textwidth}{!}{
    \begin{tabular}{l|rrr|rrr|rrr}
    \hline
    & \multicolumn{3}{c|}{MM-D-$\rm{D^3RO}$} & \multicolumn{3}{c|}{SM-D-$\rm{D^3RO}$} & \multicolumn{3}{c}{MM-D-DRO} \\
    $\epsilon$& Solution & IS Cost& OOS Cost& Solution & IS Cost& OOS Cost& Solution & IS Cost& OOS Cost\\
    \hline
   [0, 0, 0] & \textbf{[1,2,5]} & -1986 & \textbf{-3257} & [1,2,3,5] & 10036 & -2455 & [1,2] & 1015 & -1465 \\
   $[0.2, 0.2, 0.2]$ & \textbf{[1,2,5]} & -1986 & \textbf{-3257} & [1,2,3,5] & 10036 & -2455 & [1,2] & 1015 & -1465 \\
    $[0.2, 0.33, 0.5]$ & \textbf{[1,2,5]} & -1950 & \textbf{-3257} & [1,2,3,5] & 10065 & -2455 & [1,2] & 1049 & -1465 \\
     $[10, 10, 10]$ & \textbf{[1,2,5]} & 1306 & \textbf{-3257} & [1,2,3] & 12909 & -584 & [1,2] & 3936 & -1465 \\
     \hline
    \end{tabular}%
    }
  \label{tab:sensitivity-distance}%
\end{table}%

Comparing these three models, MM-D-$\rm{D^3RO}$ and MM-D-DRO are insensitive to the radius, while SM-D-$\rm{D^3RO}$ returns a different decision when we increase the radius. On the other hand,  MM-D-$\rm{D^3RO}$ always produces the same first-stage decision as  MM-M-$\rm{D^3RO}$, which obtains a better out-of-sample cost than the other two benchmarks. 
We note that the results obtained in Sections \ref{sec:results-moment} and \ref{sec:results-distance} for moment-based and distance-based ambiguity sets are further in line with our analytical results derived in Section \ref{sec:MultimodalAnalyticalResults} with better in-sample results and overall performances of the multimodal decision-dependent model MM-D-$\rm{D^3RO}$ against its single-modal counterpart model SM-D-$\rm{D^3RO}$. In the following tests, we fix the radius $\epsilon=[0.2, 0.33, 0.5]$.

We also report the sensitivity results of distance-based ambiguity sets when we vary the support set size $K$ from 100 to 500 in Table \ref{tab:sensitivity-distance-K}.

\begin{table}[ht!]
  \centering
  \caption{Solution Comparison of Distance-based Ambiguity Sets with Varying Support Size $K$}
  \resizebox{\textwidth}{!}{
    \begin{tabular}{l|rrr|rrr|rrr}
    \hline
    & \multicolumn{3}{c|}{MM-D-$\rm{D^3RO}$} & \multicolumn{3}{c|}{SM-D-$\rm{D^3RO}$} & \multicolumn{3}{c}{MM-D-DRO} \\
    $K$& Solution & IS Cost& OOS Cost& Solution & IS Cost& OOS Cost& Solution & IS Cost& OOS Cost\\
    \hline
    100 & \multicolumn{3}{c|}{{unbounded}} & \multicolumn{3}{c|}{unbounded} & [1,2] & 1020  & -1465 \\
    200 & \textbf{[1,2,5]} & -1950 & \textbf{-3257} & [1,2,3,5] & 10065 & -2455 & [1,2] & 1049  & -1465 \\
    300 & \textbf{[1,2,5]} & -1950 & \textbf{-3257} & [1,2,3] & 14949 & -584  & [1,2] & 1049  & -1465 \\
    400 & \textbf{[1,2,5]} & -1950 & \textbf{-3257} & [1,2,3] & {17815} & -584  & [1,2] & 1049  & -1465 \\
    500 & \textbf{[1,2,5]} & -1950 & \textbf{-3257} & [1,2,3] & 19575 & -584  & [1,2] & 1049  & -1465 \\
    \hline
    \end{tabular}%
    }
  \label{tab:sensitivity-distance-K}%
\end{table}%
From Table \ref{tab:sensitivity-distance-K}, when $K=100$, both MM-D-$\rm{D^3RO}$ and SM-D-$\rm{D^3RO}$ become unbounded due to an infeasible inner maximization problem. 
As we increase the support set size $K$, the gap between MM-D-$\rm{D^3RO}$ and SM-D-$\rm{D^3RO}$ increases, which is in line with our analytical results in Section \ref{sec:MultimodalAnalyticalResults}. Moreover, MM-D-DRO produces stable results when we vary the support set size $K$.


\subsubsection{Effect of Multimodality}
\label{sec:CompStudy-EffectofMultimodality}

In this section, we compare the multimodal $\rm{D^3RO}$ model with the single-modal $\rm{D^3RO}$ model under moment-based and distance-based ambiguity sets under different robustness levels and support sizes, respectively. More specifically, we vary the robustness level $\rho$ from 0 to 0.5 and the support size $K$ from 200 to 500 and display the average in-sample and out-of-sample costs over 10 independent runs in Figure \ref{fig:multimodality}. From Figure \ref{fig:multimodality}(a), the in-sample costs of MM-M-${\rm D^3RO}$ and MM-D-${\rm D^3RO}$ almost coincide, and the gap between multimodal ${\rm D^3RO}$ and their single-modal counterparts increases as we increase the robustness level $\rho$. In terms of out-of-sample cost, multimodal ${\rm D^3RO}$ produces more stable solutions under both moment-based and distance-based ambiguity sets, while single-modal counterparts' performances become worse when $\rho$ increases. When we change the support size $K$, all four models return the same in-sample and out-of-sample costs, except that SM-D-${\rm D^3RO}$ generates higher costs when $K$ increases. The gap between MM-D-${\rm D^3RO}$ and its single-modal counterpart increases as we increase the support size $K$, which agrees with our analytical results in Section \ref{sec:MultimodalAnalyticalResults}.
\begin{figure}[ht!]
    \centering
    \begin{subfigure}{0.45\textwidth}
 \resizebox{\textwidth}{!}{%
 \pgfplotsset{scaled y ticks=true}
\begin{tikzpicture}
  \begin{axis}
  [
    xlabel={Robustness Level $\rho$},
    ylabel={Average In-Sample Cost},
    xtick={0, 0.1, 0.2, 0.3, 0.4, 0.5},
    scaled y ticks = true,
cycle list name=black white,
   legend pos= outer north east
]
    \addplot coordinates {
(0,	-6203.549603)
(0.1, -5354)
(0.2, -4507.172229)
(0.3, -3666.781821)
(0.4, -2858.924107)
(0.5, -2063.480441)
    };\pgfplotsset{cycle list shift=4}
    \addplot coordinates {
(0,	-6203.549608)
(0.1, -1126)
(0.2,3849.550397)
(0.3, 8545.804414)
(0.4, 13023.28598)
(0.5, 14542.35361)
    };\pgfplotsset{cycle list shift=2}
                \addplot coordinates {
(0,	-6128.111024)
(0.1, -5277)
(0.2,-4429.570525)
(0.3, -3588.21012)
(0.4, -2781.520148)
(0.5, -1985.122634)
    };\pgfplotsset{cycle list shift=5}
                    \addplot coordinates {
(0,	-6128.111024)
(0.1, 257)
(0.2,6527.220097)
(0.3, 11820.11959)
(0.4, 14990.15817)
(0.5, 16644.64346)
    };
    \legend{MM-M-${\rm D^3RO}$, SM-M-${\rm D^3RO}$, MM-D-${\rm D^3RO}$, SM-D-${\rm D^3RO}$}
  \end{axis}
\end{tikzpicture}%
}
 \caption{In-sample cost w.r.t. robustness level $\rho$}
\end{subfigure}
\begin{subfigure}{0.45\textwidth}
\resizebox{\textwidth}{!}{%
\pgfplotsset{scaled y ticks=true}
\begin{tikzpicture}
  \begin{axis}
  [
    xlabel={Robustness Level $\rho$},
    ylabel={Average Out-of-Sample Cost},
    xtick={0, 0.1, 0.2, 0.3, 0.4, 0.5},
    scaled y ticks = true,
    cycle list name=black white,
     legend pos= outer north east
    ]
    \addplot coordinates {
(0,	-6206.25289)
(0.1, -6206)
(0.2, -6191)
(0.3, -6146)
(0.4, -6084)
(0.5, -6043.610372)
    };\pgfplotsset{cycle list shift=4}
    \addplot coordinates {
(0,	-6206.25289)
(0.1,-6129)
(0.2, -5856)
(0.3, -5130)
(0.4,-4824)
(0.5, -4120)
    };\pgfplotsset{cycle list shift=2}
        \addplot coordinates {
(0,	-6206.25289)
(0.1, -6206)
(0.2, -6191)
(0.3, -6191)
(0.4, -6044)
(0.5, -6043.610372)
    };        \pgfplotsset{cycle list shift=5}
    \addplot coordinates {
(0,	-6206.25289)
(0.1,-6206)
(0.2, -5970)
(0.3, -5161)
(0.4,-4697)
(0.5, -4823.634921)
    };
    \legend{MM-M-${\rm D^3RO}$, SM-M-${\rm D^3RO}$, MM-D-${\rm D^3RO}$, SM-D-${\rm D^3RO}$}
  \end{axis}
\end{tikzpicture}%
}
\caption{Out-of-sample cost w.r.t. robustness level $\rho$}
\end{subfigure}
\begin{subfigure}{0.45\textwidth}
 \resizebox{\textwidth}{!}{%
 \pgfplotsset{scaled y ticks=true}
\begin{tikzpicture}
    \begin{axis}
  [
        xlabel={Support size $K$},
    ylabel={Average In-Sample Cost},
    xtick={200, 300, 400, 500},
    scaled y ticks = true,
    cycle list name=black white,
    legend pos= outer north east
    ]
    \addplot coordinates {
(200, -4507.172229)
(300, -4507.172229)
(400, -4507.172229)
(500, -4507.172229)
    };\pgfplotsset{cycle list shift=4}
    \addplot coordinates {
(200, 3849.550397)
(300, 3849.550397)
(400,3849.550397)
(500, 3849.550397)
    };\pgfplotsset{cycle list shift=2}
        \addplot coordinates {
(200, -4429.570525)
(300, -4429.570525)
(400, -4429.570525)
(500, -4429.570525)
    };\pgfplotsset{cycle list shift=5}
            \addplot coordinates {
(200, 6527.220097)
(300, 11832.30017)
(400, 15225.92016)
(500, 17185.97332)
    };
    \legend{MM-M-${\rm D^3RO}$, SM-M-${\rm D^3RO}$, MM-D-${\rm D^3RO}$, SM-D-${\rm D^3RO}$}
  \end{axis}
\end{tikzpicture}%
}
\caption{In-sample cost w.r.t. support size $K$}
\end{subfigure}
\begin{subfigure}{0.45\textwidth}
 \resizebox{\textwidth}{!}{%
 \pgfplotsset{scaled y ticks=true}
\begin{tikzpicture}
    \begin{axis}
  [
        xlabel={Support size $K$},
    ylabel={Average Out-of-Sample Cost},
    xtick={200, 300, 400, 500},
    scaled y ticks = true,
    cycle list name=black white,
    legend pos= outer north east
    ]
    \addplot coordinates {
(200, -6191)
(300, -6191)
(400, -6191)
(500, -6191)
    };\pgfplotsset{cycle list shift=4}
    \addplot coordinates {
(200, -5856)
(300, -5856)
(400, -5856)
(500, -5856)
    };\pgfplotsset{cycle list shift=2}
        \addplot coordinates {
(200, -6191)
(300, -6191)
(400, -6191)
(500, -6191)
    };\pgfplotsset{cycle list shift=5}
            \addplot coordinates {
(200, -5970)
(300, -5161)
(400, -4697)
(500,-4824)
    };
    \legend{MM-M-${\rm D^3RO}$, SM-M-${\rm D^3RO}$, MM-D-${\rm D^3RO}$, SM-D-${\rm D^3RO}$}
  \end{axis}
\end{tikzpicture}%
}
\caption{Out-of-sample cost w.r.t. support size $K$}
\end{subfigure}
    \caption{In-sample and out-of-sample cost comparison between multimodal $\rm{D^3RO}$ model and its single-modal counterpart with different robustness level $\rho$ and support size $K$.}
    \label{fig:multimodality}
\end{figure}
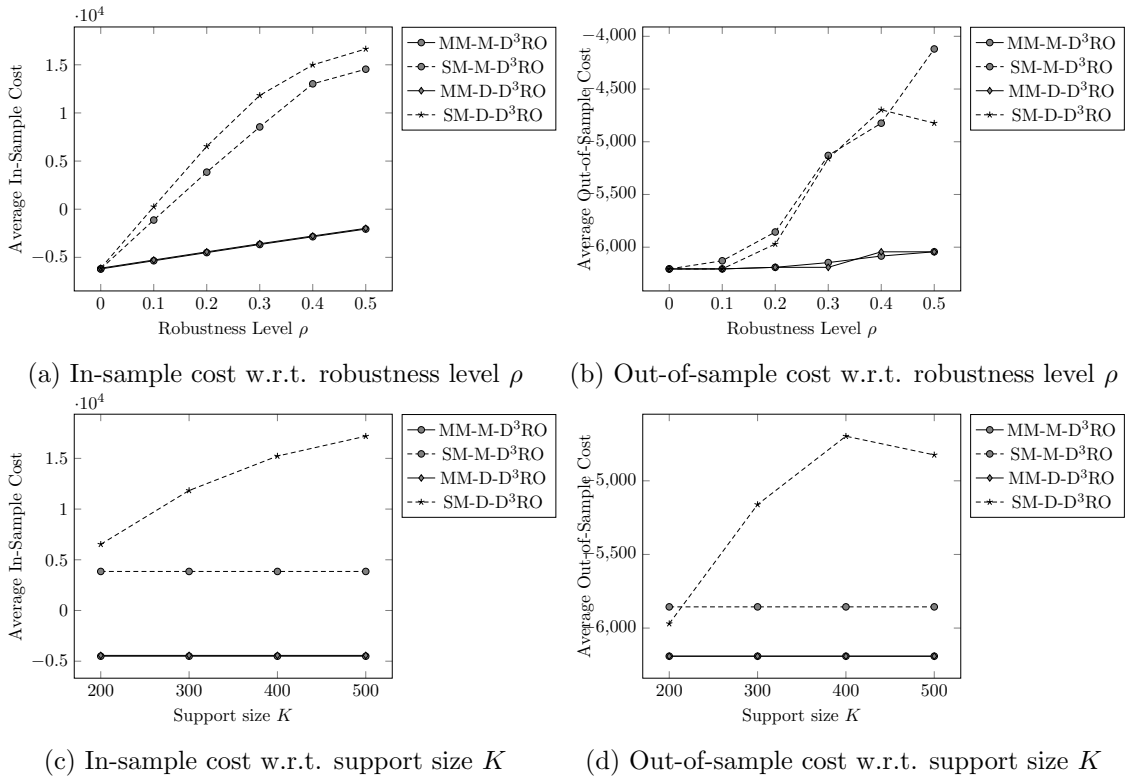

\subsubsection{Effect of Decision-Dependency}
\label{sec:CompStudy-EffectofDD}

In this section, we compare the multimodal $\rm{D^3RO}$ model with the multimodal decision-independent DRO model under moment-based and distance-based ambiguity sets, respectively. Specifically, we vary the robustness level $\rho$ from 0 to 0.5 and display the average in-sample and out-of-sample costs over 10 independent runs in Figure \ref{fig:decision-dependency}. From Figure \ref{fig:decision-dependency}, multimodal $\rm{D^3RO}$ models always generate lower in-sample and out-of-sample costs than the decision-independent counterparts, showing the benefit of considering decision-dependency. On the other hand, as we increase the robustness level $\rho$, all models produce worse in-sample and out-of-sample costs as we become more conservative.
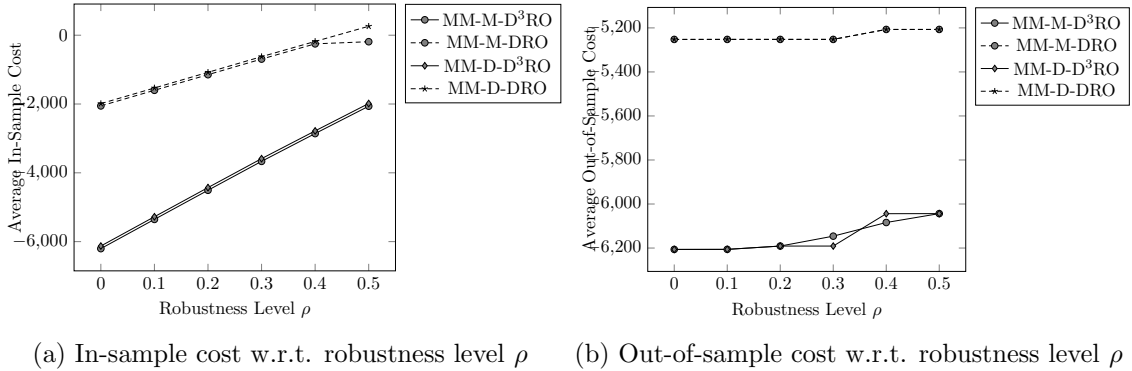
\begin{figure}[ht!]
    \centering
    \begin{subfigure}{0.45\textwidth}
 \resizebox{\textwidth}{!}{%
\begin{tikzpicture}
  \begin{axis}
  [
    xlabel={Robustness Level $\rho$},
    ylabel={Average In-Sample Cost},
    xtick={0, 0.1, 0.2, 0.3, 0.4, 0.5},
cycle list name=black white,
   legend pos= outer north east
]
    \addplot coordinates {
(0,	-6203.549603)
(0.1, -5354)
(0.2, -4507.172229)
(0.3, -3666.781821)
(0.4, -2858.924107)
(0.5, -2063.480441)
    };\pgfplotsset{cycle list shift=4}
    \addplot coordinates {
(0,	-2054.87)
(0.1, -1601)
(0.2,-1147.36)
(0.3, -693.605)
(0.4, -250.89)
(0.5, -188.985)
    };\pgfplotsset{cycle list shift=2}
            \addplot coordinates {
(0,	-6128.111024)
(0.1, -5277)
(0.2,-4429.570525)
(0.3, -3588.21012)
(0.4, -2781.520148)
(0.5, -1985.122634)
    };\pgfplotsset{cycle list shift=5}
                \addplot coordinates {
(0,	-1988.967779)
(0.1, -1534)
(0.2, -1078.599884)
(0.3, -623.4159372)
(0.4, -181.3136727)
(0.5, 259.9326129)
    };
    \legend{MM-M-${\rm D^3RO}$, MM-M-DRO, MM-D-${\rm D^3RO}$, MM-D-DRO}
  \end{axis}
\end{tikzpicture}%
}
 \caption{In-sample cost w.r.t. robustness level $\rho$}
\end{subfigure}
\begin{subfigure}{0.45\textwidth}
\resizebox{\textwidth}{!}{%
 \pgfplotsset{scaled y ticks=true}
\begin{tikzpicture}
  \begin{axis}
  [
    xlabel={Robustness Level $\rho$},
    ylabel={Average Out-of-Sample Cost},
    xtick={0, 0.1, 0.2, 0.3, 0.4, 0.5},
    cycle list name=black white,
     legend pos= outer north east
    ]
    \addplot coordinates {
(0,	-6206.25289)
(0.1, -6206)
(0.2, -6191)
(0.3, -6146)
(0.4, -6084)
(0.5, -6043.610372)
    };\pgfplotsset{cycle list shift=4}
    \addplot coordinates {
(0,	-5252.306732)
(0.1,-5252.306732)
(0.2, -5252.306732)
(0.3, -5252.306732)
(0.4,-5207)
(0.5, -5207)
    };\pgfplotsset{cycle list shift=2}
        \addplot coordinates {
(0,	-6206.25289)
(0.1, -6206)
(0.2, -6191)
(0.3, -6191)
(0.4, -6044)
(0.5, -6043.610372)
    };        \pgfplotsset{cycle list shift=5}
    \addplot coordinates {
(0,	-5252.306732)
(0.1,-5252)
(0.2, -5252)
(0.3, -5252)
(0.4,-5207)
(0.5, -5207)
    };
    \legend{MM-M-${\rm D^3RO}$, MM-M-DRO, MM-D-${\rm D^3RO}$, MM-D-DRO}
  \end{axis}
\end{tikzpicture}%
}
\caption{Out-of-sample cost w.r.t. robustness level $\rho$}
\end{subfigure}
    \caption{In-sample and out-of-sample cost comparison between multimodal $\rm{D^3RO}$ model and its decision-independent counterpart with different robustness level $\rho$.} 
    \label{fig:decision-dependency}
\end{figure}


\subsubsection{Effect of Misspecified Model}
\label{sec:CompStudy-MissecifiedModel}

In this section, we evaluate the impact when we have a misspecified model. Specifically, we consider the case when the out-of-sample scenarios are generated from a distribution different than the in-sample scenarios in Section \ref{sec:distribution-shift} and the case when the mode probabilities are misspecified in Section \ref{sec:misspecified-mode}

\paragraph{Distribution Shift}\label{sec:distribution-shift}
We compare the average out-of-sample costs over 10 independent runs when the out-of-sample scenarios are generated from a normal distribution with skewness of 0 (same distribution as the in-sample scenarios), skewness of 10 (right-skewed), skewness of -10 (left-skewed), or the mean shifted to the right by 10 (mean shift) in Table \ref{tab:shift} under both moment-based and distance-based ambiguity sets. From Table \ref{tab:shift}, as we increase the demand mean in the out-of-sample test or right skew the distribution, the out-of-sample costs all become smaller compared to the one with a well-specified model. 
Additionally, in all of the cases, the proposed model MM-M-$\rm{D^3RO}$ performs better and more stable compared to the other approaches with smaller out-of-sample cost values. 
\begin{table}[ht!]
  \centering
  \caption{Out-of-sample cost comparison with possible distribution shifts for generating out-of-sample scenarios.}
    \begin{tabular}{lrrrr}
    \hline
    Model & \multicolumn{1}{l}{skewness = 0} & \multicolumn{1}{l}{skewness = 10} & \multicolumn{1}{l}{skewness = -10} &  \multicolumn{1}{l}{mean shift}\\
    \hline
    MM-M-$\rm{D^3RO}$ & \textbf{-6191} & \textbf{-6864} & \textbf{-5512} & \textbf{-7780} \\
    SM-M-$\rm{D^3RO}$& -5856 & -6505 & -5197 & -7403 \\
    MM-M-DRO & -5252 & -5837 & -4657 & -6642 \\
        \hline
    MM-D-$\rm{D^3RO}$  & \textbf{-6191} & \textbf{-6864} & \textbf{-5512} & \textbf{-7780} \\
    SM-D-$\rm{D^3RO}$ & -5970 & -6659 & -5274 & -7598 \\
    MM-D-DRO & -5252 & -5837 & -4657 & -6642 \\
    \hline
    \end{tabular}%
  \label{tab:shift}%
\end{table}%

\paragraph{Misspecified Mode Probabilities}\label{sec:misspecified-mode}
To evaluate the performances of the proposed approaches under misspecified mode probabilities, we compare the average out-of-sample costs over 10 independent runs in Table \ref{tab:misspecified-mode} under both moment-based and distance-based ambiguity sets. Specifically, we assume that the true mode probability is $p=[\hat{p}_1(\boldsymbol{y})+\Delta, \hat{p}_2(\boldsymbol{y})-\Delta, \hat{p}_3(\boldsymbol{y})]$. Our results are in line with the findings in Section \ref{sec:distribution-shift} with better performance of the MM-M-$\rm{D^3RO}$ model, demonstrating the stability of the approach under various distributional ambiguities over mode probabilities. 
\begin{table}[htbp]
  \centering
  \caption{Out-of-sample cost comparison with misspecified mode probabilities.}
    \begin{tabular}{lrrrrr}
    \hline
    Model & $\Delta=0$ & $\Delta=0.1$ & \multicolumn{1}{l}{$\Delta=0.2$} & \multicolumn{1}{l}{$\Delta=-0.1$} & \multicolumn{1}{l}{$\Delta=-0.2$} \\
    \hline
    MM-M-$\rm{D^3RO}$& \textbf{-6191} & \textbf{-7809} & \textbf{-9428} & \textbf{-4572} & \textbf{-2953} \\
    SM-M-$\rm{D^3RO}$ & -5856 & -7388 & -8922 & -4323 & -2787 \\
    MM-M-DRO & -5252 & -6565 & -7875 & -3943 & -2629 \\
    \hline
    MM-D-$\rm{D^3RO}$ & \textbf{-6191} & \textbf{-7809} & \textbf{-9428} & \textbf{-4572} & \textbf{-2953} \\
    SM-D-$\rm{D^3RO}$  & -5970 & -7590 & -9210 & -4350 & -2731 \\
    MM-D-DRO  & -5252 & -6565 & -7875 & -3943 & -2629 \\
    
    \hline
    \end{tabular}%
  \label{tab:misspecified-mode}%
\end{table}%

\subsubsection{Computational Time}
\label{sec:CompStudy-RunTime}

Finally, we compare the computational time of different models when we increase the in-sample scenarios $\sum_{l=1}^LK_l$ for distance-based settings and support size $K$ for moment-based settings in Figure \ref{fig:time}. From Figure \ref{fig:time}(a), MM-D-${\rm D^3RO}$ is the most computationally expensive. From Figure \ref{fig:time}(b), moment-based ambiguity sets are less time consuming than the distance-based ambiguity sets, as all three models can solve all the instances within 10 seconds.
\begin{figure}[ht!]
    \centering
    \begin{subfigure}{0.45\textwidth}
 \resizebox{\textwidth}{!}{%
\begin{tikzpicture}
  \begin{axis}
  [
    xlabel={In-sample Scenarios $\sum_{l=1}^LK_l$},
    ylabel={Average Computational Time (sec.)},
    xtick={100, 200, 300, 400, 500},
cycle list name=black white,
   legend pos= outer north east
]
    \addplot coordinates {
(100, 11.54727113)
(200, 43.32906849)
(300, 108.9730914)
(400, 201.1358791)
(500, 322.6592803)
    };\pgfplotsset{cycle list shift=5}
    \addplot coordinates {
(100, 8.023281813)
(200, 23.14223015)
(300, 23.24498968)
(400, 34.93838937)
(500, 48.21160278)
    };\pgfplotsset{cycle list shift=5}
            \addplot coordinates {
(100, 0.484147692)
(200, 1.132741976)
(300, 2.07832346)
(400, 3.05647676)
(500, 3.979480195)
    };
    \legend{MM-D-${\rm D^3RO}$, SM-D-${\rm D^3RO}$, MM-D-DRO}
  \end{axis}
\end{tikzpicture}%
}
 \caption{Average computational time w.r.t. total number of in-sample scenarios $\sum_{l=1}^LK_l$}
\end{subfigure}
\begin{subfigure}{0.45\textwidth}
\resizebox{\textwidth}{!}{%
 \pgfplotsset{scaled y ticks=true}
\begin{tikzpicture}
  \begin{axis}
  [
    xlabel={Support size $K$},
    ylabel={Average Computational Time (sec.)},
    xtick={100, 200, 300, 400, 500},
    cycle list name=black white,
     legend pos= outer north east
    ]
    \addplot coordinates {
(100, 0.873581386)
(200, 2.201040602)
(300, 3.739864302)
(400, 4.844594979)
(500, 6.832366323)
    };\pgfplotsset{cycle list shift=5}
    \addplot coordinates {
(100, 0.407206345)
(200, 0.777902198)
(300, 1.277888823)
(400, 1.795866919)
(500, 2.34473083)
    };
        \addplot coordinates {
(100, 0.568350506)
(200, 1.240767646)
(300, 1.929992008)
(400, 2.705841255)
(500, 3.413926697)
    };        
    \legend{MM-M-${\rm D^3RO}$, SM-M-${\rm D^3RO}$, MM-M-DRO}
  \end{axis}
\end{tikzpicture}%
}
\caption{Average computational time w.r.t. the support set size $K$}
\end{subfigure}
    \caption{Computational time comparison with different in-sample scenarios $\sum_{l=1}^LK_l$ and support size $K$.}
    \label{fig:time}
\end{figure}
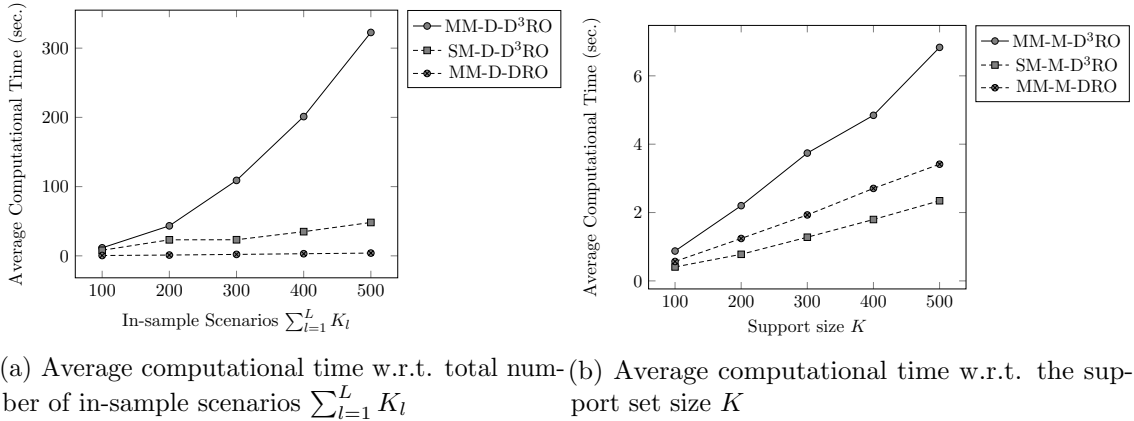

{\color{black}
We also compare the computational time of MM-M-${\rm D^3RO}$ between the discrete support problem ($K=200$) solved by Gurobi and the continuous support problem solved by our decomposition algorithm \ref{alg:solutionAlg}. The time limit for the discrete support problem is set to 3 hours. For the continuous support problem, we set the time limit for solving each master problem to 900 seconds and terminate the algorithm when the gap between LB and UB is within 1\%. We report the performance when we increase the problem size from $(I, J) = (5, 10)$ to $(I, J) = (30, 60)$ in Table \ref{tab:computational}, where the optimality gaps returned by Gurobi after the time limit are marked in the parenthesis.
\begin{table}[ht!]
  \centering
  \caption{Computational Time Comparison Between Discrete Support Problem Solved by Gurobi and Continuous Support Problem Solved by Decomposition Algorithm}
  \resizebox{\textwidth}{!}{
    \begin{tabular}{l|rrr|rrrrr}
    \hline
          & \multicolumn{3}{c|}{K = 200} & \multicolumn{5}{c}{Decomposition Algorithm for Contiuous Support} \\
          & \multicolumn{1}{r}{Time (sec.)} & IS Cost   & OOS Cost & Time (sec.)  & LB    & UB    & OOS cost & \# iterations \\
          \hline
    $(I, J) = (5, 10)$ & \multicolumn{1}{r}{1.55} & -2,079 & -3,257 & 0.32  & -2,079 & -2,079 & -3,257 & 12 \\
    $(I, J) = (10, 20)$ & \multicolumn{1}{r}{34.35} & -27,218 & -33,783 & 4.26  & -27,218 & -27,218 & -33,783 & 13 \\
    $(I, J) = (15, 30)$ & \multicolumn{1}{r}{1,613} & -1,079,154 & -99,177 & 104 & -1,079,154 & -1,079,154 & -99,177 & 13 \\
    $(I, J) = (20, 40)$ & 10,800  & -1,175,549 (18.85\%) & -203,517 & 2,711  & -1,175,549  & -1,175,549 & -203,517 & 27 \\
    $(I, J) = (25, 50)$ & 10,800 	&-1,050,037 (33.31\%) &	-129,440 & 3,615  &	-1,268,665  	&-1,268,665&	-295,309 & 19 \\
    $(I, J) = (30, 60)$  & 10,800 &	-1,192,186 (18.03\%)	&-342,615 & 4,517 &	-1,350,172 	&-1,349,350&	-541,263&	45 \\
    \hline
    \end{tabular}%
    }
  \label{tab:computational}%
\end{table}%
From Table \ref{tab:computational}, the decomposition algorithm can find optimal solutions by closing the gap between LB and UB when the problem size is within $(I,J)=(25,50)$, and it requires significantly less time compared to the discrete support problem solved by Gurobi. When the problem size increases, Gurobi fails to find an optimal solution for the discrete support problem within 3 hours and produces a large optimality gap. On the other hand, the decomposition algorithm can still reach a reasonable gap between LB and UB within one and a half hour. We also report the performance of the discrete support problem when we vary $K$ from 300 to 500 in Appendix \ref{sec:appendix-additional-results}.
}


{\color{black}
\subsection{Shipment Planning and Pricing Problem}
\label{sec:Computations-Shipment}
In this section, we follow \cite{bertsimas2020predictive} and consider a two-stage shipment planning and pricing problem, where we aim to ship products from $I$ warehouses to satisfy the demand at $J$ customer sites. Specifically, we focus on the following model:
\begin{align}\label{eq:shipment}
    \min_{y_1\in\mathbb{R}_+,\boldsymbol{y}_2 \in\mathbb{R}_+^{I}} P_1\sum_{i=1}^Iy_{2,i} + \max_{\mathbb{P} \in \Theta(y_1)}\mathbb{E}_{\boldsymbol{\xi}\sim\mathbb{P}} [h(\boldsymbol{y},\boldsymbol{\xi}(y_1))],
\end{align}
Here, $y_1\in\mathbb{R}_+$ is the price of this product, and $y_{2,i}\in\mathbb{R}_+$ is the production amount at warehouse $i$, for all $i=1,\ldots, I$. The unit production cost is denoted as $P_1\ge 0$. Then in the second stage, $h(\boldsymbol{y},\boldsymbol{\xi}(y_1))$ denotes the total last-minute production and transportation cost minus the profit:
\begin{subequations}\label{model:pricing}
\begin{align}
    h(\boldsymbol{y},\boldsymbol{\xi}(y_1)):=\min_{s_{ij},t_{i}\in \mathbb{R}_+} \quad & P_2\sum_{i=1}^It_{i}+\sum_{i=1}^I \sum_{j=1}^J c_{ij} s_{ij} - y_1\sum_{j=1}^J \xi_{j}(y_1)\\
    \textrm{s.t} \quad & \sum_{i=1}^I s_{ij} \ge \xi_{j}(y_1),\ \forall j=1,\ldots, J,\\
    \quad & \sum_{j=1}^J s_{ij} \leq y_{2,i}+t_{i},\ \forall i=1,\ldots, I,
\end{align}
\end{subequations}
 where demand $\xi_j(y_1)\in\mathbb{R}_+$ at each customer site $j$ depends on our first-stage pricing decision $y_1$. We ship $s_{ij}\ge 0$ units from warehouse $i$ to customer site $j$ at a unit transportation cost of $c_{ij}\ge 0$. We also have an option of last-minute production $t_i$ at warehouse $i$ with a higher unit production cost $P_2>P_1$.

Since Model \eqref{eq:shipment} is generally non-convex, we consider a relatively small problem size with $I=2$ and $J=1$. We set $P_1=0.01,\ P_2=0.02,\ c_{ij}=0.01,\ \forall i=1,\ldots,I,\ j=1,\ldots,J$. We assume that the random decision-dependent demand $\boldsymbol{\xi}(y_1)$ is multimodal and can have the following two modes. In the first mode, the price has a moderate impact on the demand, with a ground truth demand model $\xi_{1,j}(y_1)=\max\{0, -y_1+10\}+\epsilon_{1,j}$, where $\epsilon_{1,j}$ is the zero-mean noise term. In the second mode, the price has a higher impact on the demand: $\xi_{2,j}(y_1)=\max\{0, -2y_1+10\}+\epsilon_{2,j}$, where $\epsilon_{2,j}$ is the zero-mean noise term. We test this problem with moment-based ambiguity sets, where we set the empirical first moments to be $\mu_1(y_1)=\max\{0, -y_1+10\}$ and $\mu_2(y_1)=\max\{0, -2y_1+10\}$, respectively. The discrete support set is $\Xi_l=\{0, 1,\ldots,10\}$ for all $l=1,2$. Since the first-stage pricing decision $y_1$ is continuous, we adopt the linear scaling scheme in Section \ref{sec:linear-scaling} to describe our decision-dependent mode probability, i.e., $\hat{p}_1(y_1)=0.1y_1,\ \hat{p}_2(y_1)=1-0.1y_1$. We also add additional constraints to ensure that $\hat{p}_1(y_1)$ and $\hat{p}_2(y_1)$ are within 0 and 1. The radius of the Variation distance set is set to $\rho=0.2$ at default. 

For this shipment planning and pricing problem, the decision-independent counterpart will always drive the price $y_1$ to infinity, as the demand is assumed to be independent of the price in this case. Due to the unboundedness of the decision-independent counterpart, we only compare our MM-M-$\rm{D^3RO}$ with the single-modal benchmark (SM-M-$\rm{D^3RO}$). We first vary the radius $\epsilon^{\mu}$ from 0.1 to 0.5 and report the optimal price $y_1$ and the in-sample and out-of-sample cost in Table \ref{tab:sensitivity-moment-pricing}. From Table \ref{tab:sensitivity-moment-pricing}, as we increase the radius, the optimal prices $y_1$ in MM-M-$\rm{D^3RO}$ and SM-M-$\rm{D^3RO}$ both drop, although they differ from each other significantly. On the other hand, MM-M-$\rm{D^3RO}$ always obtains lower in-sample and out-of-sample costs, compared to the single-modal counterpart.

\begin{table}[ht!]
  \centering
  \caption{Solution Comparison of Moment-based Ambiguity Sets with Varying $\epsilon^{\mu}$}
    \begin{tabular}{l|rrr|rrr}
    \hline
    & \multicolumn{3}{c|}{MM-M-$\rm{D^3RO}$} & \multicolumn{3}{c}{SM-M-$\rm{D^3RO}$} \\
    $\epsilon^{\mu}$& Price $y_1$ & IS Cost& OOS Cost& Price $y_1$ & IS Cost& OOS Cost\\
    \hline
    0.1 & 2.79 & -12.53  & -14.35 & 7.56 & -6.39  & -13.90 \\
     0.2  & 2.76 & -12.26  & -14.30& 7.55 & -6.24 & -13.93 \\
     0.3  & 2.72 & -11.99  & -14.25 & 7.54 & -6.09 & -13.94   \\
   0.4  & 2.68 & -11.72  & -14.20 & 7.53 & -5.94 & -13.96  \\
    0.5  & 2.65 & -11.46  & -14.14 & 7.52 & -5.79 & -13.99  \\
    \hline
    \end{tabular}%
  \label{tab:sensitivity-moment-pricing}%
\end{table}%

Next, we vary $\rho$ from 0 to 0.8 and report the in-sample and out-of-sample costs in Figure \ref{fig:pricing-rho}. As we increase $\rho$, the gaps of the in-sample costs between MM-M-$\rm{D^3RO}$ and SM-M-$\rm{D^3RO}$ increase, which agrees with our numerical results in Section \ref{sec:CompStudy-EffectofMultimodality}. Moreover, MM-M-$\rm{D^3RO}$ almost always achieves a better out-of-sample cost, compared to the single-modal counterpart, and their gaps also increase as we increase $\rho$.
\begin{figure}[ht!]
    \centering
    \begin{subfigure}{0.45\textwidth}
 \resizebox{\textwidth}{!}{%
 \pgfplotsset{scaled y ticks=true}
\begin{tikzpicture}
  \begin{axis}
  [
    xlabel={Robustness Level $\rho$},
    ylabel={In-Sample Cost},
    xtick={0, 0.2, 0.4, 0.6, 0.8},
    scaled y ticks = true,
cycle list name=black white,
   legend pos= outer north east
]
    \addplot coordinates {
(0, -13.08)
(0.2, -12.2)
(0.4, -11.46)
(0.6, -11.14)
(0.8, -11.14)
    };\pgfplotsset{cycle list shift=4}
    \addplot coordinates {
(0, -13.08)
(0.2, -8.37)
(0.4, -5.79)
(0.6, -4.06)
(0.8, -3.63)
    };\pgfplotsset{cycle list shift=2}
    \legend{MM-M-${\rm D^3RO}$, SM-M-${\rm D^3RO}$}
  \end{axis}
\end{tikzpicture}%
}
 \caption{In-sample cost w.r.t. robustness level $\rho$}
\end{subfigure}
\begin{subfigure}{0.45\textwidth}
\resizebox{\textwidth}{!}{%
\pgfplotsset{scaled y ticks=true}
\begin{tikzpicture}
  \begin{axis}
  [
    xlabel={Robustness Level $\rho$},
    ylabel={Out-of-Sample Cost},
    xtick={0, 0.2, 0.4, 0.6, 0.8},
    scaled y ticks = true,
    cycle list name=black white,
     legend pos= outer north east
    ]
    \addplot coordinates {
(0, -14.62)
(0.2, -14.42)
(0.4, -14.14)
(0.6, -13.67)
(0.8, -13.67)
    };\pgfplotsset{cycle list shift=4}
    \addplot coordinates {
(0, -14.62)
(0.2, -14.72)
(0.4, -13.99)
(0.6, -11.68)
(0.8, -6.63)
    };\pgfplotsset{cycle list shift=2}
    \legend{MM-M-${\rm D^3RO}$, SM-M-${\rm D^3RO}$}
  \end{axis}
\end{tikzpicture}%
}
\caption{Out-of-sample cost w.r.t. robustness level $\rho$}
\end{subfigure}
    \caption{In-sample and out-of-sample cost comparison between multimodal $\rm{D^3RO}$ model and its single-modal counterpart with different robustness level $\rho$.}
    \label{fig:pricing-rho}
\end{figure}
}

\section{Conclusion}
\label{sec:Conclusion}

In this paper, we propose a generic DRO framework for two-stage stochastic programs with multimodal uncertainties, when the first-stage decisions impact both the mode probabilities and distribution corresponding to each mode. To formulate this problem setting, we introduce a novel ambiguity set characterizing the decision-dependent mode probabilities through a $\phi$-divergence based set, while providing both moment-based and Wasserstein distance-based settings for representing the distributions corresponding to each mode. We then present two special cases of the $\phi$-divergence based set by considering variation distance and $\chi^2$-distance. By leveraging these two cases, we first derive generic reformulations under moment-based and Wasserstein distance-based settings, and obtain additional results for the distance-based setting under objective and constraint uncertainty cases of the second-stage problem. 
We further present different functions to represent the dependency between the first-stage decisions and mode probability to provide tractable reformulations and settings that can be applicable to different applications. 
Moreover, we provide special cases to obtain MILP or MISOCP based tractable and exact reformulations over these reformulations that can be solved by the off-the-shelf solvers. 
To evaluate the value of the proposed multimodal decision-dependent DRO approach, we introduce its single-modal counterpart by providing alternative ambiguity sets and demonstrate its better performance analytically. 
{\color{black}Additionally, to solve large-scale instances, we provide a separation-based decomposition algorithm with finite convergence and optimality guarantees under certain settings.}
We present a detailed computational study on a facility location problem {\color{black}and a shipment planning problem with pricing} to illustrate our results by providing comparisons against single-modal and decision-independent approaches. We demonstrate that the proposed approach provides better in-sample and out-of-sample cost values for both moment-based and distance-based ambiguity sets under different robustness levels. These results are further validated under misspecified models to represent the distributional ambiguities for both mode probabilities and the distributions corresponding to these modes. 
{\color{black}We further illustrate the speed-ups obtained by the solution algorithm over various instances.}

Overall, our paper proposes a novel framework to address both multimodalities and decision-dependent uncertainties within a DRO problem while considering various forms of ambiguity sets, providing computationally tractable reformulations, and demonstrating its performance both analytically and computationally. 
As a future research direction, our proposed framework can be applied to various applications and extended by addressing 
potential non-convexities arising under specific forms of problem settings and ambiguity sets.

\bibliography{Xian_bib, referencesDDDR}
\bibliographystyle{apalike}

\appendix
{\color{black}
\section{Proofs for Section \ref{sec:ProblemFormulation}}
\label{sec:omittedProofs}

This section provides the omitted proofs of Theorems \ref{thm:DRO-PhiDivergence}-\ref{thm:DRO-ChiDistance}.}  

\begin{proof}[Proof of Theorem \ref{thm:DRO-PhiDivergence}]
    Denoting $\psi_l=\max_{\mathbb{P}_l \in \mathcal{U}_l(\boldsymbol{y})}\mathbb{E}_{\mathbb{P}_l}[h(\boldsymbol{y},\boldsymbol{\xi})]$, the inner maximization problem in Model \eqref{model:DRO} becomes
    \begin{subequations}\label{model:DRO-InnerMax}
    \begin{align}
        \max_{\boldsymbol{p}}\quad& \sum_{l=1}^Lp_l\psi_l\\
        \text{s.t.}\quad&\sum_{l=1}^Lp_l=1\\
        &\sum_{l=1}^L \hat{p}_l(y) \> \phi\left(\frac{p_l}{\hat{p}_l(y)}\right) \le \rho\\
        &p_l\ge 0,\ \forall l=1,\ldots,L
    \end{align}
    \end{subequations}
    This is a convex optimization problem due to the convexity of $\phi$ function.
    Assigning Lagrangian multipliers $\eta$ and $\lambda\ge 0$ to the above constraints, the Lagrange function is given by
    \begin{align*}
        L(\boldsymbol{p},\lambda,\eta)=\sum_{l=1}^Lp_l\psi_l+\eta(1-\sum_{l=1}^Lp_l)+\lambda(\rho-\sum_{l=1}^L \hat{p}_l(y) \> \phi\left(\frac{p_l}{\hat{p}_l(y)}\right))
    \end{align*}
    and the dual objective function is 
    \begin{align*}
    g(\lambda,\eta)&=\max_{\boldsymbol{p}\ge 0}L(\boldsymbol{p},\lambda,\eta)\\
    &=\eta+\rho\lambda+\max_{\boldsymbol{p}\ge 0}\left\{\sum_{l=1}^Lp_l\psi_l-\eta\sum_{l=1}^Lp_l-\lambda\sum_{l=1}^L \hat{p}_l(y) \> \phi\left(\frac{p_l}{\hat{p}_l(y)}\right)\right\}\\
    &=\eta+\rho\lambda+\sum_{l=1}^L\hat{p}_l(y)\max_{\frac{p_l}{\hat{p}_l(y)}\ge 0}\left\{\frac{p_l}{\hat{p}_l(y)}(\psi_l-\eta)-\lambda\> \phi\left(\frac{p_l}{\hat{p}_l(y)}\right)\right\}\\
    &=\eta+\rho\lambda+\sum_{l=1}^L\hat{p}_l(y)\max_{t\ge0} t(\psi_l-\eta)-\lambda\phi(t)\\
    &=\eta+\rho\lambda+\sum_{l=1}^L\hat{p}_l(y)(\lambda\phi)^*(\psi_l-\eta)\\
    &=\eta+\rho\lambda+\lambda\sum_{l=1}^L\hat{p}_l(y)\phi^*(\frac{\psi_l-\eta}{\lambda})
    \end{align*}
    
    Since $\sum_{l=1}^L\hat{p}_l(\boldsymbol{y})=1,\ \sum_{l=1}^L \hat{p}_l(y) \> \phi\left(\frac{\hat{p}_l(y)}{\hat{p}_l(y)}\right) =0< \rho$, Slater's condition holds and we can apply strong duality to recast the inner maximization problem \eqref{model:DRO-InnerMax} as $\min_{\lambda\ge 0}g(\lambda,\eta)$. This leads to the desired model \eqref{model:DRO-PhiDivergence}.
\end{proof}

\begin{proof}[Proof of Theorem \ref{thm:DRO-VariationDistance}]
    When $\phi(t)=|t-1|$, we have
    \begin{align*}
        \phi^*(s)=\begin{cases}
            -1,\ s\le -1,\\
            s,\ -1\le s\le 1,\\
            +\infty,\ s> 1.
        \end{cases}
    \end{align*}
    Denoting $r_l=\lambda \phi^*(\frac{\psi_l-\eta}{\lambda})$, we have
    \begin{align*}
        r_l&=\begin{cases}
            -\lambda, \ \psi_l-\eta\le -\lambda\\
            \psi_l-\eta,\ -\lambda\le\psi_l-\eta\le \lambda\\
            +\infty,\ \psi_l-\eta> \lambda.
        \end{cases}\\
        &=\begin{cases}
            \max\{-\lambda, \psi_l-\eta\}, \ \psi_l-\eta\le \lambda\\
            +\infty,\ \psi_l-\eta> \lambda.
        \end{cases}
    \end{align*}
    Since $\hat{\boldsymbol{p}}(\boldsymbol{y})\ge 0$, we can equivalently add constraints $r_l\ge \psi_l-\eta,\ r_l\ge -\lambda,\ \psi_l-\eta\le \lambda$. This completes the proof.
\end{proof}

\begin{proof}[Proof of Theorem \ref{thm:DRO-ChiDistance}]
    When $\phi(t)=\frac{1}{t}(t-1)^2$, we have
    \begin{align*}
        \phi^*(s)=\begin{cases}
            2-2\sqrt{1-s},\ s\le 1,\\
            +\infty,\ s> 1.
        \end{cases}
    \end{align*}
    Denoting $r_l=\lambda \sqrt{1-\frac{\psi_l-\eta}{\lambda}}$ where $\psi_l-\eta\le \lambda$, we have
    \begin{align*}
        \sqrt{r_l^2+\frac{1}{4}(\psi_l-\eta)^2}=\sqrt{\lambda^2-\lambda(\psi_l-\eta)+\frac{1}{4}(\psi_l-\eta)^2}=|\lambda-\frac{1}{2}(\psi_l-\eta)|=\lambda-\frac{1}{2}(\psi_l-\eta)
    \end{align*}
   This completes the proof.
\end{proof}

\section{Reformulation Results for Wasserstein Ambiguity Sets with Constraint Uncertainty in the Second-stage Problem}\label{append:constraint}
In this appendix, we present reformulation results for Wasserstein ambiguity set with constraint uncertainty in the second-stage problem to complement our results in Section \ref{sec:distance}. To this end, we first make the following assumption for this setting.
\begin{assumption}[Sufficiently Expensive Recourse] \label{assumption:sufficiently}
A two-stage distributionally robust linear program \eqref{model:DRO} has sufficiently expensive recourse if for any $\boldsymbol{\xi}\in\Xi$, the dual program of the second-stage LP \eqref{eq:second-stage} is feasible.
\end{assumption}
We present a reformulation under variation distance and Wasserstein ambiguity set with constraint uncertainty in the next theorem by leveraging the results obtained in \cite{hanasusanto2018conic}.
\begin{theorem}[Variation Distance + Wasserstein-based + Constraint Uncertainty]
    Suppose $\boldsymbol{Q}=0$, $\Xi_l=\mathbb{R}$ and Assumption \ref{assumption:sufficiently} holds. Then the multimodal ${\rm D^3RO}$ model \eqref{model:DRO} with variation distance set $\Delta(\hat{p}(\boldsymbol{y}))$ defined in \eqref{eq:modeDist-VariationDistance} and Wasserstein ambiguity set $\mathcal{U}_l(\boldsymbol{y})$ defined in \eqref{eq:distance-based} can be tractable for $q=1$ and admits the following equivalent formulation:
    \begin{align*}
    \min_{\boldsymbol{y},\lambda,\eta,\boldsymbol{\psi}} \quad &\boldsymbol{c}^{\mathsf T}\boldsymbol{y}+\eta+\rho\lambda+\sum_{l=1}^L\hat{p}_l(\boldsymbol{y})r_l\\
        \text{s.t.}\quad & \boldsymbol{y}\in\mathcal{Y},\ \lambda\ge 0\\
        & r_l\ge \epsilon_l\gamma_l+\frac{1}{K_l}\sum_{k=1}^{K_l}\boldsymbol{q}^{\mathsf T}\boldsymbol{x}_{lk}-\eta, \ \forall l=1,\ldots, L\\
        & r_l\ge -\lambda,\ \forall l=1,\ldots, L\\
        & \epsilon_l\gamma_l+\frac{1}{K_l}\sum_{k=1}^{K_l}\boldsymbol{q}^{\mathsf T}\boldsymbol{x}_{lk}-\eta\le \lambda,\ \forall l=1,\ldots, L\\
       &\boldsymbol{T}(\boldsymbol{y})\hat{\boldsymbol{\xi}}^T_{lk}(\boldsymbol{y})+W\boldsymbol{x}_{lk}\ge \boldsymbol{R}(\boldsymbol{y}),\ \forall k=1,\ldots,K_l,\ l=1,\ldots,L\\
        & \boldsymbol{q}^{\mathsf T}\boldsymbol{\phi}_{ln}\le \gamma_l,\ \forall n=1,\ldots,N,\ l=1,\ldots, L\\
        & \boldsymbol{q}^{\mathsf T}\boldsymbol{\psi}_{ln}\le \gamma_l,\ \forall n=1,\ldots,N,\ l=1,\ldots, L\\
        & \boldsymbol{T}(\boldsymbol{y})\boldsymbol{e}_n\le \boldsymbol{W}\boldsymbol{\phi}_{ln},\ \forall n=1,\ldots,N,\ l=1,\ldots, L\\
        & -\boldsymbol{T}(\boldsymbol{y})\boldsymbol{e}_n\le \boldsymbol{W}\boldsymbol{\psi}_{ln},\ \forall n=1,\ldots,N,\ l=1,\ldots, L\\
        & \boldsymbol{y}\in\mathcal{Y},\ \boldsymbol{x}_{lk}\in\mathbb{R}^J,\  \tau\in\mathbb{R},\ \boldsymbol{\mu}\in\mathbb{R}_+^L,\ \boldsymbol{\kappa}\in\mathbb{R}_-^L,\ \forall k=1,\ldots,K_l,\ l=1,\ldots, L\\ &\boldsymbol{\phi}_{ln},\boldsymbol{\psi}_{ln}\in \mathbb{R}^{J}, \ \forall n=1,\ldots,N,\ l=1,\ldots, L.
    \end{align*}
\end{theorem}
\begin{proof}
    Combining Theorem \ref{thm:DRO-VariationDistance} with Theorem 6 in \cite{hanasusanto2018conic} yields the desired result.
\end{proof}
Next, we present the result for $\chi^2$-distance and Wasserstein ambiguity set with constraint uncertainty.
\begin{theorem}[$\chi^2$-Distance + Wasserstein-based + Constraint Uncertainty]
    Suppose $\boldsymbol{Q}=0$,  $\Xi_l=\mathbb{R}$ and Assumption \ref{assumption:sufficiently} holds. Then the multimodal ${\rm D^3RO}$ model \eqref{model:DRO} with $\chi^2$-distance set $\Delta(\hat{p}(\boldsymbol{y}))$ defined in \eqref{eq:modeDist-ChiSquare} and Wasserstein ambiguity set $\mathcal{U}_l(\boldsymbol{y})$ defined in \eqref{eq:distance-based} can be tractable for $q=1$ and admits the following equivalent formulation:
    \begin{align*}
    \min_{\boldsymbol{y},\lambda,\eta,\boldsymbol{\psi}} \quad &\boldsymbol{c}^{\mathsf T}\boldsymbol{y}+\eta+\rho\lambda+2\lambda-2\sum_{l=1}^L\hat{p}_l(\boldsymbol{y})r_l\\
        \text{s.t.}\quad & \boldsymbol{y}\in\mathcal{Y},\ \lambda\ge 0\\
        & \sqrt{r_l^2+\frac{1}{4}(\psi_l-\eta)^2}\le \lambda-\frac{1}{2}(\psi_l-\eta), \ \forall l=1,\ldots, L\\
        & \psi_l-\eta\le \lambda, \ \forall l=1,\ldots, L\\
        & \epsilon_l\gamma_l+\frac{1}{K_l}\sum_{k=1}^{K_l}\boldsymbol{q}^{\mathsf T}\boldsymbol{x}_{lk}\le \psi_l,\ \forall l=1,\ldots, L\\
       &\boldsymbol{T}(\boldsymbol{y})\hat{\boldsymbol{\xi}}^T_{lk}(\boldsymbol{y})+W\boldsymbol{x}_{lk}\ge \boldsymbol{R}(\boldsymbol{y}),\ \forall k=1,\ldots,K_l,\ l=1,\ldots,L\\
        & \boldsymbol{q}^{\mathsf T}\boldsymbol{\phi}_{ln}\le \gamma_l,\ \forall n=1,\ldots,N,\ l=1,\ldots, L\\
        & \boldsymbol{q}^{\mathsf T}\boldsymbol{\psi}_{ln}\le \gamma_l,\ \forall n=1,\ldots,N,\ l=1,\ldots, L\\
        & \boldsymbol{T}(\boldsymbol{y})\boldsymbol{e}_n\le \boldsymbol{W}\boldsymbol{\phi}_{ln},\ \forall n=1,\ldots,N,\ l=1,\ldots, L\\
        & -\boldsymbol{T}(\boldsymbol{y})\boldsymbol{e}_n\le \boldsymbol{W}\boldsymbol{\psi}_{ln},\ \forall n=1,\ldots,N,\ l=1,\ldots, L\\
       &\boldsymbol{\phi}_{ln},\boldsymbol{\psi}_{ln}\in \mathbb{R}^{J}, \ \forall n=1,\ldots,N,\ l=1,\ldots, L.
    \end{align*}
\end{theorem}
\begin{proof}
    Combining Theorem \ref{thm:DRO-ChiDistance} with Theorem 6 in \cite{hanasusanto2018conic} yields the desired result.
\end{proof}

\section{Tractable Formulations under Special Cases}
\label{sec:AppendixSpecialCaseReformulations}

In this section, we provide monolithic mixed-integer linear programming reformulations of the two-stage multimodal $\rm{D^3RO}$ model \eqref{model:DRO} under variation distance based multimodal ambiguity with moment-based and distance-based ambiguity sets for each mode distribution, when the first-stage decisions $y$ are binary and an affine function in terms of the first-stage decisions is considered to represent $\hat{{p}}_l(\boldsymbol{y})$ for each mode $l=1,\cdots,L$ as proposed in Section \ref{sec:ModeAffineDependence} while considering the special cases proposed in Sections \ref{sec:moment} and \ref{sec:distance}.


We first derive the reformulation under the moment-based ambiguity setting corresponding to the distribution of each mode.  

\begin{theorem}
If for any feasible $\boldsymbol{y}\in\mathcal{Y} \subseteq \{0,1\}^I$, the ambiguity set $\mathcal{U}_l(\boldsymbol{y})$ defined in \eqref{eq:moment-based} is always non-empty, then the multimodal $\rm{D^3RO}$ model \eqref{model:DRO} with variation distance set $\Delta(\hat{p}(\boldsymbol{y}))$ defined in \eqref{eq:modeDist-VariationDistance} and $\hat{{p}}_l(\boldsymbol{y})$ defined in \eqref{eq:ModeAffineDependence} and moment-based ambiguity set $\mathcal{U}_l(\boldsymbol{y})$ defined in \eqref{eq:moment-based} with decision-dependent moment functions defined in Section \ref{sec:SpecialCases-MomentBased} is equivalent to the following mixed-integer linear program:
\begin{subequations}\label{model:variation+moment-special-MILP}
    \begin{align}   \min_{\boldsymbol{y},\lambda,\eta,\underline{\boldsymbol{\beta}}_l,\bar{\boldsymbol{\beta}}_l} \quad &\boldsymbol{c}^{\mathsf T}\boldsymbol{y}+\eta+\rho\lambda+\sum_{l=1}^L (\bar{p}_l r_l + \sum_{i=1}^I \lambda_{l,i}^p \varphi_{l,i})\\
        \text{s.t.}\quad & \boldsymbol{y}\in\mathcal{Y},\ \lambda,\ \underline{\boldsymbol{\beta}}_l,\ \bar{\boldsymbol{\beta}}_l\ge 0,\ \forall l=1,\ldots, L\\
        & \alpha_l+\sum_{n=1}^N\bar{\beta}_{l,n}(\bar\mu_{l,n}+\epsilon_{l,n}^{\mu})+\sum_{n=1}^N\sum_{i=1}^I\lambda_{l,n,i}^{\mu}\bar\mu_{l,n} \bar{z}_{l,n,i} +\sum_{n=1}^N\bar{\beta}_{l,N+n}(\bar\mu_{l,n}^2+\bar\sigma_{l,n}^2)\bar{\epsilon}^S_{l,n}\nonumber\\
        &+\sum_{n=1}^N\sum_{i=1}^I\lambda^S_{l,n,i}\bar{\epsilon}^S_{l,n}(\bar\mu_{l,n}^2+\bar\sigma_{l,n}^2)\bar{z}_{l,N+n,i}-\sum_{n=1}^N\underline{\beta}_{l,1+n}(\bar\mu_{l,n}-\epsilon_{l,n}^{\mu})-\sum_{n=1}^N\sum_{i=1}^I\lambda_{l,n,i}^{\mu}\bar\mu_{l,n}\underline{z}_{l,n,i}\nonumber\\
\quad& -\sum_{n=1}^N\underline{\beta}_{l,N+n}(\bar\mu_{l,n}^2+\bar\sigma_{l,n}^2)\underline{\epsilon}^S_{l,n}-\sum_{n=1}^N\sum_{i=1}^I\lambda^S_{l,n,i}\underline{\epsilon}^S_{l,n}(\bar\mu_{l,n}^2+\bar\sigma_{l,n}^2)\underline{z}_{l,N+n,i}-\eta\le \min\{r_l,\lambda\},\nonumber\\
&\hspace{29em}\forall l=1,\ldots,L\\
         & r_l\ge -\lambda,\ \forall l=1,\ldots, L\\
       & \alpha_l+\sum_{n\in [N]}\xi_{n}^k(\bar{\beta}_{l,n}-\underline{\beta}_{l,n})+\sum_{n\in[N]}(\xi_{n}^k)^2(\bar{\beta}_{l,N+n}-\underline{\beta}_{l,N+n})\ge  (\boldsymbol{Q}\boldsymbol{\xi}^k+\boldsymbol{q})^{\mathsf T}\boldsymbol{x}_{k},\nonumber\\
       &\hspace{23em}\forall l=1,\ldots, L,\ k=1,\ldots,K, \\
       & \boldsymbol{T}(\boldsymbol{y})\boldsymbol{\xi}^k+W\boldsymbol{x}_k\ge \boldsymbol{R}(\boldsymbol{y}),\ \forall k=1,\ldots,K, \\
       & (\varphi_{l,i}, r_l, y_{i})\in {\color{black}\mathcal{MC}}_{(l^{r}_l,u^{r}_l)}, \quad \forall l=1,\ldots, L, i = 1,\cdots, I, \\
       & (\bar{z}_{l,n,i}, \bar{\beta}_{l,n}, y_{i})\in {\color{black}\mathcal{MC}}_{(l^{\bar{\beta}}_{l,n},u^{\bar{\beta}}_{l,n})}, \quad (\bar{z}_{l,N+n,i}, \bar{\beta}_{l,N+n}, y_{i})\in {\color{black}\mathcal{MC}}_{(l^{\bar{\beta}}_{l,N+n},u^{\bar{\beta}}_{l,N+n})} ,\nonumber\\
       &\hspace{19em} \forall l=1,\ldots, L, i = 1,\cdots, I, n\in[N],\\
          & (\underline{z}_{l,n,i}, \underline{\beta}_{l,n}, y_{i})\in {\color{black}\mathcal{MC}}_{(l^{\underline{\beta}}_{l,n},u^{\underline{\beta}}_{l,n})}, \quad (\underline{z}_{l,N+n,i}, \bar{\beta}_{l,N+n}, y_{i})\in {\color{black}\mathcal{MC}}_{(l^{\underline{\beta}}_{l,N+n},u^{\underline{\beta}}_{l,N+n})} ,\nonumber\\
       &\hspace{19em} \forall l=1,\ldots, L, i = 1,\cdots, I, n\in[N],
    \end{align}
    \end{subequations}
where $r_l \in [l^{r}_l,u^{r}_l]$, $\bar{\beta}_{l,n}\in [l^{\bar{\beta}}_{l,n},u^{\bar{\beta}}_{l,n}]$, $\bar{\beta}_{l,N+n}\in [l^{\bar{\beta}}_{l,N+n},u^{\bar{\beta}}_{l,N+n}]$, $\underline{\beta}_{l,n}\in [l^{\underline{\beta}}_{l,n},u^{\underline{\beta}}_{l,n}]$, $\underline{\beta}_{l,N+n}\in [l^{\underline{\beta}}_{l,N+n},u^{\underline{\beta}}_{l,N+n}]$. 
\end{theorem}

\begin{proof}
Combining Model \eqref{model:variation+moment-special} with $\hat{{p}}_l(\boldsymbol{y})$ defined in \eqref{eq:ModeAffineDependence} and McCormick envelopes to linearize the resulting bilinear terms and plugging in the second-stage function \eqref{eq:second-stage} for representing $h(\boldsymbol{y},\boldsymbol{\xi}^k)$, we obtain the desired result. 
\end{proof}


Additionally, we provide the reformulation under the distance-based ambiguity setting corresponding to the distribution of each mode under objective uncertainty of the second-stage problem.  

\begin{theorem}
    Suppose $\boldsymbol{T}(\boldsymbol{y})=0$, $\Xi_l=\{\boldsymbol{\xi}:\ \boldsymbol{C}_l\boldsymbol{\xi}\le \boldsymbol{d}_l\}$, and for any given $\boldsymbol{y}\in\mathcal{Y} \subseteq \{0,1\}^I$, the feasible region $\{\boldsymbol{x}:W\boldsymbol{x}\ge \boldsymbol{R}(\boldsymbol{y})\}$ is non-empty and compact. The two-stage multimodal $\rm{D^3RO}$ model \eqref{model:DRO} with variation distance set $\Delta(\hat{p}(\boldsymbol{y}))$ defined in \eqref{eq:modeDist-VariationDistance} and $\hat{{p}}_l(\boldsymbol{y})$ defined in \eqref{eq:ModeAffineDependence} and Wasserstein ambiguity set $\mathcal{U}_l(\boldsymbol{y})$ defined in \eqref{eq:distance-based} with decision-dependent uncertainty realizations defined in Section \ref{sec:SpecialCases-DistanceBased} 
    can be tractable for $q=1$ admits the following equivalent formulation:
    \begin{subequations}\label{model:Variation+Wasserstein-Objective-MILP}
    \begin{align}
    \min\quad &\boldsymbol{c}^{\mathsf T}\boldsymbol{y}+\eta+\rho\lambda+\sum_{l=1}^L (\bar{p}_l r_l + \sum_{i=1}^I \lambda_{l,i}^p \varphi_{l,i})\\
        \text{s.t.}\quad
         & \eqref{eq:Variation-Wasserstein-domain} - \eqref{eq:Variation-Wasserstein-domain2}, \eqref{eq:Variation-Wasserstein-Constr}, \eqref{eq:Variation-Wasserstein-dualGamma} \nonumber \\
         & \boldsymbol{q}^{\mathsf T}\boldsymbol{x}_{lk}+\sum_{j=1}^Jx_{lkj}\sum_{n=1}^NQ_{jn}\bar{\xi}_{l,k,n}+\sum_{j=1}^J\sum_{n=1}^NQ_{jn}\sum_{i=1}^I\lambda^{\xi}_{l,k,n,i}\upsilon^{x}_{lkji}\\
&+\boldsymbol{d}_l^{\mathsf T}\boldsymbol{\mu}_{lk}-\sum_{h=1}^H\mu_{lkh}\sum_{n=1}^NC_{lhn}\bar{\xi}_{l,k,n}-\sum_{h=1}^H\sum_{n=1}^NC_{lhn}\sum_{i=1}^I\lambda^{\xi}_{l,k,n,i}\upsilon^{\mu}_{lkhi}\le w_{lk},\ \forall k=1,\ldots,K_l,\ l=1,\ldots, L, \\
& (\varphi_{l,i}, r_l, y_{i})\in {\color{black}\mathcal{MC}}_{(l^{r}_l,u^{r}_l)} \quad \forall l=1,\ldots, L, i = 1,\cdots, I, \\
& (\upsilon^{x}_{lkji}, x_{lkj},y_i) \in {\color{black}\mathcal{MC}}_{(l^{x}_{lkj},u^{x}_{lkj})} \quad \forall l=1,\ldots, L, k=1,\ldots,K_l, i = 1,\cdots, I, j = 1,\cdots, J,\\
& (\upsilon^{\mu}_{lkhi},\mu_{lkh},y_i) \in {\color{black}\mathcal{MC}}_{(l^{\mu}_{lkh},u^{\mu}_{lkh})} \quad \forall l=1,\ldots, L, k=1,\ldots,K_l, i = 1,\cdots, I, h = 1, \cdots, H,
    \end{align}
    \end{subequations}
where $r_l \in [l^{r}_l,u^{r}_l]$, $x_{lkj} \in [l^{x}_{lkj},u^{x}_{lkj}]$, $\mu_{lkh} \in [l^{\mu}_{lkh},u^{\mu}_{lkh}]$.
\end{theorem}

\begin{proof}
Combining the reformulation in Theorem \ref{thm:VariationWasserstein_Obj} with decision-dependent uncertainty realizations defined in Section \ref{sec:SpecialCases-DistanceBased} and $\hat{{p}}_l(\boldsymbol{y})$ defined in \eqref{eq:ModeAffineDependence} with McCormick envelopes to linearize the resulting bilinear terms, we obtain the desired result. 
\end{proof}

{\color{black}
\section{Additional Numerical Results}\label{sec:appendix-additional-results}
In this section, we report the performance of the MM-M-$\rm{D^3RO}$ model with discrete support, where we vary $K$ from 300 to 500. We set the time limit for Gurobi to 3 hours and mark the instances where the MIP gap in Gurobi exceeds 100\% within the time limit by ``-''.
\begin{table}[ht!]
  \centering
  \caption{Performance of the MM-M-$\rm{D^3RO}$ with Discrete Support $K=300$ Solved by Gurobi}
    \begin{tabular}{lrrr}
    \hline
          & \multicolumn{3}{c}{K = 300}  \\
          & Time  & IS Cost   & OOS Cost\\
          \hline
    $(I,J) = (5,10)$ & 2.53  & -2,079 & -3,257  \\
    $(I,J) = (10,20)$ & 57.06 & -27,218 & -33,783  \\
    $(I,J) = (15,30)$ & 2,781 & -1,079,154 & -99,177  \\
    $(I,J) = (20,40)$ & 10,800 & -1,175,549 (18.9\%) & -203,517  \\
    $(I,J) = (25,50)$ & 10,800 & -1,160,420 (20.55\%) & -208,752 \\
    $(I,J) = (30,60)$ & - & -     & -     \\
    \hline
    \end{tabular}%
  \label{tab:addlabel}%
\end{table}%

\begin{table}[H]
  \centering
  \caption{Performance of the MM-M-$\rm{D^3RO}$ with Discrete Support $K=400$ Solved by Gurobi}
    \begin{tabular}{lrrr}
    \hline
           & \multicolumn{3}{c}{K = 400}  \\
          & Time  & IS Cost   & OOS Cost \\
          \hline
    $(I,J) = (5,10)$ &  3.52  & -2,079 & -3257  \\
    $(I,J) = (10,20)$ & 87 & -27,218 & -33,783 \\
    $(I,J) = (15,30)$ &  6,845 & -1,079,154 & -99,177 \\
    $(I,J) = (20,40)$ & 10,800 & -1,138,923 (22.87\%) & -171,909  \\
    $(I,J) = (25,50)$ & 10,800 & -1,160,420 (20.55\%) & -208,752 \\
    $(I,J) = (30,60)$ & - & -     & -    \\
    \hline
    \end{tabular}%
  \label{tab:addlabel}%
\end{table}%

\begin{table}[H]
  \centering
  \caption{Performance of the MM-M-$\rm{D^3RO}$ with Discrete Support $K=500$ Solved by Gurobi}
    \begin{tabular}{lrrr}
    \hline
        & \multicolumn{3}{c}{K = 500} \\
          &  Time  & IS Cost   & OOS Cost \\
          \hline
    $(I,J) = (5,10)$ &4.62  & -2,079 & -3,257 \\
    $(I,J) = (10,20)$ &  130 & -27.218 & -33.783 \\
    $(I,J) = (15,30)$ &  6,064 & -1,079,154 & -99,177 \\
    $(I,J) = (20,40)$ & 10,800 & -1,175,549 (18.98\%) & -203,517 \\
    $(I,J) = (25,50)$ &  -     & -     & - \\
    $(I,J) = (30,60)$ & {-} & -     & - \\
    \hline
    \end{tabular}%
  \label{tab:addlabel}%
\end{table}%
}
\end{document}